\documentclass{article}
\usepackage[utf8]{inputenc}
\usepackage{amsmath, amscd}
\usepackage{amssymb, bm} 
\usepackage{amsthm}
\usepackage{mathabx}
\usepackage{color}
\usepackage[dvipsnames]{xcolor}
\usepackage{cancel}
\usepackage{ulem}
\usepackage{geometry}
\usepackage[all]{xy}
\usepackage{mathrsfs}
\usepackage{mathtools}
\usepackage{tikz}
\usetikzlibrary{matrix,arrows, calc}
\usepackage{multicol}
\usepackage{soul}
\usepackage{graphicx, psfrag, tikz, tikz-cd}
\usepackage{pgfplots}
\pgfplotsset{compat=1.16}

\usepackage{xr}
\makeatletter

\usepackage{caption}
\usepackage{subcaption}

\usepackage{paralist}

\numberwithin{equation}{section}

\newtheorem{theorem}{Theorem}[section]
\newtheorem{lemma}[theorem]{Lemma} 
\newtheorem{proposition}[theorem]{Proposition}
\newtheorem{corollary}[theorem]{Corollary}
\theoremstyle{definition} 
\newtheorem{definition}[theorem]{Definition}

\newtheorem{remark}[theorem]{Remark}
\newtheorem{example}[theorem]{Example}

\DeclareMathOperator{\Pic}{Pic}

\DeclareMathOperator{\Imag}{Im}	
\DeclareMathOperator{\Sp}{Sp}	
\DeclareMathOperator{\GL}{GL}	
\DeclareMathOperator{\Ext}{Ext}
\DeclareMathOperator{\Hom}{Hom}

\DeclareMathOperator{\ind}{ind}

\DeclareMathOperator{\Coh}{Coh}
\DeclareMathOperator{\Fuk}{Fuk}

\newcommand{\bC}{\mathbb C}

\newcommand{\bR}{\mathbb R}
\newcommand{\bT}{\mathbb T}
\newcommand{\bZ}{\mathbb Z}
\newcommand{\bbL}{\mathbb L}

\newcommand{\bi}{\mathsf{i}}
\newcommand{\bk}{\mathsf{k}}

\newcommand{\cA}{\mathcal A}
\newcommand{\cB}{\mathcal B}

\newcommand{\cE}{\mathcal E}
\newcommand{\cF}{\mathcal F}

\newcommand{\cH}{\mathcal H}
\newcommand{\cI}{\mathcal I}

\newcommand{\cL}{\mathcal L}
\newcommand{\cM}{\mathcal M}

\newcommand{\cO}{\mathcal{O}}
\newcommand{\cP}{\mathcal P}

\newcommand{\cS}{\mathcal S}
\newcommand{\cT}{\mathcal T}

\newcommand{\cY}{\mathcal Y}

\newcommand{\ft}{\mathfrak{t}}

\newcommand{\sk}{\mathsf{k}}

\newcommand{\szero}{\mathsf{0}}

\newcommand{\Proj}{\mathbf{Proj}}
\newcommand{\Diff}{\mathrm{Diff}} 
\newcommand{\Symp}{\mathrm{Symp}} 
\newcommand{\Ham}{\mathrm{Ham}}
 
\newcommand{\Hol}{\mathrm{Hol}}
\newcommand{\trop}{ {\mathrm{trop}} }
\newcommand{\tropp}{ {\mathrm{trop,p}} } 
\newcommand{\SYZ}{ {\mathrm{SYZ}} }

\newcommand{\NS}{\mathrm{NS}}

\newcommand{\aff}{\mathrm{aff}}

\newcommand{\Mir}{\mathrm{Mir}}

\newcommand{\tC}{\widetilde{C}}

\newcommand{\tL}{\widetilde{L}}

\newcommand{\tm}{\widetilde{m}}
\newcommand{\tn}{\widetilde{n}}

\newcommand{\la}{\lambda} 
\newcommand{\dd}{\partial}

\newcommand{\lra}{\longrightarrow} 

\definecolor{Cgreen}{RGB}{77,175,74}
\definecolor{Cblue}{RGB}{0, 150, 255}
\definecolor{Corange}{RGB}{255,127,0}

\usepackage[colorlinks=true, linkcolor=blue, filecolor=blue,citecolor=blue]{hyperref}

\title{Global SYZ mirror symmetry and homological mirror symmetry for principally polarized abelian varieties}
\author{Haniya Azam, Catherine Cannizzo, Heather Lee, Chiu-Chu Melissa Liu}
\date{}
\begin{document}

\maketitle

\noindent

\begin{abstract}
For any positive integer $g$, we introduce the moduli space $\cA^F_g =[\cH_g/P_g(\bZ)]$ 
parametrizing $g$-dimensional principally polarized abelian varieties $V_\tau$ together with a Strominger-Yau-Zalsow (SYZ) fibration, where $\tau \in \cH_g$ is the genus-$g$ Seigel upper half space and $P_g(\bZ) \subset \Sp(2g,\bZ)$ is the integral Siegel parabolic subgroup. We study global SYZ mirror symmetry over the global moduli $\cH_g$ and $\cA^F_g$, relating the B-model on $V_\tau$ and the A-model on its mirror, a compact $2g$-dimensional torus $\bT^{2g}$ equipped with a complexified symplectic form.

For each $V_\tau$, we establish a homological mirror symmetry (HMS) result at the cohomological level over $\bC$. This implies  core HMS at the cohomological level over $\bC$ and a graded $\bC$-algebra isomorphism known as Seidel's mirror map. We  study global HMS where Floer cohomology groups $HF^*(\hat{\ell}, \hat{\ell}')$ form coherent sheaves over a complex manifold parametrizing triples $(\tau, \hat{\ell}, \hat{\ell}')$ where $\tau \in \cH_g$ defines
a complexified symplectic form $\omega_\tau$ on $\bT^{2g}$ and
$\hat{\ell}$, $\hat{\ell}'$ are affine Lagrangian branes in $(\bT^{2g}, \omega_\tau)$.

\color{black}
   
\end{abstract}
\tableofcontents

\section{Introduction}

\subsection{Background and motivation} \label{sec:background}

Mirror symmetry relates the A-model on a Calabi-Yau $n$-manifold $X$, defined in terms of the symplectic structure $\omega$  on $X$, to the B-model on another  Calabi-Yau  $n$-manifold $Y$ (the ``mirror" of $X$),  defined in terms of the complex  structure $J$ on $Y$.   
Under mirror symmetry, genus-zero Gromov-Witten invariants counting holomorphic spheres in $X$ correspond to period integrals of a
nowhere vanishing holomorphic $n$-form on $Y$; Lagrangian submanifolds in $(X,\omega)$ correspond to complex submanifolds in $(Y,J)$ or more generally coherent sheaves on $(Y,J)$.
The mirror map relates local holomorphic coordinates on the  complexified K\"{a}hler moduli of $X$ to local holomorphic coordinates on the complex moduli of $Y$; a class in the complexified K\"{a}hler moduli is represented
by a complexified K\"{a}hler structure $\omega_\bC = B + i\omega$ where $B$ is a real closed 2-form on $X$ known as a B-field, and $\omega$ is a K\"{a}hler form (which is in particular a symplectic form). 

The homological mirror symmetry (HMS) conjecture proposed
by Kontsevich in his 1994 ICM talk \cite{hms} relates Fukaya's $A_\infty$-category $\Fuk(X,\omega)$ of Lagrangian submanifolds in $(X,\omega)$, also known as the category of A-branes in $(X,\omega)$, to the bounded derived category $D^b\Coh(Y,J)$ of coherent sheaves on $(Y,J)$, also known as the category of B-branes in $(Y,J)$. More precisely, the HMS conjecture asserts that there is an equivalence of triangulated categories: 
\begin{equation} \label{eqn:hms}
D^b \Fuk(X, \omega) \cong  D^b \Coh(Y, J),
\end{equation}
where $D^b \Fuk(X,\omega) = H^0(\mathrm{Tw}\Fuk(X,\omega))$ is
a triangulated category obtained by taking cohomology of the triangulated envelope $\mathrm{Tw} \Fuk(X,\omega)$ of  
the $A_\infty$-category $\Fuk(X,\omega)$. The enhanced  $A_\infty$-version of \eqref{eqn:hms} asserts that there is
a quasi-equivalence of triangulated $A_\infty$-categories:
\begin{equation}\label{eqn:hms-dg}
\mathrm{Tw} \Fuk(X,\omega) \cong D^b_{dg}\Coh(Y,J) 
\end{equation}
where $D^b_{dg}\Coh(Y,J)$ is a triangulated dg category which is a dg-enhancement of $D^b\Coh(Y,J)$.  

In 1996, Strominger-Yau-Zaslow \cite{SYZ} proposed a geometric construction of the mirror $Y$ based on T-duality. They proposed  that there exists a 
fibration $\mu: X\to \Delta$ by special Lagrangian
tori in $X$, and that the mirror $Y$ is the total space
of the dual torus fibration $\check{\mu}: Y \to \Delta$: for any $b\in \Delta$, 
$$
Y_b = \check{\mu}^{-1}(b) =\Hom(\pi_1(X_b), U(1))\cong H^1(X_b,\bR)/H^1(X_b;\bZ)
$$ 
is the dual torus of $ X_b = \mu^{-1}(b) \cong H_1(X_b,\bR)/H_1(X_b;\bZ)$. 
In other words, the mirror $Y$ is the moduli of special Lagrangian tori in $X$ together with flat $U(1)$-connections. The complexified symplectic structure $\omega_\bC = B + i \omega$ on $X$ determines the complex structure $J$ on $Y$. Moreover, the T-dual of an A-brane in $X$ is a B-brane in $Y$, so the Strominger-Yau-Zaslow (SYZ) transform also provides a geometric construction of the equivalence \eqref{eqn:hms} and the quasi-equivalence \eqref{eqn:hms-dg} at the level of objects. In general there are singular fibers so one can only apply the SYZ construction to an open subset $\Delta_0\subset \Delta$ where
the fibers are smooth tori $\bT^n=(S^1)^n$. When $\Delta_0\neq \Delta$, there exist holomorphic disks in $X$ with boundaries in a smooth fiber $\mu^{-1}(b)$; such a disk intersects smooth fibers along circles, and intersects a singular fiber at a point. The quantum correction of the mirror map is obtained by counting such holomorphic disks. The SYZ mirror symmetry is best understood when $\Delta^\circ =\Delta$ and there are no quantum corrections \cite{leung_without_corrections}, for example when 
$X= \bT^{2g}= (S^1)^{2g}$ is a symplectic torus and $Y$ is an abelian variety $V$ of (complex) dimension $g$.
Under SYZ mirror symmetry, a Lagrangian in $\bT^{2g}=\bR^{2g}/\bZ^{2g}$ (equipped with a $U(1)$-connection) which is a degree $r$ cover of the base of the SYZ fibration $\mu: \bT^{2g}\to (S^1)^g$ corresponds to a rank $r$ holomorphic vector bundle on mirror abelian vareity $V$. In particular, a Lagrangian section of $\mu:\bT^{2g}\to (S^1)^g$
corresponds to a holomorphic line bundle on the mirror $V$. 

Homological mirror symmetry (HMS) for abelian varieties has been extensively studied.  One of the earliest examples of HMS was worked out 
in the $g=1$ case for elliptic curves in \cite{zp} at the level of cohomology (see \cite{Polishchuk_massey, Polishchuk_Ainfinity} about higher products). 
The approach of \cite{zp} involves explicit counts of holomorphic triangles bounded by Lagrangian sections to compute the product $\mu^2$ in the Fukaya category of $\bT^2$ and match that  with the product formulas for theta functions which are morphisms between  line bundles on the mirror elliptic curve. A different approach was taken in  \cite{Fuk02} to prove HMS at the $A_\infty$-level, for abelian varieties of any dimension $g$. 
Following the approach of \cite{zp}, a cohomological level HMS result for abelian surfaces (i.e. the $g=2$ case) was shown in \cite{Ca} when the B-field is zero and Lagrangian objects are equipped with trivial $U(1)$-connections. In all the aforementioned work, affine Lagrangians (i.e. Lagrangians in $\bT^{2g}=\bR^{2g}/\bZ^{2g}$ defined by $g$ affine functions on $\bR^{2g}$) were used to compute the Fukaya category. For affine Lagrangians, the Floer homology and higher products $\mu^k$  can be defined over $\bC$ \cite[Remark 0.4]{Fuk02} and indeed the coefficient ring is $\bC$ in \cite{zp, Fuk02}. On the B-side, the derived category $D^b\Coh(V)$ is generated by line bundles. However, on the A-side, when $g>1$, it is not clear whether affine Lagrangians generate the entire Fukaya category. 
In \cite{Fuk02} and \cite{Ca},  explicit embeddings of (subcategories) of $D^b\Coh(V)$ into the Fukaya category were constructed, but it was not obvious that these embeddings actually induce equivalences of triangulated categories.  

A tautologically unobstructed Lagrangian is one for which there is a choice of almost complex structure $J$ such that it bounds no nonconstant $J$-holomoprhic disc.  In particular, affine Lagrangians are tautologically unobstructed. For a general pair
of tautological unobstructed Lagrangians, Floer homology is defined over an appropriate Novikov ring or Novikov field.  The technology in \cite{ab_sm} can be used to proved
the following equivalence of triangulated categories defined over the Novikov field for any positive integer $g$: 
$$
D^\pi \Fuk_{\mathfrak{tu}}(\bT^{2g}, \omega_{\text{std}}) \cong D^b\Coh(E^g)
$$
where $D^\pi \Fuk_{\mathfrak{tu}}(\bT^{2g}, \omega_{\text{std}} )$ is
the split-closed derived Fukaya category consisting of tautological unobstructed
Lagrangians in the standard symplectic torus $(\bT^{2g},\omega_{\text{std}})$, and $E$ is an elliptic curve.

Given a symplectic manifold $(X,\omega)$ equipped with a Lagrangian torus fibration which admits a Lagrangian section, Kontsevich-Soibelman \cite{ks_abel} associated a rigid analytic space $\cY$ which is the analogue of the SYZ mirror, and conjectured
that the Fukaya category of $X$ is equivalent to the derived category of (rigid analytic) coherent sheaves on $\cY$.  Given a closed symplectic manifold $(X,\omega)$ equipped with a Lagrangian torus
fibration whose base has vanishing second homotopy group, Abouzaid \cite{ab17, abouzaid2021homological}
proved that the family Floer functor defines a fully faithful embedding from the Fukaya category of tautologically unobstructed and graded Lagrangians in $(X,\omega)$ into the derived category of coherent sheaves on its rigid analytic mirror $\cY$. The embedding induces an equivalence when $X=\bT^{2g}$.

Working over the Novikov {\it ring} amounts to working on a formal neighborhood of the large radius limit in the  K\"{a}hler moduli, while working over the Novikov {\it field} amounts to working on a {\it punctured} formal neighborhood in the K\"{a}hler moduli.  The rigid analytic mirror constructed in \cite{ks_abel, ab17, abouzaid2021homological} is defined over the Novikov field. On the other hand, the mirror Calabi-Yau manifold $Y$ in the physics literature constructed by T-duality is often defined over $\bC$. 
Moreover, in some cases, the global complex moduli of $Y$ is known, and one may consider global mirror symmetry over the complex moduli of $Y$. See \cite{Chiodo-Ruan} for a version of global mirror symmetry when $X$ is a quintic Calabi-Yau threefold. 

A long-standing problem is to extend global mirror symmetry from the classical version in terms of quantum cohomology and the Dubrovin connection  on the A-model and period integrals and the classical Gauss-Manin connection on the B-model, to the categorical version in terms of the Fukaya category on the A-model and derived category of coherent sheaves on the B-model. When $X$ is a symplectic torus, the quantum cohomology is simply the classical cohomology (no quantum correction). 
Beyond affine Lagrangians, it is not clear if  the product operations in the Fukaya category can be defined over $\bC$. In this paper, following \cite{zp, Fuk02}, we  work over $\bC$ by considering 
the full subcategory of the Fukaya category generated by the affine Lagrangians, and study SYZ and homological mirror symmetry over the global complex moduli of the mirror abelian variety $Y$.  
Interestingly, moduli of tropical abelian varieties arise naturally in our study of global SYZ mirror symmetry; see Section \ref{sec: globalSYZ-intro} below. 

We now briefly describe some construction in Chapter 1 (Section 1-5) of \cite{Fuk02}.
Let $(\bT^{2g}, \omega_\bC = B + i\omega)$ be a symplectic torus with a B-field.  Given $\hat{\ell}_0 = (\ell_0, \varepsilon_0)$ and $\hat{\ell}_1 = (\ell_1, \varepsilon_1)$, where $\ell_0$ and $\ell_1$ are affine Lagrangian submanifolds of $(\bT^{2g},\omega)$ intersecting each other transversely at $r$ points, and $\varepsilon_j$ is a flat $U(1)$-connection on $\ell_j$, the Floer cohomology group $HF^*(\hat{\ell}_0, \hat{\ell}_1)$ (with complex coefficient) forms a holomorphic vector bundle $\cE(\hat{\ell}_0, \hat{\ell}_1)$ of rank $r$ over the complex manifold $\cM(\hat{\ell}_0)\times \cM(\hat{\ell}_1)$, where $\cM(\hat{\ell}_j)$ is the moduli space of $\hat{\ell}_j$ which is a complex torus of dimension $g$.  When $\ell_0$ is a fiber of $\bT^{2g}\to \bT^g = (S^1)^g$, $\cM(\hat{\ell}_0)$ can be identified with the mirror abelian variety $V$ of $(\bT^{2g},\omega_\bC)$; see Section \ref{sec:SYZ} for more details.  The holomorphic vector bundle 
$\cE(\hat{\ell}_0, \hat{\ell}_1)\big|_{\cM(\hat{\ell}_0)\times \{\hat{\ell}_1\} }$ over
$V=\cM(\hat{\ell}_0)\times \{\hat{\ell}_1\}$ is a B-brane on $V$; it is the SYZ mirror  of the A-brane $\hat{\ell}_1$ in $(\bT^{2g},\omega_\bC)$, and is also the image of $\hat{\ell}_1$ under the family Floer functor $D^\pi \Fuk_{\mathfrak{tu}}(\bT^{2g}, \omega_{\bC}) \to D^b\Coh(V, J)$. In this paper, we consider
the universal vector bundle $\cE(\hat{\ell_0}, \hat{\ell_1})$ over the universal moduli 
$\cM(\hat{\ell}_0)\times \cM(\hat{\ell}_1)$, where  $\ell_0$ and $\ell_1$ are sections or fibers of a 
$\bT^{2g} \to \bT^g$, as we vary $J$ in the global complex moduli of $V$.

\subsection{Statements of main results} \label{sec:main-results}
In this paper, we consider an abelian variety $V$ of complex dimension $g$ that is topologically the compact real $2g$-dimensional torus $(S^1)^{2g}$.  In our case
$V$ is a principally polarized abelian variety (ppav) $(V_\tau, \omega_{V_\tau})$ determined by a point $\tau$ in the genus-$g$ Siegel upper half-space 
\begin{equation} 
\cH_g :=\{ \tau=(\tau_{jk})   \in S_g(\bC)\mid  \Imag \tau >0  \text{ (i.e. $\Imag \tau$ is positive definite)}\},
\end{equation}
where $S_g(\bC)$ is the set of complex symmetric $g\times g$ matrices, and $\omega_{V_\tau}$ is the polarization.  
Denote by $M_g(\bC)$ the space of $g\times g$ matrices 
with entries in $\bC$.
Then $M_g(\bC)$ is a complex vector space of dimension $g^2$ and $S_g(\bC)$ is a complex vector subspace
of $M_g(\bC)$ of dimension $g(g+1)/2$. For any $\tau\in \cH_g$,  we define $V_\tau=(\bC^*)^g/\Gamma_\tau$,
where $\Gamma_\tau =\tau \bZ^g$ acts multiplicatively on $(\bC^*)^g$.  More explicitly, any element in $\Gamma_\tau$ is of the form
$$
\tau n  = ( \sum_{k=1}^g  \tau_{1k} n_k,\ldots, \sum_{k=1}^g  \tau_{gk }n_k)
$$
for some $n=(n_1 \ldots ,n_g)\in \bZ^g$. Let $\bi=\sqrt{-1}$. The action of $\tau n\in \Gamma_\tau$ on $(x_1,\ldots, x_g)\in (\bC^*)^g$ is given by 
\begin{equation}\label{eq:multiplicative action}
  (\tau n)\cdot (x_1,\ldots, x_g) = (e^{2\pi \bi \sum_{k=1}^g \tau_{1k} n_k} x_1, \ldots,
  e^{2\pi \bi \sum_{k=1}^g \tau_{gk} n_k}  x_g).
\end{equation}
Equivalently, we define $V^+_\tau:= \bC^g/(\bZ^g + \tau \bZ^g)$,
where the lattice $\bZ^g + \tau\bZ^g$ acts additively on $\bC^g$.
The surjective group homomorphism 
\begin{equation} \label{eqn:exp}
\bC^g \to (\bC^*)^g,   \quad (z_1,\ldots,z_g)\mapsto (e^{2\pi \bi z_1}, \ldots, e^{2\pi \bi z_g}) = (x_1, \ldots, x_g)
\end{equation} 
from the additive group $\bC^g$ to the multiplicative group $(\bC^*)^g$  descends to an isomorphism
\begin{equation} \label{eq: exponential identification}
V^+_\tau = \bC^g/(\bZ^g+\tau\bZ^g) \stackrel{\cong}{\longrightarrow}  V_\tau=(\bC^*)^g/\Gamma_\tau.
\end{equation}

The complex manifold $V_\tau$, with $\tau=B+\bi \Omega$, corresponds under mirror symmetry to the symplectic torus $(\bT^{2g}=\bR^{2g}/\bZ^{2g}, \omega_\tau)$, where $\omega_\tau$ is the complexified symplectic form 
\begin{equation}
\omega_\tau=\omega_B+ \bi \omega_\Omega, \quad \omega_B=\sum_{j,k=1}^gB_{jk}dr_j\wedge d\theta_j, \quad \omega_\Omega=\sum_{j,k=1}^g\Omega_{jk}dr_j\wedge d\theta_k, 
\end{equation}
with $\{(r_j,\theta_j): 1\leq j\leq g\}$ being coordinates on $\bR^g$.  In the above, $\omega_B$ is a real closed 2-form  known as the B-field, and $\omega_\Omega$ is a real, closed, and nondegenerate $2$-form, i.e.  a  real symplectic form.

\subsubsection{Global SYZ mirror symmetry} \label{sec: globalSYZ-intro}

In the framework of SYZ mirror symmetry \cite{SYZ, Gross}, this mirror duality can be understood as 
\begin{equation} \label{eq: SYZ fibration} 
\begin{tikzcd}
 V_\tau = (\bC^*)^g/\Gamma_\tau \ar[dr, "\pi_\tau^{\SYZ}"'] && \bT^{2g}=\bR^{2g}/\bZ^{2g} \ar[dl, "\widecheck{\pi}^{\SYZ}_\tau"] \\
 &T_\Omega = \bR^g/\Omega \bZ^g  &
\end{tikzcd}
\end{equation}

In this paper, we vary $\tau$ in $\cH_g$ to obtain a global SYZ mirror symmetry. In \eqref{eq: SYZ fibration}, the base $T_\Omega$ of the torus fibrations
$\pi_\tau^{\SYZ}$ and $\widecheck{\pi}^{\SYZ}_\tau$
is a genus-$g$ pure tropical principally
polarized abelian variety. As we vary $\tau$ in $\cH_g$, its imaginary part 
$\Omega =\mathrm{Im}\tau$ varies in the genus-$g$ pure tropical Siegel space 
$$
\cH_g^{\tropp} =\{ \Omega \in S_g(\bR)\mid \Omega \text{ is positive definite} \},
$$
the set of $g\times g$ real symmetric positive definite matrices. 
There is a map $\cH_g \lra \cH_g^{\tropp}$ given by $\tau\mapsto \mathrm{Im}\tau$. We have the following global SYZ mirror symmetry over $\cH= \cH_g$:
\begin{equation}\label{eqn:globalSYZ-H}
\xymatrix{
V_{\cH} \ar[dr]_{\pi_{\cH}^{\SYZ}}  &  &  \widecheck{V}_{\cH} \ar[dl]^{\widecheck{\pi}_{\cH}^{\SYZ}}\\
& \cT_{\cH} \ar[d] &  \\
& \cH &}
\end{equation}
The restriction of the above diagram to a point $\tau\in \cH$ is \eqref{eq: SYZ fibration}. Moreover, we introduce the moduli 
$$
\cA^F_g = [\cH_g/P_g(\bZ)]
$$
of principally polarized abelian varieties together with an SYZ fibration, where $P_g(\bZ)\subset \Sp(2g,\bZ)$ is the (integral) Siegel parabolic subgroup (see Section \ref{sec: universal SYZ} for more details). The map $\cH_g \lra \cH_g^\tropp$ induces a map $\cA^F_g\lra \cA_g^\tropp$, where 
$$
\cA_g^\tropp =[\cH_g^\tropp/\GL_g(\bZ)]
$$
is the moduli of pure tropical principally polarized abelian varieties of dimension $g$. The group $P_g(\bZ)$ acts on 
$V_{\cH}$, $\widecheck{V}_{\cH}$, and $T_{\cH}$ such that all the arrows in \eqref{eqn:globalSYZ-H}
are $P_g(\bZ)$-equivariant. Taking the (stacky) quotient of \eqref{eqn:globalSYZ-H}
by these $P_g(\bZ)$-actions, we obtain the following global SYZ mirror symmetry over $\cA^F=\cA^F_g$: 
\begin{equation}\label{eqn:globalSYZ-A}
\xymatrix{
V_{\cA^F} \ar[dr]_{\pi_{\cA^F}^{\SYZ}}  &  &  \widecheck{V}_{\cA^F} \ar[dl]^{\widecheck{\pi}_{\cA^F}^{\SYZ}}\\
& \cT_{\cA^F} \ar[d] &  \\
& \cA^F &}
\end{equation}

To state our results on homological mirror symmetry, we first introduce various categories of 
B-branes on $V_\tau$ and categories of A-branes on $(\bT^{2g}, \omega_\tau)$.

\subsubsection{Categories of B-branes} \label{sec:B-branes}
We introduce various categories of B-branes on the abelian variety $V_\tau$. 
\begin{itemize}
\item Let $D^b \Coh(V_\tau)$ be the bounded derived category of coherent sheaves on $V_\tau$. It is equivalent to
$\mathrm{Perf}(V_{\tau})$, the triangulated category of perfect complexes on $V_\tau$. 
\item Let $\cB_\tau$ be the full subcategory of $D^b\Coh(V_\tau)$ whose 
objects are of the form $L[j]$ where $L$ is an element in the group
$\Pic^{\bZ\omega_{V_\tau}}(V_\tau)$ (line bundles whose first Chern classes are integral multiples of $\omega_{V_\tau}$) and $j$ is the shift in grading.  Any element in $\Pic^{\bZ\omega_{V_\tau}}$ is of the form
$$
\cL_{\bk,[v]} := \cL^{\bk} \otimes \bbL_{[v]}
$$
where $\cL$ is the principal polarization (in particular
$c_1(\cL)=\omega_{V_\tau}$), $\bk\in \bZ$, $[v]\in V_\tau^+$, and
$\bbL_{[v]} \in \Pic^0(V_\tau)$. 
 
\item Given any $L\in \Pic(V_\tau)$ with $c_1(L) =\omega$, let
$\cB_L$ denote the full subcategory of $\cB$ whose objects are $\{ L^{\otimes \bk}[j]: j, \bk\in \bZ\}$.  
In particular, $\cB_{\cL}$ in this paper corresponds to 
$D^b_{\cL}\Coh(V_\tau)$ in \cite{Ca}.
\end{itemize}
Note that the subcategories $\cB_L$ and $\cB_\tau$ are not triangulated categories.

Let $\cT$ be a triangulated cateogory, and let $\cI$ be a full subcategory 
of $\cT$; we do not assume $\mathcal{I}$ is a triangulated category. 
Let $\langle \mathcal{I}\rangle$ denote the smallest triangulated subcategory
of $\cT$ which contains $\cI$ and is closed under direct summand.  By \cite{Orlov_generation}, 
$$
\langle \cB_L\rangle =\langle \cB_\tau \rangle = D^b\Coh(V_\tau).
$$

\subsubsection{Categories of A-branes}\label{sec:A-branes}
Given any $\bk\in \bZ$ and $[v = a +\tau b]\in V_\tau^+$, where $a, b\in \bR^g$, we construct
$$
\hat{\ell}_{\bk,[v]} = (\ell_{\bk,b}, \varepsilon_a)
$$
where $\ell_{\bk,b} = \{ (r,\theta) \in \bR^{2g}/\bZ^{2g}: \theta = b-\bk r\}$ 
is an affine Lagrangian in the symplectic torus $\bT^{2g}= \bR^{2g}/\bZ^{2g}$
and  $\varepsilon_a$ is a flat $U(1)$-connection on the trivial line bundle on $\ell_{\bk,b}$.   The Lagrangian
$\ell_{\bk,b}$ depends only on $[b]\in \bR^g/\bZ^g$, and the gauge equivalence
class of $\varepsilon_a$ depends only on $a\in \bR^g/\bZ^g$. 

We now introduce various categories of A-branes in $(\bT^{2g},\omega_\tau)$.
We consider strictly unobstructed Lagrangians so that the differential of the Lagrangian Floer complex squared to zero and the Lagrangian Floer cohomology groups $HF^*$ are defined. 
Moreover, we consider affine Lagrangians (which are in particular strictly unobstructed) so that $HF^*$ is defined not only over a Novikov ring/field but is defined over $\bC$, so that we may study global mirror symmetry over $\cH_g$ or
$\cA_g^F$.

\begin{itemize}
\item Let $\Fuk_{\aff}(\bT^{2g},\omega_\bC)$ be the category whose objects are affine Lagrangians in $(\bT^{2g},\omega_\bC)$ equipped with flat $U(1)$-connections and their shifts, 
$\Hom(\hat{\ell}[j], \hat{\ell}'[j']) = HF^*(\hat{\ell}, \hat{\ell}')[j'-j]$, the Lagrangian 
intersection Floer cohomology group with complex coefficients, which is a graded
vector space over $\bC$. 
\item Let $\cA_\tau$ be the full subcategory of $\Fuk_{\aff}(\bT^{2g},\omega_\tau)$ whose objects
are $\{ \hat{\ell}_{\bk,[v]}[j]: \bk, j\in \bZ, [v]\in V_\tau^+\}$, i.e. SYZ mirrors of objects in $\cB_\tau$.
\item Let $\cA_L$ be the full subcatgory of $\Fuk_{\aff}(\bT^{2g}, \omega_\tau)$ whose objects
are SYZ mirrors of objects in $\cB_L$. 
\end{itemize}

\subsubsection{Homological Mirror Symmetry}\label{sec:HMS-intro}
In this paper, following the explicit approach of \cite{zp} (in the $g=1$ case) and \cite{Ca} (in the $g=2$ and $a=b=0$ case), we demonstrate HMS for principally polarized abelian varieties of any dimension $g$ at the cohomological level. Let $\tau \in \cH_g$.
\begin{enumerate}
\item[(i)] Given a pair of elements $\cL_{\bk_1,[z_1]}$ and $\cL_{\bk_2, [z_2]}$ in 
$\Pic^{\bZ\omega_{V_\tau}}(V_\tau)$, where $\bk_1, \bk_2\in \bZ$  and $[z_1], [z_2]\in V_\tau^+$, we provide an explicit isomorphism
$$
\Phi_\tau^1:
HF^*(\hat{\ell}_{\bk, [z]}, \hat{\ell}_{\bk,[z']})
\stackrel{\cong}{\longrightarrow} \Ext^*(\cL_{\bk,[z]}, \cL_{\bk',[z']}) =
H^*(V_\tau, \cL_{\bk_2-\bk_2, [z_2-z_1]} )
$$
as a graded vector space over $\bC$.
\item[(ii)] Given a triple of elements $\cL_{\bk_1,[z_1]}$, $\cL_{\bk_2, [z_2]}$, 
$\cL_{\bk_3, [z_3]}$ in $\Pic^{\bZ \omega_{V_\tau}}(V_\tau)$, the following diagram commute
\[\begin{tikzcd}
 HF^{j_2}(\hat \ell_{\sk_2,[z_2]},\hat \ell_{\sk_3,[z_3]}) \otimes HF^{j_1}(\hat \ell_{\sk_1,[z_1]},\hat \ell_{\sk_2,[z_2]}) \arrow[r, "\mu^2"]\arrow[d, "\Phi^1_\tau\otimes \Phi^1_\tau"'] & HF^{j_1+j_2}(\hat \ell_{\sk_1,[z_1]},\hat \ell_{\sk_3,[z_3]})\arrow[d,"\Phi^1_\tau"] \\
 \Ext^{j_2}(\cL_{\sk_2,[z_2]},\cL_{\sk_3,[z_3]}) \otimes \Ext^{j_1}(\cL_{\sk_1,[z_1]},\cL_{\sk_2,[z_2]}) \arrow[r, "\otimes"]& \Ext^{j_1+j_2}(\cL_{\sk_1,[z_1]},\cL_{\sk_3,[z_3]})
 \end{tikzcd}.
 \]
\end{enumerate}

By (i) and (ii), we have the following version of homological mirror symmetry at the cohomological level: 
\begin{theorem}[HMS at the cohomological level] For every $\tau\in \cH$, the mirror functor 
$$
\Mir_\tau: \cA_\tau \lra \cB_\tau
$$
sending $\hat{\ell}_{\sk, [z]}[j]$ to its SYZ mirror $\cL_{\sk, [z]}[j]= \cL^{\otimes \sk}\otimes \bbL_{[z]} [j]$
(where $\sk, j\in \bZ$ and $[z]\in V_\tau^+$) is an equivalence of categories. Therefore, we have
a fully faithful embedding
$$
\Phi_\tau : \cA_\tau \longrightarrow D^b \Coh(V_\tau)
$$
of cohomological categories. In particular, the product structures match under $\Phi_\tau$: 
\begin{equation} 
\Phi_\tau \circ \mu^2_{\cA_\tau} = \mu^2_{D^b\Coh(V_\tau)} \circ \Phi_\tau.
\end{equation}
\end{theorem}

Restricting the above mirror functor to the fully faithful subcategory $\cA_L$ of $\cA_\tau$, we obtained
the following core HMS (in the sense of \cite[Definition 1.6]{PS23}) at the cohomological level. 
\begin{corollary}[core HMS at the cohomological level] 
For every $\tau\in \cH$, and any $L =\cL\otimes \bbL_{[z]} \in \Pic^{\mathsf{1}}(V_\tau)$, 
the functor 
\begin{equation} \label{eqn:A-to-B-L}
\Mir_{\tau,L}: \cA_L \lra \cB_L
\end{equation}
sending $\hat{\ell}_{\sk,[kz]}$ to 
$L^{\otimes k} = \cL_{\sk, [kz]}$
is an equivalence of categories. 
Therefore, we have
a fully faithful embedding
\begin{equation}\label{eqn:AL-to-DCoh}
\Phi_{\tau,L} : \cA_L \longrightarrow D^b \Coh(V_\tau)
\end{equation}
of cohomological categories. In particular, the product structures match under $\Phi_{\tau, L}$: 
\begin{equation}
\Phi_{\tau, L}\circ \mu^2_{\cA_L} = \mu^2_{D^b\Coh(V_\tau)} \circ \Phi_{\tau,L}.
\end{equation}
\end{corollary}

\subsubsection{Ring of sections and Seidel's mirror map}
For each $\tau\in \cH$, let $\hat{\ell}_{\sk}= \hat{\ell}_{\sk,[0]}$, and define a graded vector space over $\bC$:
\begin{equation}\label{eqn:widecheck-S-tau}
\widecheck{S}_\tau := \bigoplus_{\sk\geq 0} HF^0(\hat \ell_\szero, \hat \ell_{\sk}).
\end{equation}
where $\deg s =\sk$ if $s\in HF^0(\hat\ell_\szero, \hat\ell_{\sk})$.  

Let 
$$
S_\tau =\bigoplus_{\sk\geq 0} H^0 (V_\tau, \cL^{\otimes \sk})
$$
be the ring of sections of $\cL$, and  $\deg s =\sk$
if $s\in H^0(V_\tau, \cL^{\otimes \sk})$. Then $(S_\tau, \otimes)$ is a graded ring which is a graded commutative $\bC$-algebra. 

The special case $L=\cL$  of Corollary \ref{cor:coreHMS} implies
the following:  
\begin{theorem}\label{thm:ring} $(\widecheck{S}_\tau, \mu^2)$ is a graded 
commutative $\bC$-algebra, and 
$$
\Phi_\tau: (\widecheck{S}_\tau, \mu^2) \lra (S_\tau, \otimes)
$$
is an isomorphism of graded commutative $\bC$-algebras. 
\end{theorem}

The above isomorphism has been studied in \cite{Zaslow05, Aldi-Zaslow, Aldi09} where
it is called Seidel's mirror map.  
Theorem \ref{thm:ring} is an analogue of \cite[Theorem 1.1]{AbouzaidThesis}.

\subsubsection{Family Floer cohomology and global HMS}
We first fix a complexified symplectic structure $\omega_\tau$ on $\bT^{2g}$ where
$\tau \in \cH_g$. For each integer $\sk$, we define the moduli space
$$
\cM(\hat{\ell}_{\sk})_\tau := 
\big\{ \hat{\ell}_{\sk, [z]} \mid [z]\in V_\tau^+ \big\}
$$
of objects in $\Fuk_{\aff}(\bT^{2g}, \omega_\tau)$ mirror to line bundles in 
$$
\Pic^{\sk}(V_\tau):= \{ L\in \Pic(V_\tau)\mid c_1(L)= \sk \omega_{V_\tau} \}.
$$
There is a bijection
$$
\Mir_\tau^{\sk}: \cM(\hat{\ell}_{\sk})_\tau\lra \Pic^{\sk}(V_\tau),
\quad \hat{\ell}_{\sk,[z]} \mapsto \cL_{\sk, [z]}.
$$
We equip the set $\cM(\hat{\ell}_{\sk})_\tau$ with the structure of a complex manifold such that the above bijection is an isomorphism of complex manifolds.  Item (i) of 
Section \ref{sec:HMS-intro} is equivalent to the following statement. 
\begin{theorem} \label{thm:HF-Ext}
For any $\sk, \sk'\in \bZ$, any $w\in \{0,1,\ldots, g\}$ and any pair
$(\hat{\ell}, \hat{\ell}') \in \cM(\hat{\ell}_{\sk})_\tau \times \cM(\hat{\ell}_{\sk'})_\tau$, we have
$$
HF^w(\hat{\ell}, \hat{\ell}') \cong \Ext^w\big(\Mir^{\sk}_\tau(\hat{\ell}), 
\Mir_\tau^{\sk'}(\hat{\ell}')\big).
$$
\end{theorem} 

As we vary $(L, L')$ in $P:= \Pic^{\sk}(V_\tau)\times \Pic^{\sk'}(V_\tau)$, the extension groups  $\Ext^w(L, L')$ form a coherent sheaf $\cE^w_{\tau,\sk, \sk'}$ of $\cO_P$-modules on $P$, i.e., 
the fiber of the sheaf $\cE^w_{\tau, \sk, \sk'}$ over $(L, L') \in P$ is 
the $w$-th extension group  $\Ext^w(L, L')$ which is a complex vector space.  
We now vary $(\hat{\ell}, \hat{\ell}')$ in $\widecheck{P}:= \cM(\hat{\ell}_{\sk})_\tau\times 
\cM(\hat{\ell}_{\sk'})_\tau$ to obtain a family version of Theorem \ref{thm:HF-Ext}. There is an isomorphism 
$$
\Mir^{\sk}_\tau\times \Mir^{\sk'}_\tau:
\widecheck{P} = \cM(\hat{\ell}_{\sk})_\tau\times \cM(\hat{\ell}_{\sk'})_\tau
\lra P= \Pic^{\sk}(V_\tau)\times \Pic^{\sk'}(V_\tau)
$$
of complex projective manifolds.  

\begin{theorem}
For any $\tau\in \cH_g$,  any $\sk, \sk'\in \bZ$, and any $w\in \{0,1,\ldots,g\}$, there
is a coherent sheaf $\widecheck{\cE}^w_{\sk, \sk'}$ of $\cO_{\widecheck{P}}$-modules
on $\widecheck{P}$ whose fiber over $(\hat{\ell}, \hat{\ell}') \in \widecheck{P}$ is the $w$-th Floer cohomology group $HF^w(\hat{\ell}, \hat{\ell}')$, such that
$$
\widecheck{\cE}^w_{\tau,\sk, \sk'} \cong (\Mir^{\sk}_\tau \times \Mir^{\sk'}_{\sk})^*\cE^w_{\tau,\sk,\sk'}.
$$
\end{theorem}

Finally, we vary $\tau\in \cH=\cH_g$ to obtain a global and universal family version of Theorem \ref{thm:HF-Ext}. We  have universal moduli of objects
$$
\Mir^{\sk}_\cH: \cM(\hat{\ell}_{\sk})_\cH 
\stackrel{\cong}{\lra} \Pic^{\sk}(V_\cH)\to \cH. 
$$
The fiber of $\cM(\hat{\ell}_{\sk})_\cH \to \cH$ over
$\tau\in \cH$ is $\cM(\hat{\ell}_{\sk})_\tau$; the fiber
of $\Pic^{\sk}(V_\cH)\to \cH$ over $\tau\in \cH$ is
$\Pic^{\sk}(V_\tau)$. There is a coherent sheaf
$\cE^w_{\cH, \sk, \sk'}$ of $\cO_P$-modules on 
$$
P = \Pic^{\sk}(V_{\cH})\times_{\cH}
\Pic^{\sk'}(V_{\cH}) 
$$
such that the fiber of $\cE^w_{\cH, \sk, \sk'}$ 
over $(\tau, L, L') \in P$ is
the $w$-th extension group $\Ext^w(L, L')$. 
There is an isomorphism 
$$
\Mir^{\sk}_{\cH}\times_{\cH} \Mir^{\sk'}_{\cH}: 
\widecheck{P} = \cM(\hat{\ell}_{\sk})_{\cH}\times_{\cH} \cM(\hat{\ell}_{\sk'})_{\cH}
\lra
P = \Pic^{\sk}(V_{\cH})\times_{\cH}\Pic^{\sk'}(V_{\cH}) 
$$
of complex manifolds.
\begin{theorem} For any $\sk, \sk'\in \bZ$ and $w\in \{0,1,\ldots, g\}$ there is a coherent sheaf $\widecheck{\cE}^w_{\cH, \sk, \sk'}$ of $\cO_{\widecheck{P}}$-modules on $\widecheck{P}$ whose
fiber over $(\tau, \hat{\ell}, \hat{\ell}') \in \widecheck{P}$
is the $w$-th Floer cohomology group $HF^w(\hat{\ell}, \hat{\ell}')$, such that 
$$
\widecheck{\cE}^w_{\cH, \sk, \sk'}\cong 
(\Mir^{\sk}_{\cH}\times_{\cH} \Mir^{\sk'}_{\cH})^*\cE^w_{\cH, \sk, \sk'}.
$$

\end{theorem}

\paragraph{Acknowledgments.}
We thank Mohammed Abouzaid and Kenji Fukaya for explaining  their work \cite{abouzaid2021homological} and \cite{Fuk02}. C. Cannizzo was partially supported by NSF DMS-2316538. 

\section{Principally polarized abelian varieties \texorpdfstring{$V_\tau$}{V tau} and their SYZ mirrors \texorpdfstring{$(\bT^{2g}, \omega_\tau)$}{(T2g Omega tau)}} \label{sec: complex structure}

\subsection{Principally polarized abelian variety and their moduli}\label{sec: ppav}

We use the notation in Section \ref{sec:main-results}.  For $j\in \{1,\ldots,g\}$, let $\alpha^j:[0,1] \lra (\bC^*)^g$ be given by 
\begin{equation}\label{eq: alpha basis}
\alpha^j(s) = (1,\ldots,1, \underbrace{e^{2\pi \bi s}}_{j\text{-th}}, 1, \ldots, 1) 
\end{equation}
and let $\beta_j: [0,1] \to (\bC^*)^g$ be paths given by 
\begin{equation}\label{eq: beta basis}
\beta_j (s) = ( e^{2\pi \bi \tau_{j1}s} ,  \ldots, e^{2\pi \bi \tau_{jg} s}).
\end{equation}
Then $\alpha^j, \beta_j$ are loops in the quotient $V_\tau =(\bC^*)^g/\Gamma_\tau$;  we relax notation
and use the same symbols $\alpha^j, \beta_j$ to denote their classes in $H_1(V_\tau;\bZ)$. Then 
$\{ \alpha^j ,\beta_j \mid j=1,\ldots,g\}$ is an integral basis of $H_1(V_\tau;\bZ)\cong \bZ^{2g}$. Let
$\{ a_j ,b^j \mid j=1,\ldots,g \}$ be the dual integral basis of $H^1(V_\tau;\bZ)$, so that
$$
\int_{\alpha^j} a_k = \int_{\beta_k} b^j = \delta^j_k,\quad \int_{\alpha^j}b^k = \int_{\beta_k}a_j=0.
$$

Recall that $z_1,\ldots, z_g$ are holomorphic coordinates on $\bC^g$, while
$x_1=e^{2\pi \bi z_1},\ldots x_g = e^{2\pi \bi z_g}$ are holomorphic coordinates on $(\bC^*)^g$. The holomorphic 1-forms 
$$
dz_1,\ldots, dz_g
$$  on $\bC^g$ descend to  holomorphic 1-forms 
$$
\frac{1}{2\pi \bi} \frac{dx_1}{x_1} ,\ldots, \frac{1}{2\pi \bi}\frac{dx_g}{x_g}
$$
on $(\bC^*)^g=\bC^g/\bZ^g$ , which 
further descend to $g$ holomorphic 1-forms  on  $V_\tau=(\bC^*)^g/\Gamma_\tau$. We relax notation and
use the same symbols $dz_1,\ldots, dz_g$ to denote these $g$ holomorphic 1-forms on $V_\tau$. Then
\begin{equation}
\int_{\alpha^j} dz_ k=\delta^j_k, \quad  \int_{\beta_j} dz_k= \tau_{jk}.
\end{equation}

For $j=1,\ldots, g$, let $\hat{a}_j$ (resp.  $\hat{b}^j$)  be the unique real harmonic 1-form which represents
the class $a_j$ (resp. $b^j$) in $H^1(V_\tau;\bZ)\subset H^1(V_\tau;\bR)$. Then
\begin{equation}\label{eqn:dv-ab}
dz_j = \hat{a}_j +  \sum_{k=1}^g \tau_{jk} \hat{b}^k
\end{equation}

Define a real 2-form\footnote{More generally, 
we may consider
$\omega_{(\delta_1,\ldots,\delta_g)} =\sum_{j=1}^g \delta_j \hat{a}_j\wedge \hat{b}_j$, 
where $\delta_1,\ldots, \delta_g \in \bZ$ satisfy
$\delta_j \mid \delta_{j+1}$. See e.g. Section 6 in \cite[Chapter 2]{GH}. In this paper we will consider principal polarization
$\omega_{(1,\ldots,1)}$.}  
\begin{equation}\label{eqn:omega-ab}
\omega_{V_\tau} :=\sum_{j=1}^g  \hat{a}_j\wedge \hat{b}^j 
\end{equation}
on $V_\tau$.  Then 
$$
[\omega_{V_\tau}] = \sum_{j=1}^g a_j\cup  b^j \in H^2(V_\tau;\bZ)\subset H^2(V_\tau;\bR).
$$

Let $B:= \mathrm{Re}\, \tau$ and $\Omega:= \Imag \tau$ be the real and imaginary 
parts of $\tau$, respectively, so that $\tau= B+\bi\Omega$. Then
$\Omega= (\Omega_{jk})$ is a real symmetric positive definite
$g\times g$ matrix. Note that $\int_{\alpha^j} d\bar{z}_ k=\delta^j_k$ and $ \int_{\beta_j} d\bar{z}_k= \bar{\tau}_{jk}$, so $d\bar{z}_j=\hat{a}_j+\sum_{k=1}^g\bar{\tau}_{jk}\hat{b}^k$, then along with Equation \eqref{eqn:dv-ab}, we see that
\begin{equation}\label{eqn:omega-dv}
\omega_{V_\tau} =  \frac{\bi}{2} \sum_{j,k=1}^g \Omega^{jk}(\tau)dz_j\wedge d\bar{z}_k   
\end{equation}
where $(\Omega^{jk})$ is the inverse matrix of  $\Omega = (\Omega_{jk})$,
and $d\bar{z}_k$ is the complex conjugate of $dz_k$. So $\omega_{V_\tau}$ is a positive real (1,1)-form. 
Let $\phi_j=\mathrm{Re}z_j$
and $\xi_j = -\mathrm{Im}z_j$, so that $z_j =\phi_j - \bi \xi_j$ and $x_j = e^{2\pi (\xi_j +\bi \phi_j)}$. Then
\begin{equation}\label{eqn:omega-real-coordinate}
\omega_{V_\tau} = \sum_{j,k=1}^g \Omega^{jk}d\xi_j\wedge d\phi_k.
\end{equation}

The pair $(V_\tau,\omega_{V_\tau})$ is a principally polarized abelian variety of dimension $g$.
The polarization $\omega_{V_\tau}$ defines a symplectic form on $H_1(V_\tau;\bZ)\cong \bZ^{2g}$, and
$(H_1(V_\tau;\bZ),\omega_{V_\tau})$ is isomorphic to $(\bZ^{2g},J_g)$, where
\begin{equation}\label{eq: J_g}
J_g= \begin{bmatrix}
0  & I_g \\
-I_g  &  0 
\end{bmatrix}
\end{equation}
and $I_g$ is the $g\times g$ identity matrix. 
A {\em Torelli  structure} on a ppav $(V,\omega_V)$ is a choice of an integral symplectic basis of $(H_1(V;\bZ),\omega_V)$, or equivalently, an isometry
$\phi: (H_1(V;\bZ),\omega_V)\longrightarrow (\bZ^{2g},J_g)$, because such a choice is equivalent to specifying $\tau$ so that $V=V_\tau$.  Then $\cH_g$ is the moduli of $g$-dimensional principally polarized abelian varieties $(V,\omega_V)$ with Torelli structure.  
The universal family over $\cH_g$ is given by 
$$
\pi_{\cH}: V_{\cH} = (\cH_g \times (\bC^*)^g)/\bZ^g  \longrightarrow \cH_g,
$$
where $n=(n_1,\ldots,n_g)\in \bZ^g$ acts on $\left(\tau,x_1,\ldots,x_g \right)\in \cH_g\times (\bC^*)^g$ by
$$
(n_1,\ldots,n_g ) \cdot (\tau, x_1, \ldots, x_g)  = (\tau, e^{2\pi \bi \sum_{k=1}^g \tau_{1k} n_k} x_1,\ldots,
 e^{2\pi \bi  \sum_{k=1}^g \tau_{gk} n_k} x_g). 
$$

Let $M^T$ denote the transpose of a matrix $M$. The group
$$
\Sp(2g, \bZ) = \{ M \in \GL(2g,\bZ) \mid  M^T J_g M = J_g\}
$$
acts on the Siegel  space $\cH_g$ on the left by 
\begin{equation}\label{eq:Sp(2g,Z) action}
\left[ \begin{array}{cc}
A & B\\
C & D
\end{array} 
\right] \circ \tau =(A\tau+B)(C\tau +D)^{-1},
\end{equation}
where $\left[\begin{array}{cc}
A & B\\
C & D
\end{array}\right] \in \Sp(2g,\bZ)$ and $\tau \in \cH_g$. The quotient  
$\cA_g = [\Sp(2g, \bZ)\backslash \cH_g]$ is the moduli of $g$-dimensional principally polarized abelian varieties (ppavs). 

By \cite[Proposition 4.4 p177]{MumfordTheta}, $\Sp(2g,\bZ)$ acts on $\cH_g\times \bC^g$ by 
\begin{equation}\label{eqn:Sp(2g,Z)-Hg-Cg}
\begin{bmatrix} A & B\\ C& D \end{bmatrix} \cdot (\tau, z) = \left((A\tau +  B)(C\tau+D)^{-1}, ((C\tau +  D)^T)^{-1} z\right)
\end{equation}
where $z=\begin{bmatrix} z_1\\ \vdots \\ z_g\end{bmatrix}\in \bC^g$. 
We have a cartesian diagram
\begin{equation}
\xymatrix{
  V_{\cH_g} \ar[r]\ar[d]_{\pi_{\cH_g}} 
& V_{\cA_g}  \ar[d]^{\pi_{\cA_g}}  \\
\cH_g  \ar[r] & \cA_g
}
\end{equation}
where 
\begin{equation} \label{eqn:universal-Ag}
\pi_{\cA_g}: V_{\cA_g} =\left[ \left(\Sp(2g,\bZ)\ltimes \bZ^{2g}\right)\backslash 
\left(\cH_g\times \bC^g\right)\right] \lra 
\cA_g = [\Sp(2g,\bZ)\backslash \cH_g]
\end{equation}
is the universal family of $g$-dimensional ppavs.

\subsection{The Riemann theta function} \label{sec:theta}
In this subsection, we provide a brief review of the Riemann theta function
and the theta divisor. The main references for this subsection (Section \ref{sec:theta}) and the next three subsections (Section \ref{sec:picard}, 
Section \ref{sec:poincare}, and Section \ref{sec:shift}) are \cite{MumfordTheta} and \cite{cx_abel}. 

We first introduce some notation. We view an element $n\in \bZ^g$ as a column vector
$$
n = \begin{bmatrix} n_1\\ \vdots  \\ n_g\end{bmatrix},
$$
and let $n^T =[ n_1 \cdots n_g]$  be the transpose of $n$. Then 
$$
n^T \tau n =  \sum_{j,k=1}^g \tau_{jk} n_j n_k. 
$$

Define a holomorphic line bundle $\cL_\tau$ over $V_\tau$ by 
$$
\cL_\tau =  ( (\bC^*)^g \times \bC)/\Gamma_\tau \lra V_\tau =(\bC^*)^g/\Gamma_\tau,
$$
where $\tau n \in \Gamma_\tau$ acts on $(\bC^*)^g \times \bC$ by
$$
(\tau n) \cdot (x_1,\ldots,x_g , v ) = (e^{2\pi \bi \sum_{k=1}^g \tau_{1k}n_k}x_1,\ldots, e^{2\pi \bi \sum_{k=1}^g \tau_{gk} n_k }x_g,  e^{-\pi \bi(n^T \tau n)} 
\prod_{j=1}^g x_j ^{-n_j}v).
$$
Then $c_1(\cL_\tau) =[\omega_{V_\tau}]$. Recall that $\omega_{V_\tau}$ is a positive  real (1,1)-form, so $\cL_\tau$ is ample.  As we vary $[\tau]\in \cA_g$, we obtain a line bundle over the universal family:
\begin{equation}\label{eqn:LA}
\cL_{\cA} \longrightarrow V_{\cA} \longrightarrow  \cA_g.
\end{equation}
The restriction of \eqref{eqn:LA} to a point $[\tau]\in \cA_g$ is $\mathcal{L}_\tau \longrightarrow V_\tau \longrightarrow [\tau]$.
The (universal) Riemann theta function is
\begin{equation}
\vartheta: \cH_g \times (\bC^*)^g  \to \bC,\quad \vartheta(\tau,x) := \sum_{n\in  \bZ^g} e^{\pi \bi n^T \tau n} \prod_{j=1}^g x_j^{n_j}
=\sum_{n\in \bZ^g} e^{\pi \bi n^T \tau n} e^{2\pi \bi \sum_{j=1}^g n_j z_j},
\end{equation}
where 
$$
n=\begin{bmatrix} n_1\\ \vdots \\ n_g \end{bmatrix} \in \bZ^g, \quad 
z =\begin{bmatrix} z_1\\ \vdots \\ z_g \end{bmatrix} \in \bC^g, \quad
x_j = e^{2\pi\bi z_j}. 
$$  
The Riemann theta function is $\tau \bZ^g$-quasi-periodic on $(\bC^*)^g$ (or $(\bZ^g+\tau \bZ^g)$-quasi-periodic on $\bC^g$) in the sense that,
for $a,b\in \bZ^g$, 
\begin{equation}\label{eq: quasiperiodicity}
\vartheta(\tau, xe^{2\pi \bi(a+\tau b)})=\vartheta(\tau, e^{2\pi \bi(z+ a+\tau b)})=e^{-\bi \pi b^T\tau b-2\pi \bi b^T z} 
    \vartheta(\tau, e^{2\pi\bi z})=e^{-i\pi b^T\tau b}x^{-b} 
    \vartheta(\tau, x). 
\end{equation}
In fact, for a fixed $\tau$, any such quasi-periodic holomorphic function is a constant multiple of the $\vartheta(\tau, x)$. The map $(\tau,x)\mapsto (\tau,x,\vartheta(\tau,x))$ descends to a holomorphic section
$$
s: V_\cA = \left[ \left(\Sp(2g, \bZ)\ltimes \bZ^{2g}\right)\backslash \left(\cH_g\times \bC^g\right)\right] \lra \cL_{\cA} =\left[  
\left(\Sp(2g ,\bZ)\ltimes \bZ^{2g}\right)\backslash 
\left(\cH_g\times \bC^g\times\bC\right)\right]. 
$$ 
For fixed $[\tau]\in \cA_g$, we obtain a holomorphic section $s_\tau: V_\tau\lra \cL_\tau$.

\subsection{Principally polarized tropical abelian varieties and their moduli}\label{sec:trop_ppav}
The main reference of this subsection is \cite{CMV13}.

Let $S_g(\bR)$ be the set of real symmetric $g\times g$ matrices. It is a $g(g+1)/2$-dimensional real vector space. 
The {\em pure tropical Siegel space}  is
\begin{equation}\label{eqn:Htropp}
\cH_g^\tropp    := \{ \Omega \in S_g(\bR)\mid  \Omega \textup{ is positive definite}\} .
\end{equation}
The {\em tropical Siegel space} is 
\begin{equation} \label{eqn:Htrop}
\cH_g^\trop:= \left \{\Omega \in S_g(\bR) \left|  \begin{array}{cc} \Omega \textup{ is positive semi-definite, and there exists }  h\in \GL(g,\bZ) \\  
\text{ such that } h \Omega h^T   = \begin{bmatrix} \Omega' & 0\\ 0 & 0 \end{bmatrix} 
\text{ for some } \Omega'\in \cH^\tropp_{g'} \text{ with } 0\leq g'\leq g. \end{array}  \right. \right\} 
\end{equation}
The closure of $\cH^\trop_g$ in $S_g(\bR)$ is 
$$
\overline{\cH^{\tropp}_g  }= \left \{ \Omega \in S_g(\bR) \mid \Omega \textup{ is positive semi-definitie}  \right\}. 
$$ 
We note that $\cH_g^\tropp\subset \cH_g^\trop \subset\overline{\cH^\tropp_g}$, and
\begin{itemize}
\item  $\cH_g^\tropp$ is an open cone in $S_g(\bR)$, while  $\cH_g^\trop$
is a subcone of $\overline{\cH_g^\tropp}$; 
\item  $\cH_g^{\tropp}$  is closed under multiplication by elements in $\bR_{>0}$, while $\cH_g^\trop$ and 
$\overline{\cH_g^\tropp}$ are closed under multiplication by elements in $\bR_{\geq 0}$. 
\end{itemize} 

A {\em tropical principally polarized abelian variety} of dimension $g$ is a $g$-dimensional real torus $(S^1)^g$ equipped with a positive semi-definite flat metric.
It is of the form
$$
(\bR^g/\bZ^g, \sum_{j,k=1}^g \Omega_{jk} dr_j dr_k)
$$
where $\Omega \in \cH_g^\trop$. We say it is {\em pure}  if $\Omega \in \cH_g^\tropp$.

If $\Omega$ is positive definite (or equivalent, if $\Omega\in \cH^\tropp_g$)  then $\sum_{j,k=1}^g \Omega_{jk} dr_j dr_k$ is
a positive definite metric on $\bR^g/\bZ^g$, and there is an isometry
$$
(\bR^g/\bZ^g,\sum_{j,k=1}^g \Omega_{jk} d r_j d r_k) \lra  (T_\Omega = \bR^g/\Omega \bZ^g, \sum_{j,k=1}^g \Omega^{jk}d\xi_j d\xi_k)
$$
where  $(\Omega^{jk})$ is the inverse matrix of $(\Omega_{jk})$, and the map $\bR^g/\bZ^g \lra \bR^g/\Omega\bZ^g$ is given by
$$
r=(r_1,\ldots,r_g) \mapsto \xi = (\xi_1,\ldots,\xi_g) = (\sum_{j=1}^g \Omega_{1j}r_j,\ldots,\sum_{j=1}^g \Omega_{gj}r_j) 
$$

\subsection{The universal SYZ fibration}\label{sec: universal SYZ}
In this subsection we assume $\Omega$ is positive definite. Let  
\begin{equation}\label{eqn:log}
\log:(\bC^*)^g \lra \bR^g, \quad (x_1,\ldots,x_g) \mapsto (\xi_1,\ldots,\xi_g) =\frac{1}{2\pi}(\log|x_1|,\ldots, \log|x_g|)
\end{equation} 
be the logarithm map. Then
\begin{equation}\label{eqn:log-equivariant} 
\log(\tau n \cdot x) = \Omega n +  \log(x).
\end{equation}
So \eqref{eqn:log} descends to an SYZ fibration 
\begin{equation} \label{eqn:SYZ} 
\pi^\SYZ_\tau: V_\tau = (\bC^*)^g/\Gamma_\tau \lra  T_\Omega = \bR^g/\Omega \bZ^g. 
\end{equation}

Consider the sublattice $\Gamma_F:=H_1(T_F;\bZ)$ of $H_1(V_\tau; \bZ)\cong \bZ^{2g}$, where $T_F$ is the fiber of the SYZ fibration $\pi^{\SYZ}$ above. Then, $\Gamma_F\cong \bZ^g$ and a $\bZ$-basis $\{\alpha^j\}_{j=1}^g$ of $\Gamma_F$ can be extended to an integral symplectic basis  $\{\alpha^j, \beta_j\}_{j=1}^g$ of the symplectic lattice
$$
\big( H_1(V_\tau; \bZ),\omega_{V_\tau}\big) \cong (\bZ^{2g}, J_g).
$$
Therefore,  $\Gamma_F\otimes_\bZ \bR\cong \bR^g$ is a Lagrangian subspace of the real
symplectic vector space
$$
\big( H_1(V_\tau;\bR), \omega_{V_\tau} \big)\cong (\bR^{2g}, J_g). 
$$
Under the $\Sp(2g, \bZ)$ action on $\cH_g$ as in Equation \eqref{eq:Sp(2g,Z) action}, $\Gamma_F$ is preserved by the Siegel parabolic group $P_g(\bZ)\subset \Sp(2g, \bZ)$ defined by  
\begin{equation}\label{def: Siegel parabolic}
P_g(\bZ) =\left\{ \begin{bmatrix} A & B \\ C & D \end{bmatrix} \in \Sp(2g, \bZ)\Big| C=0 \right\}
=\left\{ \begin{bmatrix} A & B \\ 0 & (A^T)^{-1} \end{bmatrix}  \Big| A\in \GL(g,\bZ), B\in M_g(\bZ), A B^T =  B A^T \right\}.
\end{equation}
The group
  $P_g(\bZ)$ is generated by the following two subgroups:
\begin{equation}
\left\{ \begin{bmatrix} A & 0 \\ 0 & (A^T)^{-1} \end{bmatrix} \ \Big| \ A\in \GL(g,\bZ) \right\} \cong \GL(g, \bZ)\quad\text{and} \quad
\left\{ \begin{bmatrix} I & B \\ 0 & I \end{bmatrix} \ \Big| \  B\in M_g(\bZ),  B^T =  B \right\} \cong S_g(\bZ).
\end{equation}
If
$$
\begin{bmatrix} A & B \\ 0 & (A^T)^{-1} \end{bmatrix} \in P_g(\bZ),\quad \tau \in \cH_g 
$$
then
\begin{equation}\label{eq: P_g(Z) acts on tau}
\begin{bmatrix} A & B \\ 0 & (A^T)^{-1} \end{bmatrix} \circ \tau = (A\tau+B)A^T \quad \text{and} \quad \mathrm{Im} \left( (A\tau + B) A^T\right)  = A \mathrm{Im} \tau A^T.  
\end{equation}

Now let us discuss the various moduli spaces involved and how they are related.  Recall that $\cH_g$ is the moduli space of $g$-dimensional ppav's $(V,\omega_V)$ with Torelli structure.  The map from $\cH_g$ to the pure tropical Siegel space $\cH_g^{\tropp}$, which is
\begin{equation}\label{eqn:H-Htrop}
\cH_g = S_g(\bR) \times \cH^\tropp_g  \cong \bR^{g(g+1)/2} \times \cH^\tropp_g \lra \cH^\tropp_g, \quad \tau \mapsto \mathrm{Im}(\tau)
\end{equation}
descends to 
\begin{equation} \label{eqn:torus-bundle}
\cH_g/S_g(\bZ) = \left(S_g(\bR)/S_g(\bZ)\right)\times \cH^\tropp_g  \cong (S^1)^{g(g+1)/2}\times \cH^\tropp_g \lra \cH^\tropp_g,
\end{equation}
which further descends to 
\begin{equation} \label{eqn:Ag} 
\cA^F_g = [\cH_g/P_g(\bZ)]   \lra  \cA^\tropp_g = [\cH^\tropp_g/\GL(g, \bZ)] \subset \cA^\trop_g = [\cH^\trop_g/\GL(g,\bZ)], 
\end{equation} 
where $A\in \GL(g, \bZ)$ acts on $\Omega \in \cH^\trop_g$ by
$$
A\cdot \Omega = A \Omega A^T.
$$

In Equation \eqref{eqn:Ag},  $\cA^\trop_g$ is the moduli of  tropical principally polarized abelian varieties of dimension $g$, and
$\cA^\tropp_g$ is the submoduli of pure tropical principally polarized abelian variety of dimension $g$. 
$\cA^F_g$ is the moduli space of pairs $(V, \Gamma_F \subset H_1(V;\bZ))$, where
$V$ is a ppav of dimension $g$, and $\Gamma_F\cong \bZ^g$ is a rank $g$ sublattice
of $H_1(V;\bZ)\cong \bZ^{2g}$ such that a $\bZ$-basis  $\{\alpha^1,\ldots,\alpha^g\}$ of $\Gamma_F$ can be extended to an integral symplectic basis $\{\alpha^j,\beta_j\mid j=1,\ldots,g\}$ of $(H_1(V_\tau;\bZ), \omega_{V_\tau}) \cong (\bZ^{2g}, J)$.  
We may also view $\cA^F_g$ as the moduli of ppav together with a SYZ fibration $V_\tau =(\bC^*)^g/\Gamma_\tau \lra  \bR^g/\Omega \bZ^g$. Then the map $\cA^F_g \lra \cA^{\tropp}_g$ sends an SYZ fibred ppav to its base. There is a commutative diagram
\begin{equation}
\xymatrix{
\cH_g =   S_g(\bR)\times \cH_g^\tropp \cong \bR^{g(g+1)/2}\times \cH_g^\tropp   \ar[d] \ar[dr]  & \\
\cH_g/S_g(\bZ) =\left(S_g(\bR)/S_g(\bZ)\right) \times \cH_g^\tropp \cong (S^1)^{g(g+1)/2}\times \cH_g^\tropp   \ar[d] \ar[r] & \cH_g^\tropp \ar[d]\\
\cA_g^F=[\cH_g/P_g(\bZ)]  \ar[d] \ar[r]  & \cA_g^\tropp =[\cH^\tropp_g/\GL(g,\bZ)] \\
\cA_g =[\cH_g/\Sp(2g,\bZ) ] & 
}
\end{equation}
where the vertical arrows are covering maps.

We now describe universal families over various moduli spaces. Let $V_{\cA_g}\lra \cA_g$ be the universal family of principally polarized abelian varieties of dimension $g$, 
as in Equation \eqref{eqn:universal-Ag}. Let $V_{\cA_g^F}$ be the fiber product
of $\cA_g^F\to \cA_g$ and $V_{\cA_g}\to \cA_g$, so that we have the following cartesian diagram:
\begin{equation}
\xymatrix{
V_{\cA_g^F} \ar[r] \ar[d]_{\pi_{\cA^F_g}} & V_{\cA_g }  \ar[d]^{\pi_{\cA_g}} \\
\cA^F_g \ar[r] & \cA_g
}
\end{equation}
Let $\cT_{\cA^\tropp_g} \to \cA^\tropp_g$ be the universal family of pure tropical principally polarized abelian varieties of (real) dimension $g$.  Let
$\cT_{\cA^F_g}$ be the fiber product of $\cA_g^F\to \cA^\tropp_g$ and $\cT_{\cA^\tropp_g} \to \cA^\tropp_g$, so that we have the following cartesian diagram:
\begin{equation}
\xymatrix{
\cT_{\cA_g^F} \ar[r] \ar[d]_{p_{\cA_g^F}} & \cT_{\cA^\tropp_g }  \ar[d]^{p_{\cA_g^\tropp}} \\
\cA^F_g \ar[r] & \cA_g^\tropp
}
\end{equation}
The map $\pi_{\cA^F_g}:V_{\cA_g^F}\lra \cA_g^F$ is a composition 
\begin{equation}\label{eqn:universal-SYZ}
\xymatrix{
V_{\cA^F_g}\ar[r]^{\pi^{\SYZ}_{\cA_g^F}} & \cT_{\cA^F_g}\ar[r]^{p_{\cA^F_g}} & \cA^F_g
}
\end{equation}
where $\pi^{\SYZ}_{\cA^F_g}:V_{\cA^F_g}\lra \cT_{\cA^F_g}$ is the universal SYZ fibration. The restriction of \eqref{eqn:universal-SYZ} to a point
$[\tau]\in \cA_g^F$ is $V_\tau \stackrel{\pi^{\SYZ}_\tau}{\lra} T_\Omega \lra [\tau]$. We may pullback \eqref{eqn:universal-SYZ}
under the map $\cH_g\to \cA_g^F$ to obtain 
\begin{equation}\label{eqn:univeral-SYZ-H}
\xymatrix{
V_{\cH_g}\ar[r]^{\pi^{\SYZ}_{\cH_g}} & \cT_{\cH_g}\ar[r]^{p_{\cH_g}} & \cH_g
}
\end{equation}

\subsection{The Calabi-Yau metric and the metric SYZ}
As a complex manifold,  $V_\tau = \bC^g/(\bZ^g +\tau \bZ^g) = (\bC^*)^g/\Gamma_\tau$. 
The complex coordinates on $\bC^g$ are $(z_1,\ldots,z_g)$ where $z_j= \phi_j -\bi \xi_j = -\bi(\xi_j+\bi\phi_j)$, and 
the complex coordinates on $(\bC^*)^g$ are $(x_1,\ldots, x_g)$ where $x_j = e^{2\pi z_j} = e^{2\pi(\xi_j +\bi \phi_j)}$. The complex structure is given by 
$$
J \frac{\partial}{\partial \xi_j} =\frac{\partial}{\partial\phi_j},
\quad J\frac{\partial}{\partial \phi_j} = -\frac{\partial}{\partial \xi_j}.
$$
Recall that $r_k =\sum_{j=1}^g \Omega^{jk}\xi_j$. The symplectic form on $V_\tau$ is 
\begin{equation}
\omega_{V_\tau} = \frac{\bi}{2}\sum_{j,k=1}^g \Omega^{jk} dz_j\wedge d\bar{z}_k =
\sum_{j,k=1}^g \Omega^{jk} d\xi_j \wedge d\phi_k = \sum_{k=1}^g dr_k\wedge d\phi_k.
\end{equation}
So $(r_j,\phi_j)$ are Darboux coordinates on the symplectic manifold 
$(V_\tau, \omega_{V_\tau})$. Let $g_V$ be the Riemannian metric
determined by the symplectic from $\omega_{V_\tau}$ and the complex structure $J$:
$$
g_V(v,w) =\omega_{V_\tau}(v, Jw).
$$
Then 
\begin{equation}
g_V = \sum_{j,k=1}^g \Omega^{jk} (d\xi_j d\xi_k + d\phi_j d\phi_k)
=\sum_{j,k=1}^g \Omega_{jk}dr_jdr_k + \sum_{j,k=1}^g \Omega^{jk} d\phi_jd\phi_k.
\end{equation}
The holomorphic volume form $dz_1\wedge \cdots  \wedge dz_g$ on $V_\tau$ is covariant constant with respect to this flat K\"{a}hler metric
with constant norm $\det(\Omega)^{1/2}$. Let 
\begin{equation}
\Phi:= (\det \Omega)^{-1/2} dz_1\wedge \cdots \wedge dz_g.
\end{equation}
Then $\Phi$ is a holomorphic and unitary frame of the canonical line bundle
$\Lambda^g T^*_{V_\tau} = \det(T^*_{V_\tau})$ of $V_\tau$.  
The fibers of $\pi_\tau^{\SYZ}: V_\tau \to T_\Omega$  are flat tori $(T_F =\bR^g/\bZ^g, \sum_{j,k=1}^g \Omega^{jk}d\phi_j d\phi_k)$ which are isometric to $T_{\Omega^{-1}}$. 
We have
$$
\Phi|_{T_F}= \det(\Omega)^{-1/2} d\phi_1\wedge \cdots \wedge d\phi_g = \mathrm{vol}_{T_F}
$$
where $\mathrm{vol}_{T_F}$ is the volume form of $T_F$. So 
$$
\mathrm{Re}\Phi|_{T_F} = \mathrm{vol}_{T_F},\quad
\mathrm{Im}\Phi|_{T_F} = 0.
$$
Therefore, the fibers of $\pi_\tau^{\SYZ}: V_\tau\to T_\Omega$ are special Lagrangian submanifolds of the Calabi-Yau manifold
$(V_\tau, \omega_{V_\tau}, \Phi)$.

As a complex manifold,  the mirror of the Calabi-Yau manifold $V_\tau$ is $\bT^{2g} = \bR^{2g}/\bZ^{2g}=\bC^g/(\bZ^g+\bi \bZ^g)$.  The complex coordinates on $\bC^g$ are $u_j = \theta_j -\bi r_j =-\bi (r_j +\bi \theta_j)$. 
The complex structure is given by 
$$
\widecheck{J} \frac{\partial}{\partial r_j} =\frac{\partial}{\partial\theta_j},
\quad \widecheck{J}\frac{\partial}{\partial \theta_j} = -\frac{\partial}{\partial r_j}.
$$

The real part $B =\mathrm{Re}\tau$
and the imaginary part $\Omega =\mathrm{Im}\tau$ of $\tau\in \cH_g$ determine a B-field $\omega_B$ (a closed 2-form) and a symplectic form $\omega_\Omega$ on $\bT^{2g}$ on $V_\tau$, respectively:
\begin{equation}
\omega_B :=\sum_{j,k=1}^g B_{jk}dr_j\wedge d\theta_k, \quad
\omega_\Omega := \sum_{j,k=1}^g \Omega_{jk} dr_j\wedge d\theta_k
= \sum_{k=1}^g d\xi_k \wedge d\theta_k.
\end{equation}
So $(\xi_j, \theta_j)$ are Darboux coordinates on the symplectic torus $(\bT^{2g}, \omega_\Omega)$. The projection  $(\xi, \theta)\mapsto \xi$ defines a group-valued moment map $\pi: \bT^{2g}\to T_\Omega$.

The complexified K\"{a}hler form  on $\bT^{2g}$ is 
\begin{equation}
\omega_\tau := \omega_B +\bi \omega_\Omega =
\sum_{j,k=1}^g (B_{jk} +\bi \Omega_{jk}) dr_j\wedge d\theta_k. 
\end{equation}
Let $g_\Omega$ be the Riemannian metric on $\bT^{2g}$
determined by the symplectic from $\omega_\Omega$ and the complex structure $\widecheck{J}$:
$$
g_\Omega(v,w) =\omega_\Omega(v, \widecheck{J}w).
$$
Then 
\begin{equation}
g_\Omega = \sum_{j,k=1}^g\Omega_{jk}  (dr_j dr_k +  d\theta_j d\theta_k)    = \sum_{j,k=1}^g\Omega^{jk} d\xi_j d\xi_k + \sum_{j,k=1}^g\Omega_{jk}d\theta_j d\theta_k. 
\end{equation}
The holomorphic volume form $du_1\wedge \cdots \wedge du_g$ on $\bT^{2g}$ is covariant constant with respect to this flat K\"{a}hler metric with constant norm $\det(\Omega)^{-1/2}$. Let 
\begin{equation}
\widecheck{\Phi} : = \det(\Omega)^{1/2} du_1\wedge \cdots \wedge du_g.
\end{equation}
Then $\widecheck{\Phi}$ is a holomorphic and unitary frame of the canonical line bundle $\det(T^*\bT^{2g})$  of  $(\bT^{2g}, \widecheck{J})$. 
The fibers of 
\begin{equation}\label{eqn:SYZ-dual}
    \widecheck{\pi}^{\SYZ}_\tau:\bT^{2g}\to T_\Omega
\end{equation} 
sending $(\xi,\theta)$ to $\xi$ are flat tori $(\widecheck{T}_F =\bR^g/\bZ^g, \sum_{j,k=1}^g 
\Omega_{jk} d\theta_j d\theta_k)$ which are isometric to $T_\Omega$. We have
$$
\widecheck{\Phi}|_{\widecheck{T}_F}= \det(\Omega)^{1/2} d\theta_1\wedge \cdots \wedge d\theta_g = \mathrm{vol}_{\widecheck{T}_F}
$$
where $\mathrm{vol}_{T^\vee_F}$ is the volume form of $\widecheck{T}_F$.  So 
$$
\mathrm{Re}\widecheck{\Phi}|_{\widecheck{T}_F} = \mathrm{vol}_{\widecheck{T}_F},\quad
\mathrm{Im}\widecheck{\Phi}|_{\widecheck{T}_F}=0.
$$
Therefore, the fibers of $\widecheck{\pi}^{\SYZ}_\tau: \bT^{2g} \to T_\Omega$ are special Lagrangian submanifolds of the Calabi-Yau manifold
$(\bT^{2g}, \omega_{\Omega}, \widecheck{\Phi})$. The special Lagrangian torus fibration $\widecheck{\pi}^{\SYZ}_\tau: \bT^{2g}\to T_\Omega$ is the dual of the special Lagrangian 
torus fibration $\pi^{\SYZ}_\tau: V_\tau\to T_\Omega$.
In Section \ref{sec:SYZ}, we will reconstruct the torus fibration
$\pi^{\SYZ}_\tau:V_\tau \to T_\Omega$ and the complex structure on $V_\tau$
from the torus fibration $\widecheck{\pi}^{\SYZ}_\tau: \bT^{2g}\to T_\Omega$
and the complexified symplectic structure on 
$\bT^{2g}$.
The Calabi-Yau manifold
$(\bT^{2g}, \omega_\tau, \widecheck{\Phi})$ is the SYZ mirror
of the Calabi-Yau manifold $(V_\tau, \omega_{V_\tau}, \Phi)$ \cite{SYZ, leung_without_corrections}. As we vary $\tau\in \cH_g$, we obtain the following mirror family 
\begin{equation}\label{eqn:universal-SYZ-H-mirror}
\xymatrix{
\widecheck{V}_{\cH_g}\ar[r]^{\widecheck{\pi}^{\SYZ}_{\cH_g}} & \cT_{\cH_g}\ar[r]^{p_{\cH_g}} & \cH_g.
}
\end{equation}
Let $P_g(\bZ)$ act on $\widecheck{V}_{\cH_g}$ by 
\begin{equation}\label{eqn:P_g(Z)-action-on-mirror}
\begin{bmatrix} A & B \\ 0 & (A^T)^{-1} \end{bmatrix}
\circ (\tau, \xi, \theta) = \left( (A \tau + B) A^T, A\xi, (A^T)^{-1}\theta\right).
\end{equation}
Then $\widecheck{\pi}^{\SYZ}_{\cH_g}$ is $P_g(\bZ)$-equivariant.  The quotient of \eqref{eqn:universal-SYZ-H-mirror} by the above $P_g(\bZ)$-action is 
\begin{equation}\label{eqn:universl-SYZ-mirror}
\xymatrix{
\widecheck{V}_{\cA^F_g}\ar[r]^{\widecheck{\pi}^{\SYZ}_{\cA^F_g}} & \cT_{\cA^F_g}\ar[r]^{p_{\cA^F_g}} & \cA^F_g. 
}
\end{equation}

\subsection{Mirror symmetry of isomorphisms}
\subsubsection{Categorical equivalences induced by isomorphisms of the abelian variety}

Recall that  the moduli of $g$-dimensional ppavs is $\cA_g =[\cH_g/\Sp(2g,\bZ) ]$
where the $\Sp(2g;\bZ)$-action is given by 
$$
  \phi \circ \tau =(A\tau+B)(C\tau +D)^{-1}, \quad
  \phi = \begin{bmatrix} A & B\\ C & D \end{bmatrix} \in \Sp(2g,\bZ),
  \tau\in \cH_g.
$$
There is an isomorphism 
$$
\phi: V_\tau \stackrel{\cong}{\lra} V_{\phi\circ \tau},
\quad [z] \mapsto [ \left((C\tau +D)^T\right)^{-1}z]
$$
of abelian varieties, which induces a derived equivalence
 \begin{equation}\label{eqn:Coh} 
\phi^*: D^b \Coh ( V_{\phi\circ \tau}) \stackrel{\cong}{\lra} D^b \Coh(V_\tau)
\end{equation}
sending a coherent sheaf $\cF$ on $V_{\phi\circ \tau}$ to its pullback $\phi^*\cF$
on $V_\tau$. In particular, if
$\phi\circ \tau = \tau$ then we obtain an autoequivalence
$$
\phi^*: D^b \Coh(V_\tau)\stackrel{\cong}{\lra} D^b\Coh(V_\tau).
$$

By homological mirror symmetry for principally polarized abelian varieties, we should also have the following equivalence of Fukaya categories
\begin{equation}\label{eq: Fuk auto}
H^0 (D^\pi \Fuk_{\aff}(\bT^{2g}, \omega_{\phi\circ \tau})) \stackrel{\cong}{\lra} H^0 (D^\pi \Fuk_{\aff}(\bT^{2g},\omega_{\tau}))
\end{equation}
for $\Fuk_{\aff}$ defined in Subsection \ref{sec:A-branes}. The equivalence between derived categories of coherent sheaves
in Equation \eqref{eqn:Coh} is given by pullback under
the isomorphism  $\phi: V_\tau\to V_{\phi\circ \tau}$ of abelian varieties, but the equivalence between the Fukaya categories in Equation \eqref{eq: Fuk auto} is not obvious without homological mirror symmetry since for
general $\phi \in \Sp(2g;\bZ)$,  $(\bT^{2g},\omega_\tau)$ and $(\bT^{2g},\omega_{\phi\circ \omega})$ are not symplectomorphic, so the equivalence is not induced by symplectomorphisms.

By \cite[Proposition A5, p210]{MumfordTheta}, $\Sp(2g;\bZ)$ is generated by $P_g(\bZ)$ defined in Equation \eqref{def: Siegel parabolic}, which acts on $\cH_g$ by $\tau\mapsto \widetilde \tau=(A\tau+B)A^T$ where $A\in \GL(g,\bZ)$ and $B\in M_g(\bZ)$ (see Equation \eqref{eq: P_g(Z) acts on tau}), and the element 
$$
\begin{bmatrix} 0 & -I_g\\ I_g & 0 \end{bmatrix}, 
$$ which acts on $\cH_g$ by $\tau\mapsto \widetilde \tau = -\tau^{-1}$ (in particular, if $\tau=\bi\Omega$, then $\widetilde \tau=\bi \Omega^{-1}$).  

For the action by $P_g(\bZ)$, where $\widetilde \tau=(A\tau+B)A^T$, as mentioned in Section \ref{sec:trop_ppav}, geometrically this is the  action that preserves the fiber of the SYZ fibration $\pi^{\SYZ}$  and acts on the base $T_\Omega=\bR^g/\Omega\bZ^g$ by $\Omega\mapsto \widetilde \Omega=A\Omega A^T$ (according to Equation \eqref{eq: P_g(Z) acts on tau}).
Note that $\omega_\Omega$ and $\omega_{\widetilde \Omega}$ are the real symplectic forms.  So  $A\Omega A^T=\Omega$ is the condition for the tori $(\bT^{2g}, \omega_\Omega)$ and $(\bT^{2g}, \omega_{\widetilde \Omega})$ to be symplectomorphic and the equivalence of the Fukaya category is induced by this symplectomorphism.   

As for the latter case where $\widetilde \tau=-\tau^{-1}$,  \cite[Section 6]{Qi22} offered a geometric interpretation of the relationship between $(\bT^{2g},\omega_\tau)$ and $(\bT^{2g},\omega_{\widetilde\tau})$, that their respective twisted doubling tori are the same up to a B-field twist.

\subsubsection{Symplectomorphisms of $(\bT^{2g}, \omega_\Omega)$}
In this subsection, we discuss the symplectomorphisms of $(\bT^{2g},\omega_\Omega)$, which would induce autoequivalences of the Fukaya category (when the $B$-field is zero), 
and by homological mirror symmetry implies equivalences of the derived categories. 
\begin{lemma}
Let $\bT^{n} =\bR^{n} /\bZ^{n}$ be a  $n$-dimensional torus, where $n$ is any positive integer. We have a short exact sequence of groups
$$
1\to \Diff_0(\bT^{n}) \to \Diff(\bT^{n}) \to  \GL(n,\bZ)\to 1,
$$
where the map $\Diff(\bT^n) \to \GL(n,\bZ)$ sends $\phi:\bT^n\to \bT^n$ to $\phi_*: H_1(\bT^n;\bZ)=\bZ^n \to H_1(\bT^n;\bZ)$.
So 
$$
MCG(\bT^n):= \pi_0(\Diff(\bT^n))\to \GL(n,\bZ)
$$
is an isomorphism.
\end{lemma}
\begin{proof} Let  $\pi:\bR^n \to \bT^n$ be the projection, which is also the universal covering map. Given any 
$A\in \GL(n,\bZ)$ the linear map $\bR^n\to\bR^n$ given by $x\mapsto Ax$ descends to a diffeomorphism
$\phi_A: \bT^n \to \bT^n $ such that $(\phi_A)_* = A: H_1(\bT^n;\bZ)\to H_1(\bT^n;\bZ)$. So the map
$\Diff(\bT^n) \to \GL(n,\bZ)$ is surjective. 

It remains to show that if $\phi: \bT^n\to \bT^n$ is a diffeomorphism and $\phi_*: H_1(\bT^n;\bZ)\to H_1(\bT^n;\bZ)$ is the identity map, then 
$\phi \in \Diff_0(\bT^n)$. By composing with a translation which is in $\Diff_0(\bT^n)$, we may assume that  $\phi(x_0)=x_0$, where $x_0=\pi(0)$.
Then $\phi_*:\pi_1(\bT^n,x_0)\to \pi_1(\bT^n,x_0)$ is the identity map. The map $\phi\circ \pi:\bR^n \to \bT^n$ has a unique lifting 
$\tilde{\phi}:\bR^n\to \bR^n$ such that $\tilde{\phi}\circ \pi = \phi\circ \pi$ and $\tilde{\phi}(0)=0$. The condition that
$\phi_*:\pi_1(\bT^n)\to \pi_1(\bT^n)$ is the identity map implies
$\tilde{\phi}(x+n) =\tilde{\phi}(x)+n$ for any $x\in \bR^n$. Define
$$
h: [0,1]\times \bR^n \to \bR^n,  (t,x)\mapsto  (1-t)x + t\tilde{\phi}(x)
$$
The for any $t\in [0,1]$, $h_t: \bR^n\to \bR^n$ descends to a smooth map $h_t:\bT^n\to \bT^n$ such that $h_t(x_0)=x_0$. By perturbing $h$ we can make
all $h_t$ diffeomorphisms without changing $h_0 = id$ and $h_1 =\phi$. So $\phi\in\Diff(\bT^n)$.
\end{proof}

Now let us consider the symplectic torus $\left(\bT^{2g}, \omega_\Omega = \sum_{j,k=1}^g \Omega_{jk} dr_j\wedge d\theta_k\right)$ where $(\Omega_{jk})$ is a real symmetric positive definite $g\times g$ matrix. 
Suppose that $\phi\in \Symp(\bT^{2g},\omega_\Omega)$ is a symplectomorphism. Then
$$
\phi^* : H^1(\bT^{2g};\bZ)\to H^1(\bT^{2g};\bZ) 
$$
is given by 
$$
\phi^*[dr_j] =  \sum_{k=1}^g A_{jk} [dr_k] + \sum_{k=1}^g B_{jk} [d\theta_k],\quad
\phi^*[d\theta_j]= \sum_{k=1}^g C_{jk} [dr_k] + \sum_{k=1}^g  D_{jk} [d\theta_k],
$$
where
$$
\begin{bmatrix} A & B\\ C& D \end{bmatrix} \in \GL(2g,\bZ).
$$
The condition $\phi^*\omega_\Omega=\omega_\Omega$ implies $\phi^*[\omega_\Omega]=[\omega_\Omega]\in H^2(\bT^{2g};\bR)$. Therefore, 
\begin{equation}
A^T\Omega D - C^T \Omega B = \Omega,\quad 
A^T\Omega C = C^T \Omega A, \quad B^T \Omega D = D^T \Omega B.
\end{equation}
Let $\Sp(\Omega, \bZ)$ be the subgroup of $\GL(2g, \bZ)$ satisfying the above equalities. In particular, $\Sp(J_g;\bZ)=\Sp(2g;\bZ)$. 
Note that
$$
\left\{  \begin{bmatrix} aI_g & bI_g\\ cI_g & dI_g \end{bmatrix} \Big|\   ad-bc=1 \right\} \cong \mathrm{SL}(2,\bZ)
$$
is contained in $\Sp(\Omega;\bZ)$ for any $ \Omega \in \cH_g^{\trop}$. 
Given any
$$
M= \begin{bmatrix} A & B\\ C& D \end{bmatrix} \in \Sp(\Omega, \bZ)
$$
there is a unique linear isomorphism
$\tilde{\phi}:\bR^{2g}\to \bR^{2g}$ such that 
$$
\tilde{\phi}^*dr_j  =  \sum_{k=1}^g A_{jk} dr_k + \sum_{k=1}^g B_{jk} d\theta_k,\quad
\tilde{\phi}^*d\theta_j = \sum_{k=1}^g C_{jk}  dr_k + \sum_{k=1}^g D_{jk} d\theta_k. 
$$
Then $\tilde{\phi}:\bR^{2g}\to\bR^{2g}$ descends to a symplectomorphism $\phi: (\bT^{2g},\omega)\to (\bT^{2g},\omega)$ such 
that $\phi_* = M \in \Sp(\Omega;\bZ) \subset \GL(2g;\bZ)$. Therefore, we have a surjective group homomorphism
$$
\Symp(\bT^{2g},\omega_\Omega) \to \Sp(\Omega;\bZ)
$$
and the following proposition.
\begin{proposition}
We have the following short exact sequence of groups.
\begin{equation}
1\to \Symp_0(\bT^{2g},\omega_\Omega) \to \Symp(\bT^{2g},\omega_\Omega) \to \Sp(\Omega;\bZ) \to 1.
\end{equation}
\begin{equation}
1\to \Ham(\bT^{2g},\omega_\Omega)\to \Symp_0(\bT^{2g},\omega_\Omega) \to \bT^{2g}\to 1.
\end{equation}
\end{proposition}

\section{Category of coherent sheaves on \texorpdfstring{$V_\tau$}{V tau}}

\subsection{Line bundles and their moduli} \label{sec:picard} 

In this subsection (Section \ref{sec:picard})
and the next two subsections (Section \ref{sec:poincare} and Section \ref{sec:shift}), we fix $\tau\in \cH_g$ and let $\cL$ denote the holomorphic line bundle $\cL_\tau$ over $V_\tau$. 

Let $\Pic(V_\tau)$ be the Picard group of $V_\tau$ which parametrizes isomorphism classes of line bundles on $V_\tau$. 
The first Chern class defines a map
\begin{equation}\label{eqn:chern}
c_1: \Pic(V_\tau) \lra H^2(V_\tau;\bZ)\cong \bZ^{2g}  \quad L\mapsto c_1(L)
\end{equation} 
which is a group homomorphism: $c_1(L_1\otimes L_2) = c_1(L_1) + c_1(L_2)$. In particular, for any $\sk\in \bZ$, $c_1(\cL^{\otimes \sk}) =\sk \omega_{V_\tau}$. 

The image of \eqref{eqn:chern}  is  $\NS(V_\tau)$, the N\'{e}ron-Severi group of $V_\tau$. We have
$$
\NS(V_\tau) = H^2(V_\tau;\bZ)\cap H^{1,1}(V_\tau).
$$ 
The rank  $\rho(V_\tau)$ of $\NS(V_\tau)$ is called the Picard number of $V_\tau$. 
For generic $\tau \in \cH_g$, $\rho(V_\tau)=1$ and  $\NS(V_\tau) = \bZ \omega_{V_\tau}$.  

Let $\Pic^0(V_\tau)$ be the kernel of \eqref{eqn:chern}. Then $\Pic^0(V_\tau)$  is isomorphic to the dual abelian variety of $V_\tau$.  We have a short exact sequence of abelian groups
$$
1 \to \Pic^0(V_\tau) \to \Pic(V_\tau) \to \NS(V_\tau) \to 0. 
$$

The abelian variety $V_\tau\cong V_\tau^+ = \bC^g/(\bZ^g+\tau \bZ^g)$ is an additive group. Given any
$z  \in \bC^g$, let $[z] \in V_\tau^+$ denote its equivalence class. Let
\begin{equation}\label{eq: shift}
\top_{[z]}: V^+_\tau\lra V^+_\tau, \quad [u]\mapsto [u+z] = [u]+[z]
\end{equation}
be translation by $[z]$.  Composing with the exponential function, we can define 
\begin{equation} \label{eq: Lu}
\phi_{\cL}: V_\tau \to \Pic^0(V_\tau),\quad  [z] \mapsto  \bbL_{[z]}:= \top_{[z]}^*\cL\otimes \cL^{-1}. 
\end{equation}
The condition that $\cL$ is a {\em principal} polarization on $V_\tau$, that is, $c_1(\cL)=[\omega_{V_\tau}]$, implies that $\phi_{\cL}$ is an isomorphism of abelian varieties. In particular  
$$
\bbL_{ [z]+[z'] } = \bbL_{[z]}\otimes \bbL_{[z']}. 
$$
For any $z\in \bC^g$ and $\sk \in \bZ$, 
\begin{equation}\label{semi-homogeneous}
\top_{[z]}^*(\cL^{\otimes \sk}) = (\top_{[z]}^* \cL)^{\otimes \sk} =  \left(\cL\otimes\bbL_{[z]}\right)^{\otimes \sk} =  \cL^{\otimes \sk}\otimes \bbL_{[\sk z]}. 
\end{equation}

\subsection{The Poincar\'{e} line bundle \texorpdfstring{$\cP$}{P} and universal line bundles \texorpdfstring{$\cP_{\sk}$}{P k}} \label{sec:poincare}
Let $\cP \lra V_\tau \times \Pic^0(V_\tau)$ be the Poincar\'{e} line bundle, characterized by 
\begin{enumerate}[(i)]
\item $\cP\big|_{V_\tau \times \{ L\} } =L$ for every $L\in \Pic^0(V_\tau)$, and 
\item  $\cP\big|_{ \{ [0] \} \times \Pic^0(V_\tau) }$ is trivial and equals $\cO_{V_\tau}$. 
\end{enumerate}
Then 
$$
\left(\mathrm{id}_{V_\tau} \times \phi_{\cL}\right)^* \cP \cong  m^*\cL \otimes p_1^* \cL^{-1},  
$$
where $p_1: V_\tau\times V_\tau \to V_\tau$  is the projection to the first factor,  and $m:V_\tau\times V_\tau \to V_\tau$ is
the addition: $m([z],[z']) = [z]+[z'] =[z+z']$. 

Let $\Pic^{\bZ \omega_{V_\tau}}(V_\tau)$ denote the subgroup of $\Pic(V_\tau)$ generated by 
$\Pic^0(V_\tau)$ and $\cL$; for generic $\tau$,  $\Pic^{\bZ \omega_{V_\tau}} (V_\tau)=\Pic(V_\tau)$. 
If $L$ is in  $\Pic^{\bZ\omega_{V_\tau}}(V_\tau)$ then  there exist a unique $\sk\in \bZ$ and a unique $[z]\in V_\tau$ such that $L\cong \cL^{\otimes \sk}\otimes \bbL_{[z]}$.  
Write $z = d+ \tau c $, where $d, c \in \bR^g$.

Given any $\sk\in \bZ$, let $\Pic^{\sk}(V_\tau)\subset \Pic^{\bZ\omega_\tau}(V_\tau)$ be the connected component corresponding to line bundles $L$ with 
$c_1(L) = \sk \omega_{V_\tau}$. Tensoring by 
$\cL^{\sk}$ defines a map $t_{\sk}: \Pic^0(V_\tau)\to \Pic^{\sk}(V_\tau)$ which is an isomorphism of smooth varieties. Let $p_1: V_\tau\times \Pic^{\sk}(V_\tau)\to V_\tau$ be the projection to the first factor, and let
$\cP_{\sk}$ be the line bundle on $V_\tau\times \Pic^{\sk}(V_\tau)$ defined by 
\begin{equation}\label{eqn:universal-line-bundle}
\cP_{\sk}:=  (t_{\sk}^{-1})^*\cP \otimes p_1^*\cL^{\otimes \sk}. 
\end{equation}
Then $\cP_{\sk}\big|_{V_\tau \times \{ L\} } =L$ for every $L\in \Pic^{\sk}(V_\tau)$, and $\cP\big|_{ \{ [0] \} \times \Pic^0(V_\tau) }$ is trivial, so 
$\cP_{\sk}\to V_\tau\times \Pic^{\sk}(V_\tau)$ is a universal line bundle. The composition 
$$
\phi_{\cL}^{\sk}:= t_{\sk}\circ \phi_\cL: V_\tau \to \Pic^{\sk}(V_\tau), \quad [z]\mapsto \cL^{\otimes \sk}\otimes \bbL_{[z]}
$$
is an isomorphism of smooth varieties, and
$$
(\mathrm{id}_{V_\tau}\times \phi_{\cL}^{\sk})^*\cP_{\sk} \cong m^*\cL \otimes p_1^* \cL^{\sk-1}.
$$

\subsection{Shifted theta functions and cohomology of line bundles} \label{sec:shift}
In this subsection, we introduce the shifted theta functions corresponding to sections of line bundles on $V_\tau$. 

 A function $f: (\bC^*)^g \to \bC$ defines a section of the line bundle $\cL^{\otimes \sk} \otimes \bbL_{[ d+ \tau c ]}$  on $V_\tau = (\bC^*)^g/\Gamma_\tau$ if and only if
 it satisfies the following quasi-periodicity property: 
 \begin{equation}\label{eqn:shift} 
f(x e^{2\pi \bi ( a + \tau b)}) =    e^{- \sk \bi\pi b^T \tau b} x^{-\sk  b} e^{2\pi \bi (c^T a - d^T b)} f(x) \quad \forall a,b\in \bZ^g. 
\end{equation} 

For $c, d\in \bR^g$, shifted theta functions are defined by 
\begin{equation}
\begin{split}
 \vartheta[ c,  d](\tau,x=e^{2\pi \bi z}) & :=e^{\bi\pi c^T \tau c+2\pi \bi  c^T (z+d)}\vartheta(\tau, e^{2\pi \bi(z+ d + \tau c)})\\
 & \ = \sum_{n\in \bZ^g} e^{\pi \bi  (n+c)^T \tau (n+ c) + 2\pi \bi (n+ c)^T(z+ d)}\\
 & \ =\sum_{n\in \bZ^g} e^{\pi \bi (n+c)^T \tau (n+ c) + 2\pi \bi(n+ c)^T d} x^{n+ c}.\\
\end{split}
 \end{equation}
An integral change to the shifts changes the shifted theta functions in the following way 
\begin{equation} \label{eqn:integral-shift} 
\vartheta[ c+c', d+d'](\tau, x)=e^{2\pi \bi c^T d'}\vartheta[c,d](\tau, x), \quad \text{ for } c, d\in \bR^g, \ c', d'\in \bZ^g.
\end{equation}
Note that it preserves the theta function when $c^T d'$ is an integer.  It is straightforward to check that 
$$
\vartheta[c, d](\tau, xe^{2\pi \bi(a+\tau b)})
=e^{2\pi \bi \left(c^T a- d^T b\right)}e^{-\pi \bi b^T \tau b} x^{-b}\vartheta[c, d](\tau, x),\quad \text{ for } c,d \in \bR^g, \  a,b\in \bZ^g.
$$
Therefore, the following function defines a section of $\cL\otimes \bbL_{[d+ \tau c]} = \top_{[d+ \tau c]}^*\cL$: 
$$
(\bC^*)^g\to \bC, \quad \vartheta[c,d](\tau,x).
$$
 
In particular, given a positive integer $\sk$,  $c, d\in \bR^g$, and $\lambda \in I_{g,\sk}:= \{0,1,\ldots,\sk-1\}^g$, we define
\begin{equation}\label{eq: shifted theta section}
f_{\tau, \sk, d+ \tau c, \lambda}: (\bC^*)^g \to \bC,\quad
x\mapsto
f_{\tau,\sk, d+ \tau c,\lambda}(x) := \vartheta\Big[\frac{c+ \lambda}{\sk}, d \Big](\sk\tau, x^\sk).
\end{equation}
Then  for $a,b\in \bZ^g$, 
\begin{eqnarray*} 
f_{\tau,\sk, d+ \tau c,\lambda}(x e^{2\pi \bi(a+\tau b)}) &=&  \vartheta\Big[\frac{c+ \lambda}{\sk}, d \Big](\sk\tau, x^\sk e^{2\pi \bi (\sk a + (\sk \tau) b)}) \\
&=&  e^{2\pi \bi \left(c^T a- d^T b\right)}e^{- \pi \bi b^T \sk \tau b} x^{-\sk b} \vartheta\Big[\frac{c+ \lambda}{\sk}, d\Big](\sk\tau, x^\sk) \\
&=&  e^{2\pi \bi \left(c^T a- d^T b\right)}e^{- \pi \bi b^T \sk \tau b} x^{-\sk b} f_{\tau,\sk, d+ \tau c,\lambda}(x),
\end{eqnarray*} 
which is exactly the quasi-periodic property \eqref{eqn:shift}. Therefore, $f_{\tau,\sk, d+ \tau c,\lambda}$  defines a section 
$$
s_{\tau,\sk, d+\tau c,\lambda} \in H^0\left(V_\tau, \cL^{\sk}\otimes \bbL_{[ d+ \tau c]}\right).
$$

Indeed, we have the following theorem. 
\begin{theorem}\label{thm: basis sections}  
For any positive integer $\sk$ and any vector $z \in \bC^g$,  
$$
\big\{ s_{\tau,\sk, z,\lambda}: \lambda \in  I_{g,\sk}= \{0,1,\ldots,\sk-1\}^g \big\}
$$
is a basis of $H^0(V_\tau,\cL^{\otimes \sk}\otimes \bbL_{[z]} ) \cong \bC^{\sk^g}$, and
$$
H^w \left(V_\tau,\cL^{\otimes \sk}\otimes \bbL_{[z]}\right) =0 \quad\text{if } w>0.
$$
\end{theorem} 

By Serre duality and triviality of the canonical line bundle of 
$V_\tau$, there is a perfect pairing 
$$
H^w(V_\tau, L)\otimes H^{g-w}(V_\tau, L^{-1}) \to H^g(V_\tau, \cO_{V_\tau})
\cong \bC,  
$$
so Theorem \ref{thm: basis sections} determines the cohomology of any line bundle  $L$ such that $c_1(L) = \sk \omega$ for some non-zero integer $\sk$.

If $c_1(L)=0$ then there is a unique $[z]\in V_\tau$ such that $L=\bbL_{[z]}$, and the cohomology of $L$ is given by the following theorem. 
\begin{theorem} \label{thm: picard zero}
$$
H^w(V_\tau,\bbL_{[0]}) = H^w(V_\tau,\cO_{V_\tau}) \cong \Lambda^w \bC^g
= \bC^{\binom{g}{w}}. 
$$
If $[z]\neq [0]\in V_\tau$ then 
$$
H^w(V_\tau,\bbL_{[z]})=0 \quad \text{for all }w.
$$
\end{theorem}

For any two line bundles $L_1, L_2$ on $V$, 
\begin{equation}\label{eqn:Ext}
\Ext^w(L_1,L_2) = H^w(V_\tau, L_2\otimes L_1^{-1}). 
\end{equation} 
Combining Theorem \ref{thm: basis sections}, Theorem \ref{thm: picard zero}, Serre duality, and \eqref{eqn:Ext}, we obtain the following result.

\begin{theorem} \label{thm: Ext on V}
For integers $\sk, \sk'$ and any vectors $z,z'\in \bC^g$, 
\begin{eqnarray*}
&& \Ext^w(\cL^{\otimes \sk} \otimes\bbL_{[z]} ,\cL^{\otimes \sk'} \otimes \bbL_{[z']} ) = H^w(V_\tau,\cL^{ \otimes(\sk'-\sk)} \otimes \bbL_{[z'-z]})  \\
&= & \begin{cases}
\bC^{|\sk'-\sk|^g} & \textup{if $\sk' > \sk $ and $w=0$, or if $\sk' < \sk$ and $w=g$};\\
\bC^{\binom{g}{w}} & \textup{if $\sk= \sk' \in \bZ$ and $[z]=[z'] \in V_\tau$}; \\
0 & \textup{otherwise}. 
\end{cases}
\end{eqnarray*} 
\end{theorem}

\begin{definition} \label{C-s-dual}
For any positive integers $\sk', \sk''$ and $z', z''\in \bC^g$, we have a map
$$
H^0(V_\tau, \cL^{\sk'}\otimes \bbL_{[z']})\times
H^0(V_\tau, \cL^{\sk''}\otimes \bbL_{[z'']}) \to 
H^0(V_\tau, \cL^{\sk'+\sk''} \otimes \bbL_{[z'+z'']})
$$
given by $(s',s'')\mapsto s'\otimes s''$. We fix $\tau \in \cH_g$ and 
let $s_{\sk,z,\lambda} = s_{\tau,\sk,z,\lambda}$. Define
$$
C^{(\sk'+\sk'', z'+z'',\lambda)}_{(\sk',z',\lambda'), (\sk'',z'',\lambda'')} \in \bC
$$
by 
$$
s_{\sk',z',\lambda'}\otimes s_{\sk'',z'',\lambda''} = \sum_{\lambda\in I_{g,\sk'+\sk''}}
C_{(\sk',z',\lambda'), (\sk'',z'',\lambda'')}^{(\sk'+\sk'', z'+z'',\lambda)} s_{\sk'+\sk'',z'+z'',\lambda}.
$$

For any positive integer $\sk$ and any $z\in \bC^g$, the map
\begin{equation}\label{eq:B-model Serre}
H^0(V_\tau, \cL^{\sk}\otimes \bbL_{[z]})
\times H^g(V_\tau, \cL^{-\sk}\otimes \bbL_{[-z]})
\to H^g(V_\tau, \cO_{V_\tau})\cong H^{0,g}(V_\tau) 
=\bC d\bar{z}_1\wedge\cdots\wedge d\bar{z}_g\cong \bC
\end{equation}
is a perfect pairing. Let $\{ s^{\sk,z,\lambda}: \lambda\in I_{g,\sk}\}$ be the $\bC$-basis
of $H^g(V_\tau,\cL^{-\sk}\otimes \bbL_{[-z]})\cong \bC^{\sk^g}$ dual to 
the $\bC$-basis $\{s_{\sk,z,\lambda}:\lambda\in I_{g,\sk}\}$ of 
$H^0(V_\tau, \cL^{\sk}\otimes \bbL_{[z]})$ with respect to the above pairing. More explicitly, $s^{\sk,z,\lambda}$ is uniquely characterized by 
\begin{equation}
s_{\sk,z,\lambda'}\otimes s^{\sk,z,\lambda} = \delta_{\lambda'}^\lambda d\bar{z}_1\wedge \cdots \wedge d\bar{z}_g,
\end{equation}
where
$$
\delta_{\lambda'}^\lambda= 
\begin{cases}
1 ,& \lambda'=\lambda,\\
0, & \lambda'\neq \lambda.
\end{cases}
$$
\end{definition}

\begin{lemma}\label{lem:C_s_dual_calc} Let $\sk'$ and $\sk''$ be positive integers such that $\bk'<\bk''$, and let
$z',z''\in \bC^g$. The map
\begin{equation}\label{eq:dual product}
H^0(V_\tau, \cL^{\sk'}\otimes \bbL_{[z']})\times
H^g(V_\tau, \cL^{-\sk''}\otimes \bbL_{[-z'']})
\to H^g(V_\tau, \cL^{\sk'-\sk''}\otimes \bbL_{[z'-z'']})
\end{equation}
is given by 
\begin{equation}
s_{\sk',z',\lambda'}\otimes s^{\sk'',z'',\lambda''} =
\sum_{\lambda\in I_{g,\sk''-\sk'}} C^{(\sk'',z'',\lambda'')}_{(\sk',z',\lambda'), (\sk''-\sk', z''-z',\lambda)}
s^{\sk''-\sk', z''-z',\lambda}.
\end{equation}
\end{lemma}
\begin{proof} We define $\tC^{(\sk'',z'',\lambda'')}_{(\sk',z',\lambda'), (\sk''-\sk', z''-z',\lambda)}\in \bC$ by 
$$
s_{\sk',z',\lambda'}\otimes s^{\sk'',z'',\lambda''} =
\sum_{\lambda\in I_{g,\sk''-\sk'}} \tC^{(\sk'',z'',\lambda'')}_{(\sk',z',\lambda'), (\sk''-\sk', z''-z',\lambda)}
s^{\sk''-\sk', z''-z',\lambda}.
$$
Then 
$$
s_{\sk',z',\lambda'}\otimes s^{\sk'',z'',\lambda''}\otimes s_{\sk''-\sk', z''-z', \lambda}
= \tC^{(\sk'',z'',\lambda'')}_{(\sk',z',\lambda'), (\sk''-\sk', z''-z',\lambda)} d\bar{z}_1\wedge \cdots \wedge d\bar{z}_g.
$$
On the other hand,
\begin{eqnarray*}
&& s_{\sk',z',\lambda'}\otimes s^{\sk'',z'',\lambda''}\otimes s_{\sk''-\sk', z''-z', \lambda} \\
&=& \sum_{s\in I_{g, \sk''}} C^{(\sk'',z'',s)}_{(\sk',z',\lambda'), (\sk''-\sk', z''-z',\lambda)}
s_{\sk'',z'',s} \otimes s^{\sk'',z'',\lambda''} \\
&=& \sum_{s\in I_{g,\sk''}} C^{(\sk'',z'',s)}_{(\sk',z',\lambda'), (\sk''-\sk', z''-z',\lambda)} \delta_s^{\lambda''}
d\bar{z}_1\wedge \cdots \wedge d\bar{z}_g \\
&=& C^{(\sk'',z'',\lambda'')}_{(\sk',z',\lambda'), (\sk''-\sk', z''-z',\lambda)} d\bar{z}_1\wedge \cdots \wedge d\bar{z}_g.
\end{eqnarray*}
Therefore, 
$$
\tC^{(\sk'',z'',\lambda'')}_{(\sk',z',\lambda'), (\sk''-\sk', z''-z',\lambda)}
=C^{(\sk'',z'',\lambda'')}_{(\sk',z',\lambda'), (\sk''-\sk', z''-z',\lambda)}.
$$
\end{proof}
\begin{remark}\label{rmk:dual product Serre}
Lemma \ref{lem:C_s_dual_calc} shows that, as a consequence of Serre duality, the computation of the product in Equation \eqref{eq:dual product} can be reduced to that of 
\begin{equation}\label{eq:dual product Serre}
H^0(V_\tau, \cL^{\sk'}\otimes \bbL_{[z']})\times
H^0(V_\tau, \cL^{\sk''-\sk'}\otimes \bbL_{[z''-z']})
\to H^0(V_\tau, \cL^{\sk''}\otimes \bbL_{[z'']}).
\end{equation}
\end{remark}

The coefficients  $C^{(\sk'+\sk'', z'+z'',\lambda)}_{(\sk',z',\lambda'), (\sk'',z'',\lambda'')}$ 
in Definition \ref{C-s-dual} can be determined explicitly from the following multiplication formula.

\begin{proposition}[multiplication formula] \label{multiplication} 
Let $\sk',\sk''$ be positive integers, and let  $c',d', c'',d''\in \bR^g$. 
Then
\begin{equation}  \label{eqn:multiply-i} 
\begin{aligned}
& \vartheta\left[\frac{c'}{\sk'},d'\right](\sk' \tau, x') \vartheta\left[\frac{c''}{\sk''}, d'' \right](\sk''\tau, x'')  \\
=& \sum_{w\in I_{g, \sk'+\sk''}}  \vartheta\left[\frac{ \sk'\sk'' w + \sk''c'-\sk' c''}{\sk'\sk''(\sk'+\sk'')}, \sk'' d' -\sk'd''\right] \left( \sk'\sk''(\sk'+\sk'')\tau, (x')^{\sk''} (x'')^{-\sk'} \right) \\
& \hspace{1.5cm} \cdot \vartheta\left[\frac{ \sk' w +c'+c''}{\sk'+\sk''}, d'+d'' \right]((\sk'+\sk'')\tau, x' x'') .
\end{aligned}
\end{equation} 

 In particular, when $x'= x^{\sk'}$ and $x''=x^{\sk''}$, 
\begin{equation} \label{eqn:multiply-ii} 
\begin{aligned}
& \vartheta\left[\frac{c'}{\sk'},d' \right](\sk'\tau, x^{\sk'}) \vartheta\left[\frac{c''}{\sk''}, d''\right](\sk''\tau, x^{\sk''})  \\
=& \sum_{w\in I_{g, \sk'+\sk''}}\vartheta\left[\frac{ \sk'\sk'' w + \sk'' c'-\sk' c''}{\sk'\sk''(\sk'+\sk'')}, \sk'' d'-\sk'd''\right] ( \sk'\sk''(\sk'+\sk'')\tau, 1)  \\
 &\hspace{1.5cm}\vartheta\left[\frac{\sk' w+c' +c''}{\sk'+\sk''}, d'+d''\right]((\sk'+\sk'')\tau, x^{\sk'+\sk''}).
 \end{aligned} 
\end{equation}
\end{proposition}

Setting $d'=d''=0$ in Proposition \ref{multiplication} and noting that  $\vartheta[c+1,0](\tau,x) =\vartheta[c,0](\tau,x)$ by Equation \eqref{eqn:integral-shift},
we obtain the following corollary, which is
Proposition (6.4) in \cite[Chapter II]{MumfordTheta}. 
\begin{corollary}
Let $\sk',\sk''$ be positive integers, and let  $c',c'' \in \bR^g$. 
Then
\begin{eqnarray*}
&& \vartheta\left[\frac{c'}{\sk'},0 \right](\sk' \tau, x') \vartheta\left[\frac{c''}{\sk''}, 0 \right](\sk''\tau, x'')  \\
&=& \sum_{w \in  I_{g, \sk'+\sk''}}  \vartheta\left[\frac{ \sk'\sk'' w + \sk''c'-\sk' c''}{\sk'\sk''(\sk'+\sk'')}, 0\right] \left( \sk'\sk''(\sk'+\sk'')\tau, (x')^{\sk''} (x'')^{-\sk'} \right) 
\vartheta\left[\frac{ \sk' w +c'+c''}{\sk'+\sk''}, 0 \right]((\sk'+\sk'')\tau, x' x'').
\end{eqnarray*}
In particular, when  $x'= x^{\sk'}$ and $x''=x^{\sk''}$, \begin{eqnarray*}
&& \vartheta\left[\frac{c'}{\sk'},0 \right](\sk'\tau, x^{\sk'}) \vartheta\left[\frac{c''}{\sk''}, 0\right](\sk''\tau, x^{\sk''})  \\
&=& \sum_{w \in  I_{g, \sk'+\sk''} }  \vartheta\left[\frac{ \sk'\sk'' w + \sk'' c'-\sk' c''}{\sk'\sk''(\sk'+\sk'')},0\right] ( \sk'\sk''(\sk'+\sk'')\tau, 1)  
\vartheta\left[\frac{\sk' w+c' +c''}{\sk'+\sk''}, 0\right]((\sk'+\sk'')\tau, x^{\sk'+\sk''}).
 \end{eqnarray*}
\end{corollary}

The following lemma will be used in the proof of Proposition \ref{multiplication}. 
\begin{lemma}\label{n-n}
Let $\sk', \sk''$ be positive integers. 
Given any $n', n'' \in \bZ^g$ there exists a unique triple $(n, \tn, w)$, where $n,\tn \in \bZ^g$ 
and $w\in I_{g, \sk'+\sk''}$,  such that
$$
n'=  n + \sk'' \tn + w, \quad n'' =n  - \sk' \tn.  
$$
\end{lemma} 
\begin{proof}  Let $j \in \{1,\ldots,g \}$.  It suffices to prove that, for any $n'_j, n_j'' \in \bZ$ there exists a unique triple $(n_j, \tn_j, w_j)$, where $n_j, \tn_j, w_j \in \bZ$ and
$w_j \in [0, \sk'+\sk''-1]$, such that
$$
n'_j = n_j  + \sk'' \tn_j + w_j ,\quad n''_j = n_j - \sk' \tn_j.  
$$
Given any $n'_j, n''_j\in \bZ$, there exists a unique  $\tn_j \in \bZ$ and $w_j \in  \{0,1,\ldots,\sk'+\sk''-1\}$ such that
\begin{equation}\label{eq:remainder}
n'_j - n''_j  = (\sk'+\sk'')\tn_j  + w_j. 
\end{equation}
Let $n_j = n''_j + \sk'\tn_j  \in \bZ$. Then
$$
n'_j  = n_j + \sk''\tn_j  + w_j , \quad  n''_j = n_j  - \sk' \tn_j.  
$$
\end{proof}

\begin{proof}[Proof of Proposition \ref{multiplication}]
\begin{eqnarray*}
&& \vartheta\left[\frac{c'}{\sk'},d'\right](\sk'\tau, x') \vartheta\left[\frac{c''}{\sk''}, d''\right](\sk''\tau, x'')  \\
&=& \sum_{n',n''\in \bZ^g} \exp\Big( \bi \pi  (n'+\frac{c'}{\sk'})^T \sk'\tau(n'+\frac{c'}{\sk'}) + \bi\pi (n''+\frac{c''}{\sk''})^T \sk''\tau(n''+\frac{c''}{\sk''})\Big) \\
&& \hspace{1cm} \cdot  \exp\Big(2\pi \bi (n'+\frac{c'}{\sk'})^Td' +2\pi \bi (n''+\frac{c''}{\sk''})^T d''\Big)  \cdot (x')^{n'+ c'/\sk'} (x'')^{n''+ c''/\sk'' }  \\
&=& \sum_{w_1,\ldots, w_g=0}^{\sk'+\sk''-1} \sum_{n,\tn\in \bZ^g}
\exp\Big( \bi \pi (n +\sk''\tn + w +\frac{c'}{\sk'})^T \sk'\tau(n +\sk''\tn+ w +\frac{c'}{\sk'}) + \bi\pi (n-\sk'\tn +\frac{c''}{\sk''})^T \sk''\tau(n-\sk'\tn+\frac{c''}{\sk''}) \Big) \\
&&  \hspace{1.5cm} \cdot  \exp\Big(2\pi \bi (n + \sk''\tn+w +\frac{c'}{\sk'})^Td' + 2\pi \bi (n -\sk'\tn +\frac{c''}{\sk''})^T d'') \big)\Big)  
 \cdot (x')^{n +\sk''\tn +w + c'/\sk'} (x'')^{n-\sk'\tn + c''/\sk'}. 
\end{eqnarray*}
The last equality follows from  Lemma \ref{n-n}. We now rewrite each term in the above infinite sum.  We observe that
$$
n + \sk''\tn + w +\frac{c'}{\sk'} =\mu + \sk'' \nu, \quad n- \sk'\tn +\frac{c''}{\sk''} = \mu-\sk' \nu,
$$
where
$$
\mu =  n +\frac{\sk'w+ c'+c''}{\sk'+\sk''}, \quad
\nu =  \tn +\frac{w}{\sk'+\sk''} + \frac{c'}{\sk' (\sk'+\sk'')}- \frac{c''}{\sk''(\sk'+\sk'')}.
$$ 
We have
\begin{eqnarray*}
(\mu + \sk'' \nu)^T \sk'\tau (\mu +\sk''\nu) + (\mu-\sk'\nu) ^T \sk''\tau (\mu-\sk'\nu) 
&=& (\sk'+\sk'') \mu^T \tau \mu +  \sk'\sk''(\sk'+\sk'')  \nu^T \tau \nu ,\\
(\mu + \sk''\nu)^T d' + (\mu-\sk'\nu)^T d'' &=&  \mu^T(d'+d'') +\nu^T(\sk''d'-\sk'd''),\\
(x')^{\mu + \sk''\nu} (x'')^{\mu-\sk'\nu}  &=& (x'x'')^\mu \left( (x')^{\sk''} (x'')^{-\sk'}\right)^{\nu}  
\end{eqnarray*} 
Therefore, 
\begin{eqnarray*}
&& \vartheta\left[\frac{c'}{\sk'},d'\right](\sk'\tau, x') \vartheta\left[\frac{c''}{\sk''}, d''\right](\sk''\tau, x'') \\
&=& \sum_{w_1,\ldots, w_g=0}^{\sk'+\sk''-1} \sum_{n,\tn\in \bZ^g}
\exp\Big( \bi \pi (n +\frac{\sk'w+ c' +c''}{\sk'+\sk''})^T (\sk'+\sk'')\tau (n +\frac{\sk'w+ c' +c''}{\sk'+\sk''}) \Big) \\
&&\hspace{1cm} \cdot \exp\Big( \bi\pi (\tn+ \frac{ \sk'\sk''w  + \sk''c'-\sk' c''}{\sk'\sk''(\sk'+\sk'')})^T \sk' \sk'' (\sk'+\sk'') \tau
(\tn+\frac{ \sk'\sk'' w+ \sk''c'-\sk' c''}{\sk'\sk''(\sk'+\sk'')} )\big) \Big) \\
&& \hspace{1cm} \cdot  \exp\Big(2\pi \bi (n + \frac{\sk'w+ c'+c'' }{\sk'+\sk''})^T(d'+d'')  
+ 2\pi \bi (\tn+ \frac{ \sk'\sk'' w + \sk''c'-\sk' c''}{\sk'\sk''(\sk'+\sk'')} )^T(\sk''d'-\sk'd'')\Big)  \\
&& \hspace{1cm} \cdot  (x'x'')^{n +\frac{\sk'w+ c'+c''}{\sk'+\sk''}}  \left( (x')^{\sk''} (x'')^{-\sk'}\right)^{ \tn +\frac{w}{\sk'+\sk''} + \frac{c'}{\sk' (\sk'+\sk'')}- \frac{c''}{\sk''(\sk'+\sk'')}}  \\
&=& \sum_{w_1,\ldots, w_g = 0}^{\sk'+\sk''-1}  \vartheta\left[\frac{ \sk'\sk''w + \sk''c'-\sk' c''}{\sk'\sk''(\sk'+\sk'')}, \sk'' d' -\sk'd''\right] ( \sk'\sk''(\sk'+\sk'')\tau, (x')^{\sk''} (x'')^{-\sk'})  \\
&& \hspace{1.5cm} \cdot \vartheta\left[\frac{\sk'w+ c'+c''}{\sk'+\sk''}, d'+d'' \right]((\sk'+\sk'')\tau, x' x'') 
\end{eqnarray*}
\end{proof}

\subsection{Categories of B-branes and the ring of sections} \label{sec:categories-of-B-branes}
We recall from Section \ref{sec:B-branes}, categories of B-branes on the complex projective variety $V_\tau$. 
\begin{itemize}
\item Let $D^b \Coh(V_\tau)$ be the bounded derived category of coherent sheaves on $V_\tau$. It is equivalent to $\mathrm{Perf}(V_{\tau})$, the triangulated category of perfect complexes on $V_\tau$. 

\item Let $\cB_\tau$ be the full subcategory of $D^b\Coh(V_\tau)$ whose 
objects are of the form $L[j]$ where $L$ is an element in $\Pic^{\bZ\omega_{V_\tau}}(V_\tau)$. 
Recall that any element in $\Pic^{\bZ\omega_{V_\tau}}(V_\tau)$ is of the form
$$
\cL_{\sk,[z]} := \cL^{\otimes \sk} \otimes \bbL_{[z]}
$$
where $\sk\in \bZ$ and $[z]\in V_\tau$. We have
\begin{equation} \label{eqn:Hom-cB}
\Hom_{\cB_\tau}\left(\cL_{\sk,[z]}[j], \cL_{\sk',[z']}[j']\right)
=\Ext^*(\cL_{\sk,[z]}, \cL_{\sk',[z']})[j'-j]
=H^*(V_\tau, \cL_{\sk'-\sk, [z'-z]})[j'-j]
\end{equation}
which is a graded complex vector space given by Theorem \ref{thm: Ext on V}.

\item Let $L$ be a line bundle on $V_\tau$ such that $c_1(L)= \sk\omega_{V_\tau}$
for some positive integer $\sk$. Then $L$ is an ample line bundle on $V_\tau$. Let
$\cB_L$ denote the full subcategory of $\cB$ whose objects are $\{ L^{\sk}[j]: \sk, j\in \bZ\}$.  In particular, $\cB_{\cL}$ in this paper corresponds to 
$D^b_{\cL}\Coh(V_\tau)$ in \cite{Ca}. 
\end{itemize}

Note that the subcategories $\cB_L$ and $\cB_\tau$  of $D^b \Coh(V_\tau) =\mathrm{Perf}(V_\tau)$ are not triangulated categories.

Let $\cT$ be a triangulated category, and let $\cI$ be a full subcategory 
of $\cT$; we do not assume $\mathcal{I}$ is a triangulated category. 
Let $\langle \mathcal{I}\rangle$ denote the smallest triangulated subcategory
of $\cT$ which contains $\cI$ and is closed under direct summand.  By \cite[Theorem 4]{Orlov_generation}, 
$$
\langle \cB_L\rangle =\langle \cB_\tau \rangle = D^b\Coh(V_\tau).
$$

The {\em ring of sections} of an ample line bundle $L$ on a projective variety $Y$ is
the graded ring
$$
\bigoplus_{d\geq 0} H^0(Y, L^{\otimes d}).
$$
where $\deg s = d$ of $s\in H^0(Y,L^{\otimes d})$, and 
the product structure is given by the tensor product.
The projective variety $Y$ can be reconstructed
from the above ring: 
$$
Y = \Proj \left(\bigoplus_{d\geq 0} H^0(Y, L^{\otimes d}) \right) 
$$
In particular, the graded ring
\begin{equation}\label{eqn:coordinate-ring}
S_\tau := \bigoplus_{d\geq 0} H^0(V_\tau,\cL_\tau^{\otimes d})
\end{equation}
is described explicitly in Section \ref{sec:shift}. We have
\begin{equation}\label{eqn:Proj}
V_\tau = \Proj(S_\tau). 
\end{equation}

\subsection{Higher direct images} \label{sec:higher-direct-image}
For any integer $\sk$, let $\cP_{\sk}\to V_\tau \times \Pic^{\sk}(V_\tau)$ be the universal line bundle defined by \eqref{eqn:universal-line-bundle}. In particular, $\cP_0=\cP$ is the Poincar\'{e} line bundle. 

Given any  $\sk, \sk'\in \bZ$, and $1\leq i<j\leq 3$, we have the following projections $\pi_{ij}$:  
\begin{equation} \label{eqn:three-to-two}
\begin{aligned}
\pi_{12}:\ & V_\tau\times \Pic^{\sk}(V_\tau)\times \Pic^{\sk'}(V_\tau) \lra V_\tau\times\Pic^{\sk}(V_\tau), \\
\pi_{13}:\ & V_\tau\times \Pic^{\sk}(V_\tau)\times \Pic^{\sk'}(V_\tau) \lra
V_\tau\times\Pic^{\sk'}(V_\tau), \\
\pi_{23}:\ & V_\tau\times \Pic^{\sk}(V_\tau)\times \Pic^{\sk'}(V_\tau) \lra
\Pic^{\sk}(V_\tau) \times \Pic^{\sk'}(V_\tau),
\end{aligned}
\end{equation}
Let 
\begin{equation}\label{eqn:cE-tau-k-k}
\cE^w_{\tau, \bk, \bk'} := R^w\pi_{23 *} \left(\pi^*_{13}\cP_{\sk'} \otimes \pi^*_{12}\cP_{\sk}^{-1}\right).
\end{equation}
Then $\cE^w_{\tau, \bk, \bk'}$ is a coherent sheaf of $\cO_P$-modules on 
$P=\Pic^{\sk}(V_\tau)\times \Pic^{\sk'}(V_\tau)$ whose fiber over
$(L, L') \in \Pic^{\sk}(V_\tau)\times \Pic^{\sk'}(V_\tau)$ is
$$
H^w(V_\tau, L'\otimes L^{-1})= \Ext^w(L, L'). 
$$
Varying the pair $(L, L')$ in $\Pic^{\sk}(V_\tau)\times \Pic^{\sk'}(V_\tau)$,
we obtain the following family version of Theorem \ref{thm: Ext on V}.
\begin{theorem} \label{thm:E-tau-k-k}
\begin{enumerate}
\item If $\sk<\sk'$  then 
$\cE_{\tau, \sk,\sk'}^0$ is a locally free sheaf of rank $(\sk'-\sk)^g$ 
and $\cE_{\tau, \sk,\sk'}^w=0$ if $w>0$.
\item If $\sk >\sk'$  then 
$\cE_{\tau, \sk,\sk'}^g$ is a locally free sheaf of rank $(\sk-\sk')^g$ 
and $\cE_{\tau, \sk,\sk'}^w=0$ if $w\neq g$.
\item  $\cE_{\tau,\sk,\sk}^w$ is a torsion sheaf on $\Pic^{\sk}(V_\tau)\times \Pic^{\sk}(V_\tau)$ supported along the diagonal if $0\leq w\leq g$, and 
$\cE_{\tau,\sk,\sk}^w=0$ if $w>g$.
\end{enumerate}
\end{theorem}

\subsection{The universal family and universal Picard varieties} \label{sec:universal-picard}
We view $\cH_g$ as a complex manifold, equipped with the standard complex structure as an open subset of $S_g(\bC)\cong \bC^{g(g+1)/2}$. In this subsection, we fix $g$ and let $\cO_{\cH}$ be the sheaf of germs of {\em holomorphic} functions on $\cH=\cH_g$.

As we vary $ \tau \in \cH$, we obtain a holomorphic line bundle over the universal family:
\begin{equation}\label{eqn:LA-I}
\cL_{\cH} \lra V_{\cH} \stackrel{\pi_{\cH}}{\longrightarrow}  \cH
\end{equation}
The restriction of \eqref{eqn:LA-I} to a point $\tau\in \cH$ is $\mathcal{L}_\tau \longrightarrow V_\tau \longrightarrow \tau$.

For any integer $d\geq 0$, define 
$$
(S_{\cH})_d  := \pi_{\cH *}\left( \cL_{\cH}^{\otimes{d} } \right)
$$
which is a locally free sheaf of $\cO_{\cH}$-modules on $\cH$. Then 
\begin{equation}\label{eqn:coordinate-ring-H}
S_{\cH} := \bigoplus_{d\geq 0}(S_{\cH})_d
\end{equation}
is a quasi-coherent sheaf of $\cO_{\cH}$-modules which has the structure of a sheaf
of graded $\cO_{\cH}$-algebras. The fiber of $\cS$ over $\tau\in \cH$ is the graded ring $S_\tau$ in \eqref{eqn:coordinate-ring}, so \eqref{eqn:coordinate-ring-H} can be viewed as the global version of \eqref{eqn:coordinate-ring}. We have
\begin{equation}\label{eqn:Proj-H}
V_{\cH}=\Proj_{\cH}(\cS)
\end{equation}
which is the global version of \eqref{eqn:Proj}. 

For any $\sk\in \bZ$, let 
$$
\Pic^{\sk}(V_{\cH})\to \cH
$$
be the universal Picard variety, which is a holomorphic fibration where the fiber
over $\tau\in \cH$ is $\Pic^{\sk}(V_\tau)$. In other words, a point in $\Pic^{\sk}(V_{\cH})$ is
a pair $(\tau, L)$ where $\tau\in \cH$ and $L$ is a line bundle on $V_\tau$ with $c_1(L)=\sk\omega_{V_\tau}$, and the projection $\Pic^{\sk}(V_{\cH})\to \cH$ is given by 
$(\tau, L)\mapsto \tau$. The fiber product
$$
V_{\cH}\times_{\cH} \Pic^{\sk}(V_{\cH})\to \cH
$$
is a holomorphic fibration where the fiber over $\tau\in \cH$ is $V_\tau\times \Pic^{\sk}(V_\tau)$. Let 
\begin{equation} \label{eqn:universal-universal-line}
    \cP_{\sk}(\cH)\to V_{\cH}\times_{\cH} \Pic^{\sk}(V_{\cH})\to \cH 
\end{equation}
be the universal line bundle. The restriction of \eqref{eqn:universal-universal-line} to a point $\tau\in \cH$ is $\cP_{\sk} \to V_\tau \times \Pic^{\sk}(V_\tau)\to \tau$.

Given any  $\sk, \sk'\in \bZ$, and $1\leq i<j\leq 3$, we have the following projections:  
\begin{equation} \label{eqn:three-to-two-H}
\begin{aligned}
\pi^{\cH}_{12}:\ & V_\cH \times_\cH \Pic^{\sk}(V_\cH)\times_\cH \Pic^{\sk'}(V_\cH) \lra V_\cH\times_\cH \Pic^{\sk}(V_\cH), \\
\pi^{\cH}_{13}:\ & V_\cH\times_\cH \Pic^{\sk}(V_\cH)\times_\cH \Pic^{\sk'}(V_\cH) \lra
V_\cH \times_{\cH}\Pic^{\sk'}(V_\cH), \\
\pi^{\cH}_{23}:\ & V_\cH\times_\cH \Pic^{\sk}(V_\cH)\times_\cH \Pic^{\sk'}(V_\cH) \lra
\Pic^{\sk}(V_\cH) \times_\cH \Pic^{\sk'}(V_\cH).
\end{aligned}
\end{equation}
Let $S_{ij}$ and $T_{ij}$ be the source and target of the projection $\pi^{\cH}_{ij}$
in  Equation \eqref{eqn:three-to-two-H}, respectively. We have a commutative diagram:
$$
\xymatrix{
S_{ij} \ar[r]^{\pi^{\cH}_{ij}} \ar[dr] & T_{ij} \ar[d] \\
& \cH
}
$$
where $S_{ij}\to \cH$ and $T_{ij}\to \cH$ are holomorphic fibrations.
The restriction of the projection $\pi_{ij}^{\cH}: S_{ij}\lra T_{ij}$ in Equation \eqref{eqn:three-to-two-H} to the fibers $S_{ij,\tau}\cong V_\tau\times V_\tau\times V_\tau \lra T_{ij,\tau}\cong V_\tau\times V_\tau$ over $\tau\in \cH$ is the projection $\pi_{ij}$ in Equation \eqref{eqn:three-to-two}. 

A point in $\Pic^{\sk}(V_\cH)\times_{\cH} \Pic^{\sk'}(V_\cH)$ is a triple
$(\tau, L, L')$ where $\tau\in \cH$, $L\in \Pic^{\sk}(V_\tau)$, and $L' \in \Pic^{\sk'}(V_\tau)$. Let
\begin{equation}\label{eqn:cE-k-k}
\cE^w_{\cH,\bk, \bk'} := R^w\pi^\cH_{23 *} \left( \left(\pi^\cH_{13}\right)^*\cP_{\sk'}(\cH) \otimes 
\left(\pi^\cH_{12}\right)^*\cP_{\sk}(\cH)^{-1}\right).
\end{equation}
Then $\cE^w_{\cH,\bk, \bk'}$ is a coherent sheaf of $\cO_P$-modules on 
$P=\Pic^{\sk}(V_\cH)\times_{\cH} \Pic^{\sk'}(V_\cH)$ whose fiber over
$(\tau, L, L') \in \Pic^{\sk}(V_\cH)\times_{\cH} \Pic^{\sk'}(V_\cH)$ is
$$
H^w(V_\tau, L'\otimes L^{-1})= \Ext^w(L, L'). 
$$
The restriction of $\cE^w_{\cH, \bk, \bk'}$ to $\Pic^{\sk}(V_\tau)\times \Pic^{\sk'}(V_\tau)$, 
which is the fiber of $\Pic^{\sk}(V_\cH)\times_\cH \Pic^{\sk'}(V_\cH) \to \cH$
over a point $\tau\in \cH$,  is the sheaf $\cE^w_{\tau,\bk, \bk'}$ in Equation \eqref{eqn:cE-tau-k-k}.  Varying $\tau$ in $\cH$ in Theorem \ref{thm:E-tau-k-k}, or equivalently varying the triple $(\tau, L, L')$ in $\Pic^{\sk}(V_\cH)\times_{\cH} \Pic^{\sk'}(V_\cH)$ in Theorem \ref{thm: Ext on V}, we obtain the following Theorem \ref{thm:E-k-k}, which is the most global and universal family version of Theorem \ref{thm: Ext on V}. 
\begin{theorem} \label{thm:E-k-k}
\begin{enumerate}
\item If $\sk<\sk'$  then 
$\cE_{\cH, \sk,\sk'}^0$ is a locally free sheaf of rank $(\sk'-\sk)^g$ 
and $\cE_{\cH, \sk,\sk'}^w=0$ if $w>0$.
\item If $\sk >\sk'$  then 
$\cE_{\cH, \sk,\sk'}^g$ is a locally free sheaf of rank $(\sk-\sk')^g$ 
and $\cE_{\cH, \sk,\sk'}^w=0$ if $w\neq g$.
\item  $\cE_{\cH, \sk,\sk}^w$ is a torsion sheaf on $\Pic^{\sk}(V_\cH)\times_{\cH} \Pic^{\sk}(V_\cH)$ supported along the diagonal (consisting of points $(\tau, L, L)$ where
$\tau \in \cH$ and $L\in \Pic^{\sk}(V_\tau)$) if $0\leq w\leq g$, and 
$\cE_{\cH, \sk,\sk}^w=0$ if $w>g$.
\end{enumerate}
\end{theorem}

We next consider the $\Sp(2g,\bZ)$-action. 
Recall that $\Sp(2g,\bZ)$ acts on $\cH_g\times \bC^g$ by 
$$
\phi\cdot (\tau, z) 
=\left(\phi\circ \tau, ((C\tau +D)^T)^{-1}z\right)
$$
where 
$$
\phi=\begin{bmatrix}A & B\\ C & D\end{bmatrix} \in\Sp(2g,\bZ), \quad
\tau\in \cH_g, \quad
\phi\circ \tau = (A\tau+B)(C\tau+D)^{-1}, \quad 
z = \begin{bmatrix} z_1\\\vdots \\ z_g\end{bmatrix}.
$$
The universal family of $g$-dimensional ppavs is
$$
V_{\cA} = \left[
\left( \Sp(2g,\bZ)\ltimes \bZ^{2g}\right)\backslash
\left(\cH\times \bC^g\right) \right]\lra
\cA = [\Sp(2g,\bZ)\backslash \cH].
$$
In particular, for any $\tau\in \cH$ and any $\phi\in \Sp(2g,\bZ)$, we have an isomorphism
$$
\phi: V_\tau \lra V_{\phi\circ \tau}, \quad
[z]\mapsto [\left((C\tau+D)^T\right)^{-1}z]
$$
of abelian varieties. For any $\sk\in \bZ$, we have
an isomorphism
$$
\phi^*: \Pic^{\sk}(V_{\phi\circ \tau})\lra 
\Pic^{\sk}(V_\tau), \quad L\mapsto \phi^*L. 
$$
of complex projective varieties. Let $\Sp(2g,\bZ)$ act on 
$\Pic^{\sk}(V_{\cH})\times_{\cH} \Pic^{\sk'}(V_{\cH})$ 
$$
\phi\cdot (\tau, L, L') = (\phi\circ \tau, (\phi^{-1})^*L, (\phi^{-1})^*L').
$$
Note that $((\phi_1\circ \phi_2)^{-1})^* = (\phi_2^{-1}\circ \phi_1^{-1})^*
=(\phi_1^{-1})^* (\phi_2^{-1})^*$, so this is a left action. We have a commutative diagram
\begin{equation}
\xymatrix{\Pic^{\sk}(V_\cH)\times_{\cH}\Pic^{\sk}(V_\cH) \ar[r] \ar[d]
& \Pic^{\sk}(V_{\cA^F})\times_{\cA^F}\Pic^{\sk'}(V_{\cA^F}) \ar[r] \ar[d] 
& \Pic^{\sk}(V_{\cA})\times_{\cA}\Pic^{\sk'}(V_{\cA}) \ar[d] 
\\
\cH \ar[r] & \cA^F=[P_g(\bZ)\backslash \cH] \ar[r] & \cA = [\Sp(2g,\bZ)\backslash \cH]
}
\end{equation}
where the middle (resp. right) column is the quotient of the left column by 
the left action of the discrete group $P_g(\bZ)$ (resp. $\Sp(2g,\bZ)$). In particular, the horizontal arrows are covering maps, and  the vertical arrows are holomorphic fibrations. 

For any $(L, L')\in \Pic^{\sk}(V_\tau)\times \Pic^{\sk'}(V_\tau)$, we have an isomorphism
$$
(\phi^{-1})^*: \Ext^w(L,L')
\lra \Ext^w((\phi^{-1})^*L, (\phi^{-1})^*L')
$$
of complex vector spaces. The sheaf $\cE^w_{\cH,\sk, \sk'}$ on $\Pic^{\sk}(V_\cH)\times_\cH \Pic^{\sk'}(V_\cH)$
descends to $\cE^w_{\cA^F,\sk, \sk'}$ on $\Pic^{\sk}(V_{\cA^F})\times_{\cA^F} \Pic^{\sk'}(V_{\cA^F})$ and 
$\cE^w_{\cA,\sk, \sk'}$ on $\Pic^{\sk}(V_\cA)\times_\cA \Pic^{\sk'}(V_\cA)$.

\section{The Fukaya category of \texorpdfstring{$(\bT^{2g}, \omega_\tau)$}{T 2g, omega tau}} 

\subsection{Objects}
\label{sec: fiber Fuk objects}

 As in Section \ref{sec:shift}, let $[v]$ denote the class of $v\in \bC^g$ in $V_\tau^+ =\bC^g/ (\bZ^g+\tau \bZ^g)\cong  V_\tau=(\bC^*)^2/\Gamma_\tau$. 
For each integer $\sk\in \bZ$ and each point $[v] \in V_\tau^+$, consider the holomorphic line bundle 
\begin{equation}\label{eq: line bundle mirror to slope k lag}
\cL_{\sk,[v]}:=\cL^{\otimes \sk} \otimes \bbL_{[v]} = \top^*_{[\frac{v}{\sk}]}\left(\cL^{\otimes \sk} \right).
\end{equation}
In Section \ref{sec: fiber slope Lagrangian} below, we will define
an object  $\hat{\ell}_{\sk,[v]}$ in $\Fuk(\bT^{2g},\omega_\tau)$ which is mirror to this line bundle.

For each point $[v]\in V_\tau^+$, consider the skyscraper sheaf
\begin{equation}\label{eq: skyscraper }
\cO_{[v]} =\top_{[-v]}^* \cO_{[0]}.
\end{equation}
In Section \ref{sec: fiber vertical Lagrangian}, we will define an object $\hat{\ell}_{\infty,[v]}$ in $\Fuk(\bT^{2g},\omega_\tau)$ which is mirror to the skyscraper sheaf.
Following \cite[Definition 1.1]{Fuk02} and \cite[Section 2]{ACLLb}, we consider
objects which are of the form $(\ell,\varepsilon)$, where $\ell$ is a Lagrangian submanifold of $(\bT^{2g}, \omega_\Omega)$ and $\varepsilon$  is a complex line bundle on $\ell$ 
equipped with a unitary connection whose curvature is equal to a multiple of the restriction of the B-field to $\ell$: 
\begin{equation}
2\pi \bi \sum_{j,k=1}^gB_{jk}dr_j\wedge d\theta_k\bigg \vert_{\ell}. 
\end{equation}
\subsubsection{The A-brane \texorpdfstring{$\hat{\ell}_{\sk,[v]}$}{ell hat k, v} mirror to the line bundle \texorpdfstring{$\cL_{\sk,[v]}$}{mathcal L k,v}}
\label{sec: fiber slope Lagrangian}

Given $\sk\in \bZ$ and $[v=a+\tau b] \in V_\tau^+$,  where $a, b \in \bR^g$, we will define a pair
$\hat{\ell}_{\sk, [v=a+\tau b]}= (\ell_{\sk, b}, \varepsilon_a)$,   where $\ell_{\sk, b}$ is a Lagrangian submanifold in $(\bT^{2g}, \omega_\Omega)$ 
and $\varepsilon_a$ is the trivial complex line  bundle $\ell_{\sk,b}\times \bC$  on $\ell_{\sk,b}$,  equipped with a flat $U(1)$ connection $\nabla_a$. (We will see that
the restriction of the B-field to the Lagrangian $\ell_{\sk,b}$ is zero.)

The Lagrangian submanifold $\ell_{\sk,b} \subset \bT^{2g} = \bR^{2g}/\bZ^{2g}$ is define by
\begin{equation}\label{eq:ell_k_definition}
\ell_{\sk,b}  :=\left\{ (r, \theta)\in \bR^{2g}/\bZ^{2g}: \theta \equiv b -\sk r \right\}
\end{equation} 
The notation $\equiv$ means modulo $\bZ^g$.  Note that
the Lagrangian $\ell_{\sk,b}$ depends only on the class $[b]=b+\bZ^g \in \bR^g/\bZ^g=(\bR/\bZ)^g$. 

Note that on $\ell_{\sk,b}$, we have $\theta_j= b_j -\sk r_j$, so
\begin{equation}
\left.\sum_{j,l=1}^gB_{jl}dr_j\wedge d\theta_l\right\vert_{\ell_{\sk,b}}=-\sk\sum_{j,l=1}^gB_{jl}dr_j\wedge dr_l=0
\end{equation}
because $B$ is symmetric.  Therefore, the $U(1)$ connection $\nabla_a$ is flat and the line bundle is topologically trivial. Up to gauge transformation, 
which is a $C^\infty$ map $\ell_{\sk,b} \to U(1)$,  the connection 1-form is 
\begin{equation}\label{eq: connection 1-form}
2\pi \bi adr, 
\end{equation}
where $a =(a_1,\ldots,a_g) \in [0,1)^g$ and $adr =\sum_{j=1}^g a_j dr_j$. Given any $a\in \bR^g$, let 

\begin{equation}\label{eq:epsilon-a}
\varepsilon_a := (\ell_{k,b}\times \bC, \nabla_a)
\end{equation}
be the trivial  complex line bundle
$\ell_{\sk,b}\times \bC$  on $\ell_{\sk,b}$ equipped with the flat $U(1)$-connection 
\begin{equation}\label{eq: flat U(1) connection}
\nabla_a = d+2\pi \bi adr. 
\end{equation}
The gauge equivalence class of $\nabla_a$ depends only on the class $[a] = a +\bZ^g \in \bR^g/\bZ^g =(\bR/\bZ)^g$. 

\begin{remark} \label{rmk: generation}  By \cite[Theorem 4]{Orlov_generation}, the set of line bundles $\{\cL_{\sk, [0]}: \sk\in \bZ\}$, generates $D^b\Coh(V_\tau)$.  That is, the images of the objects $\hat\ell_{\sk,[0]}$,  $\sk \in \bZ$, under the mirror functor $\Phi_\tau$, generate  $D^b\Coh(V_\tau)$.  \end{remark}

\subsubsection{The A-brane \texorpdfstring{$\hat{\ell}_{\infty,[v]}$}{ell hat infinity, v} mirror to the skycraper sheaf \texorpdfstring{$\cO_{[v]}$}
{O\_v}} \label{sec: fiber vertical Lagrangian}
Given $[v=a+\tau b] \in V_\tau^+$,  where $a, b \in \bR^g$, we will define a pair
$\hat{\ell}_{\infty,[v=a+\tau b]}= (\ell_{\infty,b}, \varepsilon_a)$,   where $\ell_{\infty,b}$ is a Lagrangian submanifold in $(\bT^{2g}, \omega_{\Omega})$ 
and $\varepsilon_a$ is the trivial complex line  bundle $\ell_{\infty, b} \times \bC$  on $\ell_{\infty,b}$,  equipped with a flat $U(1)$ connection $\nabla_a$. (We will see that
the restriction of the B-field to the Lagrangian $\ell_{\infty,b}$ is zero.)

The Lagrangian submanifold $\ell_{\infty,b} \subset \bT^{2g}=\bR^{2g}/\bZ^{2g}$ is define by
\begin{equation}\label{eq:t_b_definition}
\ell_{\infty,b} :=\left\{ (r, \theta)\in \bR^{2g}/\bZ^{2g}:  r=b \right\}.
\end{equation}
Note that Lagrangian $\ell_{\infty,b}$ depends only on the class $[b]=b+\bZ^g \in \bR^g/\bZ^g=(\bR/\bZ)^g$. 

On $\ell_{\infty,b}$, we have $d r_j =0 $, so
\begin{equation}
\left.\sum_{j,k=1}^gB_{jk}dr_j\wedge d\theta_k\right\vert_{\ell_{\infty,b}}= 0.
\end{equation}
Up to gauge transformation,  the connection 1-form is 
\begin{equation}\label{eq: connection 1-form vertical}
2\pi \bi ad\theta, 
\end{equation}
where $a \in [0,1)^g$. Given any $a\in \bR^g$, let $\varepsilon_a$ be the trivial  complex line bundle
$\ell_{\infty,b}\times \bC$  on $\ell_{\infty,b}$ equipped with the flat $U(1)$-connection 
\begin{equation}\label{eq: flat U(1) connection vertical}
\nabla_a = d+2\pi \bi a d\theta. 
\end{equation}
The gauge equivalence class of $\nabla_a$ depends only on the class $[a] = a +\bZ^g \in \bR^g/\bZ^g =(\bR/\bZ)^g$. 

\subsubsection{SYZ mirror symmetry for abelian varieties} \label{sec:SYZ}

Collecting the vertical Lagrangian objects previously defined into a family produces the SYZ mirror to $\bT^{2g}$. We use the word family because Family Floer theory (studied by Fukaya \cite{FukFF} and Abouzaid \cite{Ab14}) refers to the morphism groups, the Floer groups, in similar families. 

Starting with the symplectic torus $(\bT^{2g},\omega_\tau)$, we show that its complex mirror $(\bT^{2g},\omega_\tau)^\vee$, as a moduli space of the above vertical Lagrangians with connection, is $V_\tau$. The SYZ mirror prescription of \cite{SYZ} produces a mirror manifold given a special Lagrangian torus fibration. The idea, roughly, is to fix a torus fiber of the fibration, and then construct complex coordinates on a mirror chart where the norm of the complex coordinate is measured by the distance from the fixed fiber and the angle is measured by the holonomy of a flat connection on the fiber. 

Without B-field and with singular fibers, these are the complex coordinates $z_A$ defined in \cite[Lemma 2.7]{t_duality} for 
$A \in H_2(M,L;\bZ)$ a disc class in a symplectic manifold $M$ with boundary on Lagrangian $L$. For our torus $M=\bT^{2g}$, the Lagrangian torus fibration doesn't have singular fibers, so the rough idea is like that in \cite[Equation (2.3)]{AAK} where we have cylinders between elements of $H_1$ instead of one end of the cylinder pinching to a point at a singularity to produce a disc $A$. (The Lagrangians we consider are ``tautologically unobstructed" - they don't bound discs.) What was the area of the disc in defining $z_A$ is instead then the area of the cylinder, measuring the distance between two fibers, and corresponds to the moment map lengths $\xi_j$. 

These coordinates appear in \cite[Proposition 1.1]{Fuk02} specifically Equation (1.2) where $s,t$ parametrize the cylinder in $\phi_j$, where $j$ corresponds to the $j$th generator of $H_1(\bT^g;\bZ)$ for $L=\bT^g$. The proof that they are well-defined with the inclusion of the B-field is proven in \cite[Lemma 1.1]{Fuk02} for abelian varieties, by finding a symplectomorphic neighborhood of the fixed fiber Lagrangian in which nearby Lagrangians are graphs of closed 1-forms and we can explicitly write down the notion of distance between two fibers. (One place a general proof of the well-definedness of this quantity on a symplectic manifold with B-field is proven is in our paper \cite[Lemma 5.7]{ACLLb}.) With a B-field the connection now has curvature $2\pi \bi B|_L$.

Now we write down further detail in coordinates in our setting. Our Lagrangian torus fibration on $\bT^{2g}$ comes from the moment map. The prescription of \cite{SYZ} is that the SYZ mirror torus fibration has the same base; the SYZ dual fiber is the space of unitary connections on the trivial line bundle on the original fiber, with curvature $2\pi\bi B$. But as we saw above, $B$ restricted to linear and vertical Lagrangians is 0, so the connections are flat. The data needed to describe the connection is then a collection of $g$ complex numbers in $U(1)$ which describe the holonomy around each loop in $\pi_1(\bT^g)$ where $\bT^g$ is a Lagrangian torus fiber:
$$
\Hom(\pi_1(\bT^g)=\bZ^g, U(1)) \ni (e^{2\pi \bi \phi_1}, \ldots, e^{2\pi \bi \phi_g}).
$$
This corresponds to the connection
\begin{equation}
d-2\pi \bi\sum_{j=1}^g \phi_j d \theta_j
\end{equation}
on the trivial line bundle on $\bT^g$. Let $\xi_k = \sum_{j=1}^g \Omega_{jk} r_j$ be the moment map coordinates so
\begin{equation}
\omega_\Omega = \sum_{k=1}^g  d\xi_k \wedge d\theta_k, 
\end{equation}
and
\begin{equation}
(\xi_1,\ldots, \xi_g) + \Omega n = \Omega (r_1+n_1,\ldots, r_g+n_g).
\end{equation}
Then $(\phi_1,\ldots,\phi_g)$ describe coordinates on the dual fiber of the SYZ mirror $\pi_\tau^{\text{SYZ}}$ defined in Equation \eqref{eqn:SYZ}, up to a shift from the B-field as explained in Remark \ref{rem:SYZ_signs}: 
$$
V_\tau \xrightarrow[]{\pi_\tau^{\text{SYZ}}} T_\Omega
$$
\begin{equation}\label{eq:SYZ_complex}
(x_1,\ldots,x_g)=(e^{2\pi \bi (-\phi_1 - \sum_{j,k=1}^g B_{1j}\Omega^{jk}\xi_k-\bi\xi_1)}, \ldots,e^{2\pi \bi (-\phi_g - \sum_{j,k=1}^g B_{gj}\Omega^{jk}\xi_k-\bi\xi_g)})  \mapsto (\xi_1,\ldots,\xi_g).
\end{equation}
   
We can relate this to \cite[Chapter 1]{Fuk02}, where one SYZ mirror $(\bT^{2g},\omega_\tau^\bC)^\vee$ is $\cM(\tilde{L}_{\text{pt}})$ where $\tilde{L}_{\text{pt}}$ is the universal cover of what we have been calling a vertical Lagrangian
$$
{L}_{\text{pt}} = \ell_{\infty, 0}, \quad {L}_{\text{pt}}(b) = \ell_{\infty, b}
$$ 
so that $\text{pt} = 0 \in \bR^g$ and ${L}_{\text{pt}}(b)$ shifts the Lagrangian by adding $b \in \bR^g$. Then Fukaya defines the moduli space of Lagrangians with connection 
$$
\cM(\tilde L)=\{[L,\cL] \mid L \text{ is a flat Lagrangian submanifold of $\bT^{2g}$ parallel to $L(0)$} \}
$$
in other words, equivalence classes of pairs of Lagrangians $L$ with flat line bundle with connection $\cL$ that are parallel to $L(0)$ for linear Lagrangian $L(0)$ whose universal cover passes through all lattice points $\bZ^{g}$. Fukaya shows directly that this moduli space is connected and Hausdorff. Then again letting $\tilde L$ denote a lift to the universal cover, plugging in our notation to \cite[Proposition 1.3]{Fuk02} we see that the SYZ mirror to $(\bR^{2g}/\Omega \bZ^g \oplus \bZ^g, \bi\sum_{j=1}^g d\xi_j \wedge d\theta_j+\sum_{j,k=1}^g B_{jk}dr_j \wedge d\theta_k) = (\bR^{2g}/ \bZ^{2g}, \sum_{j,k=1}^g (B_{jk}+\bi\Omega_{jk})dr_j \wedge d\theta_k)$ is 
$$
\cM(\tilde L_{\text{pt}}) = \frac{\bR^{2g}/\tilde \ell_{\infty, 0} \oplus \Hom_\bR(\tilde \ell_{\infty, 0}, \bR)}{(\Omega \bZ^g \oplus \bZ^g)/(\Omega \bZ^g \oplus \bZ^g \cap \tilde\ell_{\infty, 0}) \oplus \Hom_\bR(\Omega \bZ^g \oplus \bZ^g \cap \tilde \ell_{\infty, 0},\bR)} = \left(\frac{\bR^{g}}{\Omega\bZ^g}\right)_{\{\xi_j\}} \oplus \frac{\Hom_\bR(\tilde \ell_{\infty, 0}, \bR)}{\Hom_\bR(\bZ^g ,\bR)}
$$
$$
\ni  -(\xi_1,\ldots,\xi_g)\oplus (\phi_1 + (B\Omega^{-1} \xi)_1,\ldots,\phi_g + (B\Omega^{-1} \xi)_g)
$$
$$
=(-\phi_1 - \sum_{j,k=1}^g B_{1j}\Omega^{jk}\xi_k-\bi\xi_1, \ldots,-\phi_g - \sum_{j,k=1}^g B_{gj}\Omega^{jk}\xi_k-\bi\xi_g) 
$$
because $(V,\Omega)$ in \cite{Fuk02} is $(\bR^{2g},\sum_{j,k=1}^g (B_{jk}+\bi\Omega_{jk})dr_j \wedge d\theta_k)$ and $\Gamma = \Omega \bZ^g \oplus \bZ^g\implies \Gamma \cap \tilde \ell_{\infty,0}=\bZ^g$ corresponding to the lattice in the $\theta$ coordinates. This calculation comes from the mirror expression above in Equation \eqref{eq:SYZ_complex}. The sign is due to our choice of $\tau=B+\bi \Omega_{\mathrm{ACLL}}$ for the coefficients in the complexified symplectic form, versus $\bi\Omega_{\mathrm{Fuk}}=-B+\bi \omega$ in \cite[p411]{Fuk02}. (A remark on notation: our $\Omega=\Omega_{\mathrm{ACLL}}$ is the real part while his 
$\Omega=\Omega_{\mathrm{Fuk}}$ is the whole complexified form.)

\subsubsection{Moduli spaces of objects}\label{sec:moduli-of-objects}
We first fix a complexified symplectic structure $\omega_\tau$ on $\bT^{2g}$ where $\tau\in \cH_g$. 
For each integer $\bk$, we define the moduli space
$$
\cM(\hat{\ell}_{\bk})_\tau := \{ \hat{\ell}_{\bk,[v]} \mid [v]\in V_\tau\}
$$
of objects in $\Fuk(\bT^{2g}, \omega_\tau)$ mirror to line bundles in $\Pic^{\bk}(V_\tau)$. There is a bijection 
\begin{equation} 
\Mir^{\bk}_\tau: \cM(\hat{\ell}_{\bk})_\tau   \lra \Pic^{\bk}(V_\tau), \quad \hat{\ell}_{\bk,[v]}\mapsto \cL_{\bk,[v]}
\end{equation}
We equip the set $\cM(\hat{\ell}_{\bk})_\tau$ with the structure of a complex manifold such that the above bijection is an isomorphism of complex manifolds. Let $\cM(\hat{\ell}_\infty)_\tau $ denote the moduli space $\cM(\tL_{\mathrm{pt}})$ defined in the previous subsubsection (Section \ref{sec:SYZ}). We have an isomorphism of complex manifolds:
\begin{equation}
\mathrm{Mir}^\infty_\tau: \cM(\hat{\ell}_\infty)_\tau \lra V_\tau, \quad  \hat{\ell}_{\infty, [v]}\mapsto [v].
\end{equation}

Varying $\tau\in \cH=\cH_g$, we obtain universal moduli spaces of objects
\begin{equation} \label{eqn:universal-M-H}
\cM(\hat{\ell}_{\bk})_{\cH} \cong \Pic^{\bk}(V_{\cH}) \lra \cH, \quad
\cM(\hat{\ell}_\infty)_{\cH} \cong V_{\cH} \lra \cH
\end{equation}
over $\cH$, and similarly 
\begin{equation} \label{eqn:universal-M-A}
\cM(\hat{\ell}_{\bk})_{\cA} \cong \Pic^{\bk}(V_{\cA}) \lra \cA, \quad
\cM(\hat{\ell}_\infty)_{\cA} \cong V_{\cA} \lra \cA
\end{equation}
over $\cA=\cA_g$ or $\cA_g^F$.

\subsection{Morphisms and differentials}
\label{subsec:mors_fiber}
Recall the notation for Lagrangian objects from Section \ref{sec: fiber Fuk objects} that $\hat{\ell}_{\sk,[a+\tau b]} = (\ell_{\sk,b},\varepsilon_{a})$, and for the vertical lagrangians $\hat{\ell}_{\infty,[a+\tau b]} = (\ell_{\infty,b},\varepsilon_{a})$, for $[a+\tau b]\in V_{\tau}^+$.  In this section, given two Lagrangian objects as above, we describe the morphism between them, which is a Floer cochain complex. 
 Then we discuss the differential on the Floer complex and  in turn the Floer cohomology.  We will discuss different cases separately below.

\subsubsection{Between $\hat{\ell}_{\sk_1,[a_1+\tau b_1]} = (\ell_{\sk_1,b_1},\varepsilon_{a_1})$ and $\hat{\ell}_{\sk_2,[a_2+\tau b_2]} = (\ell_{\sk_2,b_2},\varepsilon_{a_2})$ when $\sk_1\neq \sk_2$}
\label{subsubsec:k1-not-equal-to-k2}

When $\sk_1\neq \sk_2$, the Lagrangians $\ell_{\sk_1,b_1}$ and $\ell_{\sk_2,b_2}$ intersect transversely at $|\sk_2-\sk_1|^g$ points, which are
\begin{equation}
\ell_{\sk_1,b_1} \cap \ell_{\sk_2,b_2} = \left\{ p_{\sk_1,b_1,\sk_2,b_2}(\lambda): \lambda\in I_{g, |\sk_2-\sk_1|} =\{0,\ldots, |\sk_2-\sk_1|-1\}^g  \right\}, 
\end{equation}
where $p_{\sk_1, b_1, \sk_2,b_2}(\lambda) \in \bR^{2g}/\bZ^{2g}$ is given by 
\begin{equation}\label{eq:T4_Lag_intersec}
r\equiv \frac{\lambda + b_2-b_1}{\sk_2-\sk_1},\quad \theta\equiv \frac{-\sk_1\lambda  + \sk_2 b_1 -\sk_1 b_2}{\sk_2-\sk_1}.
\end{equation} The morphism between $\hat{\ell}_{\sk_1,b_1}$ and $\hat{\ell}_{\sk_2,b_2}$ is the Floer cochain complex 
\begin{equation}
\Hom_{\Fuk(\bT^{2g})}(\hat{\ell}_{\sk_1,[a_1+\tau b_1]}, \hat{\ell}_{\sk_2,[a_2+\tau b_2}]) = CF^*(\hat{\ell}_{\sk_1,[a_1+\tau b_1]}, \hat{\ell}_{\sk_2,[a_2+\tau b_2}]).
\end{equation}
As a complex vector space, 
\begin{equation}
CF^*(\hat{\ell}_{\sk_1,[a_1+\tau b_1]}, \hat{\ell}_{\sk_2,[a_2+\tau b_2}])= \bigoplus_{p\in \ell_{\sk_1,b_1} \cap \ell_{\sk_2,b_2}}  \Hom( (\varepsilon_{a_1})_p, (\varepsilon_{a_2})_p),
\end{equation}
where $\Hom((\varepsilon_1)_p, (\varepsilon_2)_p) \cong\bC$ is the space of complex  linear maps from 
the 1-dimensional complex vector space $(\varepsilon_{a_1})_p$ to $(\varepsilon_{a_2})_p$.  As shown in \cite[Example 4.6]{ACLLb}, a  $\bZ$-grading can be defined on $CF^*(\hat{\ell}_{\sk_1,[a_1 + \tau b_1]}, \hat{\ell}_{\sk_2,[a_2 + \tau b_2]})$, and the degree of a generator is 0 if $\sk_1 <\sk_2$, and $g$  if $\sk_1>\sk_2$.  Then, recalling notation from Equation \eqref{eq: line bundle mirror to slope k lag} and comparing with Theorem \ref{thm: Ext on V},  we see that when $\sk_1\neq \sk_2$, we indeed have 
\begin{equation}\label{eq:mirror morphisms}
CF^*(\hat{\ell}_{\sk_1,[a_1 + \tau b_1]}, \hat{\ell}_{\sk_2,[a_2 + \tau b_2]}) = \Ext^*(\cL_{\sk_1, [a_1+\tau b_1]} ,\cL_{\sk_2, [a_2+\tau b_2]} ) . 
\end{equation}

The differential on $CF^*(\hat{\ell}_{\sk_1,[a_1 + \tau b_1]}, \hat{\ell}_{\sk_2,[a_2 + \tau b_2]})$ is zero because it is given by the count of pseudoholomorphic bigons with boundaries on $\ell_{\bk_1, b_1}$ and $ \ell_{\bk_2, b_2}$, which lift to embedded bigons in the universal cover of $\bT^{2g}= \bR^{2g}/\bZ^{2g}$.  However, any lifts of  $\ell_{\bk_1,b_1}$ and $\ell_{\bk_2,b_2}$ to the universal cover are planes with only one intersection point, so they cannot bound a bigon.  So the Floer cochain complex is equal to the Floer cohomology $HF^*(\hat{\ell}_{\sk_1,[a_1 + \tau b_1]}, \hat{\ell}_{\sk_2,[a_2 + \tau b_2]})$.
Therefore, for any $\sk_1,\sk_2\in \bZ$ where $\sk_1\neq \sk_2$, and any 
$\hat\ell_{\sk_i,[v_i]} = \hat\ell_{\sk_i,[a_i +\tau b_i]} \in \cM(\hat{\ell}_{\bk_i})_\tau \cong\Pic^{\sk_i}(V_\tau)$ where $i=1,2$, we have
\begin{equation} \label{eq:mirror-morphisms-HF}
HF^*\big(\hat{\ell}_{\sk_1,[v_1]}, \hat{\ell}_{\sk_2,[v_2]}\big) \cong \Ext^*\big(\cL_{\sk_1, [v_1]} ,\cL_{\sk_2, [v_2]} \big) . 
\end{equation}
In other words, for any $\sk, \sk'\in \bZ$ where $\sk\neq \sk'$, any $w\in \{0,1,\ldots, g\}$, and any pair $(\hat{\ell}, \hat{\ell}')\in \cM(\hat{\ell}_{\sk})_\tau\times \cM(\hat{\ell}_{\sk'})_\tau$, we have
\begin{equation} \label{eqn:HF-k-kprime}
HF^w\big(\hat{\ell}, \hat{\ell}'\big) \cong 
\Ext^w\big(\Mir^{\sk}_\tau(\hat{\ell}) ,\Mir^{\sk'}_\tau(\hat{\ell}')\big)  
\end{equation}
where $\Mir^{\sk}_\tau$ is defined in Section \ref{sec:moduli-of-objects}. 

We now vary the pair $(\hat{\ell}, \hat{\ell}')$ in $\cM(\hat{\ell}_{\sk})_\tau\times \cM(\hat{\ell}_{\sk'})_\tau$ to obtain a family version of \eqref{eqn:HF-k-kprime}. 
We fix $\tau\in \cH$ and $\sk, \sk'\in \bZ$, where $\sk\neq \sk'$. There is an isomorphism
$$
\Mir^{\sk}_\tau\times \Mir^{\sk'}_\tau: 
\widecheck{P}:= \cM(\hat{\ell}_{\sk})_\tau\times \cM(\hat{\ell}_{\sk'})_\tau
\longrightarrow P:= \Pic^{\sk}(V_\tau) \times \Pic^{\sk'}(V_\tau)
$$
of complex projective manifolds. Let $\cE^w_{\tau,\sk,\sk'}$ be the locally free sheaf  of $\cO_P$-modules on $P$ in part 1 and part 2 of Theorem \ref{thm:E-tau-k-k}, so that the fiber of $\cE^w_{\tau, \sk, \sk'}$ over $(L,L')\in P$ is $\Ext^w(L,L')$.

\begin{theorem}  For any $w\in \{0,1,\ldots, g\}$ there is a locally free sheaf 
$\widecheck{\cE}_{\tau, \sk, \sk'}$ of $\cO_{\widecheck{P}}$-modules on $\widecheck{P}$
whose fiber over $(\hat{\ell}, \hat{\ell}')\in \widecheck{P}$ is the $w$-th Floer cohomology $HF^w(\hat{\ell}, \hat{\ell}')$, such that 
\begin{equation}\label{eqn:E-tau-k-kprime-mirror}
\widecheck{\cE}^w_{\tau, \sk, \sk'} \cong 
(\Mir^{\sk}_\tau \times \Mir^{\sk'}_\tau)^* \cE^w_{\tau, \sk, \sk'}.
\end{equation}
\end{theorem}

Finally, we vary $\tau\in \cH$ to obtain a global and universal
family version of \eqref{eqn:HF-k-kprime}. There is an isomorphism 
$$
\Mir^{\sk}_{\cH}\times_{\cH} \Mir^{\sk'}_{\cH}: 
\widecheck{P}:= \cM(\hat{\ell}_{\sk})_\cH\times_\cH \cM(\hat{\ell}_{\sk'})_\cH
\longrightarrow P:= \Pic^{\sk}(V_\cH) \times_\cH \Pic^{\sk'}(V_\cH)
$$
of complex manifolds. Let $\cE^w_{\cH,\sk,\sk'}$ be the locally free sheaf  of $\cO_P$-modules on $P$ in part 1 and part 2 of Theorem \ref{thm:E-k-k}, so that the fiber of $\cE^w_{\cH, \sk, \sk'}$ over $(\tau, L,L')\in P$ is $\Ext^w(L,L')$. We have the following global and universal family version of \eqref{eqn:HF-k-kprime}.

\begin{theorem}  For any $w\in \{0,1,\ldots, g\}$ there is a locally free sheaf 
$\widecheck{\cE}^w_{\cH, \sk, \sk'}$ of $\cO_{\widecheck{P}}$-modules on $\widecheck{P}$
whose fiber over $(\tau, \hat{\ell}, \hat{\ell}')\in \widecheck{P}$ is the $w$-th Floer cohomology $HF^w(\hat{\ell}, \hat{\ell}')$, such that 
\begin{equation}\label{eqn:E-k-kprime-mirror}
\widecheck{\cE}^w_{\cH, \sk, \sk'} \cong 
(\Mir^{\sk}_\cH\times_{\cH} \Mir^{\sk'}_\cH)^* \cE^w_{\cH, \sk, \sk'}.
\end{equation}
\end{theorem}

\subsubsection{Between $\hat{\ell}_{\sk_1,[a_1+\tau b_1]} = (\ell_{\sk_1,b_1},\varepsilon_{a_1})$ and $\hat{\ell}_{\sk_2,[a_2+\tau b_2]} = (\ell_{\sk_2,b_2},\varepsilon_{a_2})$  when $\sk_1=\sk_2=:\sk$}
\label{subsubsection:k1-equal-to-k2}

 Note that $\ell_{\sk, b_1}=\ell_{\sk, b_2}$  if and only if $b_1-b_2\in \bZ^g$, i.e. $[b_1]=[b_2]=[b]$.  When  $[b_1]\neq [b_2]$, the intersection $\ell_{\sk, b_1}\cap\ell_{\sk, b_2}$ is empty.   When $[b_1]=[b_2]=[b]$, we have $\ell_{\sk, b_1}=\ell_{\sk, b_2}=\ell_{\sk,b}$, and in order to define the Floer complex between the two objects $\hat \ell_{\sk, [a_1+\tau b_1]}$ and $\hat\ell_{\sk, [a_2+\tau b_2]}$, we can perturb  $\ell_{\sk, b_1}$  to $\phi_H^1(\ell_{\sk, b_1})$, so that it intersects transversely with $\ell_{\sk, b_1}$, where $\phi_H^1$ is the time-1 flow of a Hamiltonian $H$ on $\bT^{2g}$.  Then we can define $CF^*(\hat{\ell}_{\sk,[a_1 + \tau b_1]}, \hat{\ell}_{\sk,[a_2 + \tau b_2]})$ to be
 \begin{equation}\label{eq:perturbed morphism}
CF^*(\phi^1_{H}(\hat{\ell}_{\sk,[a_1 + \tau b_1]}), \hat{\ell}_{\sk,[a_2 + \tau b_2]}) = \bigoplus_{p\in \phi^1_H(\ell_{\sk_1,b_1}) \cap \ell_{\sk_2,b_2}}  \Hom( ((\phi^1_H)^*\varepsilon_{a_1})_p, (\varepsilon_{a_2})_p).
 \end{equation}
 
 The Lagrangian $\ell_{\sk,b}\cong (S^1)^g$ is topologically a product of $g$ circles, and we can construct a Hamiltonian $H$ so that each $S^1$ copy is perturbed in a way that results in two transverse intersection points, one of degree 0 and another of degree 1.   See Figure \ref{fig:k1=k2 in T2} for an illustration of this Hamiltonian perturbation when $g=1$, and one can also view this figure as the images of $\ell_{\sk,b}$ and $\phi^1_{H}(\ell_{\sk,b})$ under the projection to one of the $g$ components. Because the Maslov index is additive with respect to taking products, the number of degree $w$ intersection points in $\phi^1_H(\ell_{\sk,b})\cap \ell_{\sk,b} \subset \bT^{2g}$ is $\binom{g}{w}$, by choosing $w$ of the $g$ components to have degree 1 and the rest degree 0.  Hence, 
\begin{equation}\label{eq: CF for k1=k2}
CF^w(\hat{\ell}_{\sk,[a_1 + \tau b_1]}, \hat{\ell}_{\sk,[a_2 + \tau b_2]}) = \begin{cases}
0, & \text{if } [b_1]\neq [b_2],\\
\bC^{\binom{g}{w}}, & \text{if } [b_1]= [b_2].
\end{cases}
\end{equation}
 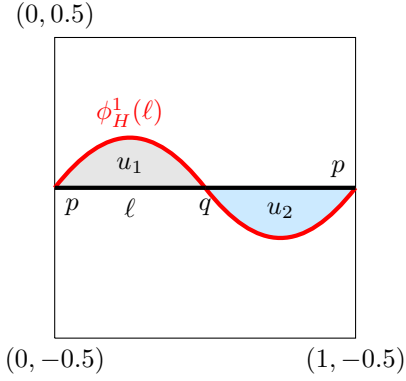
\begin{figure}
\centering                    \begin{tikzpicture}
   \draw (0,0) node[below] {$(0,-0.5)$} --(4,0) node[below ]{$(1,-0.5)$}--(4,4)--(0,4) node[above]{$(0,0.5)$}--(0,0);
   \filldraw[ultra thick, color=red, fill=gray!20] (0,2) .. controls (0.7, 2.89) and  (1.3,2.89).. (2,2) node[midway, above]{$\phi^1_H(\ell)$};
      \filldraw[ultra thick, color=red,fill=Cblue!20] (2,2) .. controls (2.7,1.11) and (3.3, 1.11).. (4,2);
         \draw[ultra thick] (0,2) node [below right] {$p$}--(2,2) node[below]{$q$} node[midway, below]{$\ell$};
         \draw[ultra thick] (2,2)--(4,2) node [above left]{$p$};
        \node at (1, 2.3) {$u_1$};
        \node at (3, 1.7) {$u_2$};
   \end{tikzpicture}
    \caption{The $1\times 1$ square above, with the opposite sides identified, illustrates $\bT^{2g}=\bR^{2g}/\bZ^{2g}$ in the case when $g=1$.  The thick black line illustrates a circle that is the  linear Lagrangian $\ell=\ell_{0,0}$ of slope zero.  The red curve is the perturbed Lagrangian $\phi^1_H(\ell)$.  These two circles, $\ell$ and $\phi^1_H(\ell)$, intersect at two points, $p$ and $q$, with $p$ of degree $0$ and $q$ of degree $1$.  They bound two bigons $u_1$ and $u_2$ of equal area but in the opposite directions, shaded by gray and blue colors, respectively.}
    \label{fig:k1=k2 in T2}
\end{figure}

More specifically, let
\begin{equation}\label{eq:y}
y_j=\sum_{l=1}^g \Omega_{jl} (\theta_l - b_l+  \sk r_l), \quad j=1,\ldots,g. 
\end{equation} 
Then $(r_1,\ldots, r_g, y_1,\ldots,  y_g)$ are 
local coordinates on $\bT^{2g}$. Note that 
$dr_j$, $d\theta_j$, $dy_j$ are {\em global} closed 1-forms on $\bT^{2g}$, and 
\begin{equation}\label{eq:dy-dtheta-dr}
dy_j =\sum_{l=1}^g \Omega_{jl}(d\theta_l+\sk dr_l), \quad j=1,\ldots, g. 
\end{equation}
We have
$$
\sum_{j=1}^g dr_j\wedge dy_j 
=\sum_{j,l=1}^g dr_j\wedge \Omega_{jl}(d\theta_l +\sk d r_l) 
= \omega_\Omega+\sk \sum_{j,l=1}^g \Omega_{jl} dr_j\wedge dr_l =\omega_\Omega,
$$
where the last equality holds because $\Omega_{jl}=\Omega_{lj}$. 
Therefore, $(r_j,y_j)$ are local Darboux coordinates on the symplectic torus $(\bT^{2g}, \omega_\Omega)$.
The Lagrangian $\ell_{\sk,b}$ is given by $y_j=0$ in coordinates $(r_j,y_j)$. Let $\epsilon>0$
be a small positive number.  Then we can choose the Hamiltonian to be 
\begin{equation}\label{eq: self intersection H}
H := \epsilon \sum_{j=1}^r \cos(2\pi r_j),
\end{equation}
which is a smooth function on $\bT^{2g}$, and 
$$
dH = -2\pi \epsilon \sum_{j=1}^r  \sin(2\pi r_j) dr_j = i_{X_H}\omega_\Omega
\quad \text{where} \quad
X_H = 2\pi\epsilon  \sum_{j=1}^g \sin(2\pi r_j) \frac{\partial}{\partial y_j}.
$$
So the flow of the Hamiltonian vector field $X_H$ is 
$\phi_H^t(r_j,y_j)= (r_j,  y_j+ 2\pi\epsilon \sin(2\pi r_j) t)$, and 
$\phi^1_H(\ell_{\sk,b})$ is given by $y_j = 2\pi\epsilon \sin(2\pi r_j)$, which
intersects $\ell_{\sk,b}$ transversely at $2^g$ points
\begin{equation}\label{eq: self intersection pts}
\left\{ p_\delta= \Big(\frac{\delta_1}{2},\ldots, \frac{\delta_g}{2}, 0,\ldots,0\Big) \mid \delta=(\delta_1,\ldots, \delta_g), \delta_j\in \{0,1 \} \right\}.
\end{equation}
 For each $j$, the points $(r_j, y_j)=(\frac{\delta_j}{2}, 0)$, with $\delta_j=0,1$, are the  two intersection points between $S^1$ and its perturbation, where $S^1\subset \bT^{2}$ is one of the $g$ components in the product $\ell\cong (S^1)^g\subset (\bT^2)^g$.

\begin{remark}\label{rmk: morse} The perturbed Lagrangian $\phi^1_H(\ell_{\sk,b})$ is contained in 
an open neighborhood $U$ of $\ell_{\sk,b}$ in $\bT^{2g}$ given by $y_j \in (-7\epsilon, 7\epsilon)$.
There is a symplectomorphism $\Psi: U\to U'$ where $U'$ is an open neighborhood
of the zero section in the cotangent $T^*\ell_{\sk,b}$ equipped with the symplectic form
$-d(\sum_{j=1}^g y_j dr_j) = \sum_{j=1}^g dr_j\wedge dy_j$ (the negative of the standard symplectic form on $T^*\ell_{\sk,b}$). 
Let $f = -H|_{\ell_{\sk,b}} = -\epsilon\sum_{j=1}^g \cos(2\pi r_j)$, which is a smooth function on $\ell_{\sk,b}$ (note that $\ell_{\sk,b}$ is the zero section in $T^*\ell_{\sk,b}$ defined by $y=0$, so $(r_1,\ldots, r_g)$ are  coordinates on $\ell_{\sk,b}$). Let $\Gamma(df)\subset T^*\ell_{\sk,b}$
be the graph of the exact 1-form $df = 2\pi \epsilon \sum_{j=1}^g \sin(2\pi r_j)$. Then $\Gamma(df)=\Psi(\phi^1_H(\ell_{\sk,b}))\subset U'$. 
The critical points of $f$, or equivalently the zeros of $df$, are
\begin{equation}
\left\{ r_\delta=\Big(\frac{\delta_1}{2},\ldots, \frac{\delta_g}{2} \Big): \delta_j \in \{ 0,1 \} \right\},
\end{equation}
which coincides with Equation \eqref{eq: self intersection pts}, which are the self-intersection points of $\ell_{\sk,b}\subset \bT^{2g}$ after the Hamiltonian perturbation.
\end{remark}

For each intersection point $p_\delta$ in Equation \eqref{eq: self intersection pts}, let us now discuss its degree, $\deg(p_{\delta})$, in the $\bZ$-graded Floer complex $CF^*(\phi^1_H(\ell_{\sk,b}), \ell_{\sk, b})$.  We compute below, culminating in Equation \eqref{eq: deg p},  that $\deg(p_\delta)=\sum_{j=1}^g \delta_j$, which is equal to the number of $j$ for which $\delta_j=1$.  Therefore, for each $1\leq w\leq g$, there are $\binom{g}{w}$ generators $p_\delta$ for $\deg(p_\delta)=w$, and hence $CF^w (\phi^1_H(\ell_{\sk,b}), \ell_{\sk, b})$ has dimension $\binom{g}{w}$ as stated in Equation \eqref{eq: CF for k1=k2}. (Note that $2^g=(1+1)^g=\sum_{w=0}^g \binom{g}{w}$, so this covers all $2^g$ intersection points.)

This degree computation is similar to that of \cite[Example 4.6]{ACLLb}, and see \cite[Section 4.5]{ACLLb} for a more detailed introduction on grading; we will use similar notations here.  The Lagrangian Grassmannian bundle $LGr:=LGr(\bT^{2g}, \omega_\Omega)$ is trivial, i.e. $LGr=\bT^{2g}\times U(g)/O(g)$.  Over the point $p_\delta \in \bT^{2g}$, under the trivialization of the tangent bundle of $\bT^{2g}$ using the $(r, y)$ coordinates and the identification with the cotangent bundle using the flat metric $\sum_{j=1}^g dr_j^2+dy_j^2$, the tangent space $T_{p_\delta}\ell_{\sk,b}\subset (\bR^{2g},\sum_{j=1}^gdr_j\wedge dy_j)$ is given by $\{y=0\}$ and $T_{p_\delta}\phi^1_H(\ell_{\sk,b})\subset (\bR^{2g},\sum_{j=1}^gdr_j\wedge dy_j)$ is given by 
\[
\{ y_j=4\pi^2\epsilon \cos (\pi \delta_j)r_j=4\pi^2\epsilon(-1)^{\delta_j}r_j,\  j=1,\ldots, g\}.
\]
Hence, under the identification  $LGr_{p_\delta}\cong LGr(\bR^{2g}, \sum_{j=1}^gdr_j\wedge dy_j)\cong U(g)/O(g)$, the tangent spaces   $T_{p_{\delta}}\ell_{\sk,b}$ and $T_{p_\delta}\phi^1_H(\ell_{\sk,b})$ correspond to the  following respective unitary matrices representing the cosets of $O(g)$ in $U(g)$:
\begin{equation}\label{eq: Lagrangian to matrix}
T_{p_\delta}\ell_{\sk,b} \mapsto I_g \quad \text{and} \quad T_{p_\delta}\phi^1_H(\ell_{\sk,b})\mapsto D:=\mathrm{diag}(e^{\bi \pi (-1)^{\delta_1}\psi_\epsilon}, 
\ldots,e^{\bi \pi(-1)^{\delta_g}\psi_\epsilon} ), 
\end{equation}
where $I_g$ denotes the identity matrix, and 
$
\psi_\epsilon=\frac{1}{\pi}\arctan(4\pi^2\epsilon)$. The squared phase map $U(g)/O(g)\to U(1)$ given by $aO(g)\mapsto \det(a)^2$ sends 
\begin{equation}\label{eq: squared phase map}
I_g\mapsto 1 \quad \text{and} \quad D\mapsto e^{2\pi \bi \psi_\epsilon\sum_{j=1}^g (-1)^{\delta_j}}.
\end{equation}
Composing Equations \eqref{eq: Lagrangian to matrix} and \eqref{eq: squared phase map}, we get that the squared phase functions $\alpha_L: L\to U(1)$, with $L$ being $\ell_{\sk,b}$ and $\phi^1_H(\ell_{\sk,b})$, evaluated at $p_\delta$ are
    $\alpha_{\ell_{\sk,b}}(p_\delta)=1$ and $\alpha_{\phi^1_H(\ell_{\sk,b})}=e^{2\pi \bi \psi_\epsilon\sum_{j=1}^g (-1)^{\delta_j}}$.
Then the gradings   $\widetilde \alpha_L: L\to \bR$ on the Lagrangians defined by $\alpha_L=e^{2\pi \bi \widetilde \alpha_L}$, with $L$ being $\ell_{\sk,b}$ and $\phi^1_H(\ell_{\sk,b})$,  evaluate at $p_\delta$ to
\begin{equation}
    \widetilde \alpha_{\ell_{\sk,b}}(p_\delta)=0 \quad \text{and} \quad \widetilde \alpha_{\phi^1_H(\ell_{\sk,b})}(p_\delta)=\psi_\epsilon\sum_{j=1}^g (-1)^{\delta_j}.
\end{equation}
The canonical short path $\lambda(t)$ from  $T_{p_\delta}\phi^1_H(\ell_{\sk,b})$ to $T_{p_{\delta}}\ell_{\sk,b}$ rotates clockwise for  each component.  Using the identification of the tangent spaces with the unitary matrices given in Equation \eqref{eq: Lagrangian to matrix}, $\lambda:[0,1]\to U(g)/O(g)$ can be explicitly written as
\begin{equation}
\lambda(t) = \mathrm{diag}(e^{-\bi\pi \left(\delta_1+(-1)^{\delta_1}\psi_\epsilon\right)t},\ldots,e^{-\bi\pi \left(\delta_g+(-1)^{\delta_g}\psi_\epsilon\right)t})D.
\end{equation}
Then for $\alpha_\lambda:[0,1]\to U(1)$ defined by post-composing $\lambda$ with the the squared phased map and for $\widetilde{\alpha_\lambda}:[0,1]\to \bR$ defined by $\alpha_\lambda=e^{2\pi \bi \widetilde{\alpha_\lambda}}$ and $\widetilde{\alpha_\lambda}(0)=\widetilde\alpha_{\phi^1_H(\ell_{\sk,b})}(p_{\delta})$, we have 
\begin{equation}
    \widetilde{\alpha_\lambda}(t)=\sum_{j=1}^g -\delta_j t+(1-t)(-1)^{\delta_j}\psi_\epsilon.  \end{equation}
Then 
\begin{equation}\label{eq: deg p}
\deg(p_\delta)=(\widetilde\alpha_{\ell_{\sk,b}}(p)-\widetilde\alpha_{\phi^1_H(\ell_{\sk,b})}(p)) - (\widetilde{\alpha_\la}(1) -\widetilde{\alpha_\la}(0))=\widetilde\alpha_{\ell_{\sk,b}}(p)- \widetilde{\alpha_\la}(1)= \sum_{j=1}^g \delta_j.
\end{equation}

\begin{remark}[Continuation of Remark \ref{rmk: morse}]\label{rmk: morse index}
The Hessian of $f$ with respect to the flat Riemannian metric
$\sum_{j=1}^g dr_j^2$ on $\ell_{\sk,b}$ is 
$$
\mathrm{Hess}(f) = 4\pi^2 \epsilon \sum_{j=1}^r \cos(2\pi r_j) dr_j^2. 
$$
The eigenvalues of $\mathrm{Hess}(f)(\frac{\delta_1}{2},\ldots, \frac{\delta_g}{2})$ are
$\{ 4\pi^2\epsilon (-1)^{\delta_j}: j=1\ldots, g\}$. In particular, $f$ is a Morse function, which is consistent with the fact that $\Gamma(df) = \Psi(\phi_H^1(\ell_{\sk,b}))$ intersects the zero section $\Gamma(0)= \Psi(\ell_{\sk,b})$ transversally. The Morse index of the critical point $(\frac{\delta_1}{2},\ldots, \frac{\delta_g}{2})$ of $f$ is equal to the number of $j$ for which $\delta_j=1$, which is equal to the Floer cohomology degree of the intersection point
$(\frac{\delta_1}{2},\ldots, \frac{\delta_g}{2}, 0,\ldots, 0)$ of the Lagrangians $\phi^1_H(\ell_{\sk,b})$  and $\ell_{\sk,b}$ in $\bT^{2g}$. 
\end{remark} 
 Recalling notation from Equation \eqref{eq: line bundle mirror to slope k lag} and comparing  Equation \eqref{eq: CF for k1=k2} with  Theorem \ref{thm: Ext on V}, we can see that in this case with $\sk_1=\sk_2=:\sk$, indeed the morphisms between the Lagrangian objects matches the extension groups which are the morphisms between the mirror objects in $V_\tau$, i.e.
\begin{equation}\label{morphisms_match}
CF^w(\hat{\ell}_{\sk,[a_1 + \tau b_1]}, \hat{\ell}_{\sk,[a_2 + \tau b_2]})=
\Ext^w(\cL_{\sk, [a_1+\tau b_1]} ,\cL_{\sk, [a_2+\tau b_2]} ),
\end{equation}
except when $[b_1]=[b_2]$ but $[a_1]\neq [a_2]$, in which case the extension group above is zero but the Floer cochain complex is nonzero.  This is because in the case when $[b_1]=[b_2]$, we used a Hamiltonian perturbation to define the morphism $CF^p(\hat{\ell}_{\sk,[a_1 + \tau b_1]}, \hat{\ell}_{\sk,[a_2 + \tau b_2]})$ as in Equation \eqref{eq:perturbed morphism}.  Hamiltonian perturbation preserves the Floer cohomology group but not the cochain complex.  As we discuss the differentials below, we shall see that the Floer cohomology group is indeed zero when $[a_1]\neq [a_2]$ and stays nonzero when $[a_1]=[a_2]$, and hence  \begin{equation}HF^w(\hat{\ell}_{\sk,[a_1 + \tau b_1]}, \hat{\ell}_{\sk,[a_2 + \tau b_2]})=
\Ext^w(\cL_{\sk, [a_1+\tau b_1]} ,\cL_{\sk, [a_2+\tau b_2]} )
\end{equation}
for all cases.

{\bf Parametrizing the bigons in the differential.} Consider the differential
$$
\dd: p \in \bC^{\binom{g}{w}} \cong CF^w(\phi^1(\ell_{\sk,b}), \ell_{\sk,b}) \to CF^{w+1}(\phi^1(\ell_{\sk,b}), \ell_{\sk,b}) \cong \bC^{\binom{g}{w+1}}\ni q
$$
corresponding to two intersection points labeled by a choice of $\delta, \delta' \in \bZ_2^g$. In each coordinate where they differ we can have a bigon parametrized by $t$ in the base and $s$ in the height as in Figure \ref{fig:k1=k2 in T2}:
$$
r_j(t) = \begin{cases}
    \frac{\delta_j}{2}, & \delta_j=\delta'_j\\
    \frac{t}{2} \text{ or } 1-\frac{t}{2}, & \delta_j'- \delta_j=1\\
    \frac{1-t}{2} \text{ or } \frac{1+t}{2}, & \delta_j'-\delta_j=-1
\end{cases}, \qquad  \theta_j(s,r) =  s \cdot \sin(2\pi r_j).
$$
However,  only pseudo-holomorphic bigons are counted in the differential. Since $\omega_{\Omega}$ is trivialized in $(r_j,y_j)$ coordinates, we can take the almost complex structure $J$ on $\bT^{2g}$ to be trivial as well, so the two structures are compatible. In particular, in each factor the bigon should be orientation-preserving. Since $\sin(\pi(1\pm t)) = \mp \sin(\pi t)$ reverses orientation, we cannot have a $J$-holomorphic bigon between $p$ and $q$ if $\delta_j'-\delta_j=-1$ for some $j$. Thus, because $\deg(q) = \deg(p)+1$, then $q$ and $p$ can only differ in one coordinate in which $p$ is 0 and $q$ is $\frac{1}{2}$. Let $j(pq)$ be this coordinate.

\begin{example}
    (1) For example, in Figure \ref{fig:k1=k2 in T2}, $g=1$, $\delta=0$, $\delta'=1$, with $p=(r,y)=(0,0)$ and $q=(\frac{1}{2},0)$ we have the two bigons
$$
u_1:[0,1] \times [0,2\pi \epsilon] \to \bT^{2}: (t,s) \mapsto (\frac{t}{2},s \cdot \sin(\pi t))
$$
and
$$
u_2:[0,1] \times [0,2\pi \epsilon] \to \bT^{2}: (t,s) \mapsto (1-\frac{t}{2},-s \cdot \sin(\pi t)).
$$

(2) For an example with $g=4$ if $p=(0,0,\frac{1}{2},0,0,0,0,0)=(r,y)$ and $q=(\frac{1}{2},\frac{1}{2},0,0,0,0,0,0)=(r,y)$ there are no $J$-holomorphic bigons because we can't go from $\frac{1}{2}$ to $0$ in the third coordinate, but if $p=(0,0,\frac{1}{2},0,0,0,0,0)$ and $q=(\frac{1}{2},0,\frac{1}{2},0,0,0,0,0)$ then $\langle \dd p, q\rangle$ counts the two bigons 
$$
u_1(t,s) =(\frac{t}{2},0 ,\frac{1}{2},0,s \cdot \sin(\pi t), 0,0,0)
$$
and
$$
u_2(t,s) =(1-\frac{t}{2},0 ,\frac{1}{2},0,-s \cdot \sin(\pi t), 0,0,0).
$$
\end{example}

{\bf Calculate the weights of discs in the differential.} To finish the calculation of the differential, we need the area of the bigons and the holonomy of the connection 1-form around the bigons. The area is: 
\[
A=\int_{[0,1] \times [0,2\pi \epsilon]} u_1^*\omega_\tau=\int_{[0,1] \times [0,2\pi \epsilon]} u_2^*\omega_\tau.
\] 

We first write the B-field $\omega_B =\sum_{i,j=1}^g B_{jk} dr_j\wedge d\theta_k$ in terms of the Darboux coordinates $(r_j,y_j)$. By 
Equation \eqref{eq:dy-dtheta-dr},
$$
dy_j = \sum_{l=1}^g \Omega_{jl}(d\theta_l+\sk dr_l) 
\implies d\theta_k = \sum_{l=1}^g \Omega^{kl} dy_l -\sk dr_k.
$$
So 
$$
\omega_B =\sum_{j,k=1}^g B_{jk} dr_j \wedge d\theta_k 
=\sum_{j,k,l=1}^g B_{jk}\Omega^{kl} dr_j\wedge dy_l
-\sk \sum_{j,k=1}^g B_{jk} dr_j\wedge dr_k
=\sum_{j,k,l=1}^g B_{jk}\Omega^{kl} dr_j\wedge dy_l
$$
where the last equality holds because $B_{jk}=B_{kj}$. The Hamiltonian flow $\phi_H: \bT^{2g}\times [0,1] \to \bT^{2g}$ is given by 
$$
\phi_H(r_j,  y_j,\ft) = 
\left(r_j, y_j+ \ft 2\pi\epsilon \sin(2\pi r_j)\right).
$$
(We use $\mathfrak{t}$ to parametrize the isotopy so as not to confuse with $t$ which parametrizes the boundary of the bigon.)
Therefore, 
\begin{equation}\label{eq:pullback-B}
\phi_H^*\omega_B = \sum_{j,k,l=1}^g B_{jk}\Omega^{kl} dr_j \wedge
d(y_l + \ft2\pi \epsilon \sin(2\pi r_l) ).
\end{equation}

We know the connection on $\ell_{\sk,b}$, so  it remains to find the connection on $\phi^1_H(\ell_{\sk,b})$. Under Hamiltonian isotopy, the connection on $\ell_{\sk,b}\times [0,1]$ is
$$
d-2\pi \bi (a_1-\mathfrak{t}2\pi \epsilon \sum_{i,j,k=1}^g B_{ij}\Omega^{jk}\sin(2\pi r_k) dr_i)
$$ 
so that it interpolates between that on $\ell_{\sk,b}$ at time $\ft=0$ and the connection on $\phi^1(\ell_{\sk,b})$ at $\ft=1$. This is a correct choice because its curvature equals the restriction of 
$-2\pi\bi \phi_H^*\omega_B$ to $\ell_{\sk,b}\times [0,1]$:  
\begin{equation}
    \begin{aligned}
        -2\pi \bi  \phi_H^*\omega_B \big|_{\ell_{\sk,b}\times [0,1]} 
    &=-2\pi \bi \sum_{j,k,l=1}^g B_{jk}\Omega^{kl} dr_j \wedge
d(y_l + \ft 2\pi \epsilon \sin(2\pi r_l)) \big|_{y=0} \\
& = -2\pi \bi \sum_{j,k,l=1}^g B_{jk}\Omega^{kl} dr_j \wedge
d(\ft2\pi \epsilon \sin(2\pi r_l)) \\
&=  d (-2\pi \bi (a_1-\mathfrak{t}2\pi \epsilon \sum_{j,k,l=1}^g B_{jk}\Omega^{kl}\sin(2\pi r_l) dr_j)).
    \end{aligned}
\end{equation}

Now we have all we need to compute the holonomy. Under the parametrization of the boundaries $\dd u_1,\dd u_2$ in $t$, along $\ell_{\sk,b}$ the holonomy contributes
\[
\int_{p\xrightarrow[\partial u_1]{\ell_{\sk,b}}q} 2\pi \bi a_{2,j(pq)}dr_{j(pq)} = \int_{0}^{1}2\pi \bi  a_{2,j(pq)} d(t/2) = \pi \bi a_{2,j(pq)}
\]
\[
\int_{p\xrightarrow[\partial u_2]{\ell_{\sk,b}}q} 2\pi \bi a_{2,j(pq)}dr_{j(pq)} =  \int_{0}^{1}2\pi \bi  a_{2,j(pq)} d(1-t/2)=-\pi \bi a_{2,j(pq)}.
\]
Going back along $\phi^1(\ell_{\sk,b})$ in the bigon, $t$ goes in the opposite direction:
\[
\int_{q\xrightarrow[\partial u_1]{\phi^1(\ell_{\sk,b})}p} 2\pi \bi ( a_{1,j(pq)} dr_{j(pq)} -2\pi \epsilon \sum_{j,k,l=1}^g B_{jk}\Omega^{kl}\sin(2\pi r_l) dr_j) =2\pi \bi (\int_{1}^0  a_{1,j(pq)} d(t/2) + A')=-\pi \bi a_{1,j(pq)} + 2\pi \bi A',
\]
\[
\int_{q\xrightarrow[\partial u_2]{\phi^1(\ell_{\sk,b})}p}2\pi \bi (a_{1,j(pq)} dr_{j(pq)}-2\pi \epsilon \sum_{j,k,l=1}^g B_{jk}\Omega^{kl}\sin(2\pi r_l) dr_j) =\pi \bi  a_{1,j(pq)} + 2\pi \bi A'.
\]
where
$$
A'=2\pi \epsilon \int_{t=0}^{1} \sum_{i,k=1}^g B_{ik}\Omega^{kj(pq)}\sin(\pi t) d(t/2)
$$
which has an integrand even in $t$ because both $\sin(\pi t)$ and $dt$ change sign, so it is counted the same in both bigons. 

The last step in finding the differential is to determine the signs on the counts of moduli spaces, where each moduli space has one bigon. If $j_1<\ldots<j_w$ are the indices where $\delta$ is $1/2$ in $p$, then define $k$ by $j_k<j(pq)<j_{k+1}$. We have a Koszul sign from identifying an intersection point $p$ labelled by $\delta$ with its orientation line $o_p=e_{j_1} \wedge \ldots \wedge e_{j_w}$ where $e_{j_l} \cong (T_{(1/2,0)}\phi^1(\ell_{\sk,b})_{j_l})=\{\frac{\dd}{\dd y_{j_l}}=-\frac{\dd}{\dd r_{j_l}}\}$ is the tangent space of $
\phi^1(\ell_{\sk,b})$ projected to the $j_l$ factor. We only consider those $j_l$ for which $\delta_{j_l} = \frac{1}{2}$ because they contribute index 1, \cite[Equation (13.6)]{seidel}. We count with trivial spin structure and follow \cite[Equation (11.20), Equation (12.18), Definition of $\mu^1$ p184]{seidel}, where $|c_u|$ of \cite[Equation (12.18)]{seidel} is the sign change $(-1)^{w+k}$ needed to move $e_{j(pq)}$ from the end to the correct ordered position. 

Putting it all together, the $pq$-th coefficient of the differential is
\begin{equation}\label{eq:diffl_calc}
\begin{aligned}
\langle \partial p,q \rangle & =(-1)^w(-1)^{w+k} (e^{2\pi \bi(A+A')} e^{\pi \bi  a_{2,j(pq)}}e^{-\pi \bi   a_{1,j(pq)}}  - e^{2\pi \bi (A+A')}e^{-\pi \bi a_{2,j(pq)}}e^{\pi \bi  a_{1,j(pq)}})\\
& =(-1)^{k} e^{2\pi \bi(A+A')}\left(e^{\pi \bi (a_{2,j(pq)}-a_{1,j(pq)})}-e^{-\pi \bi ( a_{2,j(pq)}-a_{1,j(pq)})}\right) .
\end{aligned}
\end{equation}

{\bf Find the cohomology of the differential.} We prove that
$$
HF^w(\hat{\ell}_{\sk,[a_1+\tau b]}, \hat{\ell}_{\sk,[a_2+\tau b]}) = \begin{cases}
    0, & [a_1] \neq [a_2]\\
    \bC^{\binom{g}{w}}, & [a_1] = [a_2]
\end{cases}, \qquad w=0,\ldots,g
$$
by induction on $g$. For ease of notation, let 
$$
C_g^w:=CF^w(\phi^1(\hat \ell_{\sk,[a_1+\tau b]}),\hat \ell_{\sk,[a_2+\tau b]})=\bC^{\binom{g}{w}}
$$
and
$$
\Delta_k:=e^{2\pi \bi(A+A')}\left(e^{\pi \bi (a_{2,k}-a_{1,k})}-e^{-\pi \bi  (a_{2,k}-a_{1,k})}\right).
$$

{\it Base case: $g=1$.}  For the base case with $g=1$, we need to show $\partial p =0$ iff $[a_1] \neq [a_2] \in \bR/\bZ$. Note $a_{j}=a_{j,1}$ have one component for $j=1,2$. The chain complex is
$$
0 \to \bC^{\binom{1}{0}}=\bC \cdot p \xrightarrow[]{\dd^1} \bC^{\binom{1}{1}} = \bC \cdot q \to 0
$$
where
$$
\dd p = e^{2\pi \bi(A+A')}\left(e^{\pi \bi (a_{2,1}-a_{1,1})}-e^{-\pi \bi  (a_{2,1}-a_{1,1})}\right) \cdot q.
$$
In particular, $\dd^1$ is surjective because $[a_2-a_1] \neq 0$; we have $x-\frac{1}{x} = 0 \iff x^2-1=0 \iff x=\pm 1$ $\iff$ $a_2-a_1 \in \bZ$. The sequence is exact in degree 0 because $\dd^1$ is injective and in degree 1 because $\dd^1$ is surjective. Thus the $g=1$ case is proved.

{\it General case: result for $g-1$ implies result for $g$.} Recall the identity
$$
{\binom{g}{w}} = \binom{g-1}{w} + \binom{g-1}{w-1}.
$$
We use this to split our chain complexes into two pieces, depending on if the last coordinate of the intersection points are 0 or 1/2. Thus the sequence $C_{g}^{w-1} \xrightarrow[]{\dd_g^{w-1}} C_{g}^w \xrightarrow[]{\dd_g^w} C_g^{w+1}$ decomposes
\begin{center}
\begin{tikzpicture}
\node (v1) at (-5,1) {$(C_{g-1}^{w-2}, \frac{1}{2})\ni(x', \frac{1}{2})$};
\node at (-5,0) {$\oplus$};
\node (v3) at (-5,-1) {$(C_{g-1}^{w-1}, 0)\ni(y',0)$};
\node (v2) at (-1.5,1) {$(C_{g-1}^{w-1}, \frac{1}{2})\ni(x,\frac{1}{2})$};
\node at (-1.5,0) {$\oplus$};
\node (v4) at (-1.5,-1) {$(C_{g-1}^{w}, 0)\ni(y,0)$};
\node (v5) at (1.5,1) {$(C_{g-1}^{w}, \frac{1}{2})$};
\node at (1.5,0) {$\oplus$};
\node (v6) at (1.5,-1) {$(C_{g-1}^{w+1}, 0)$};
\draw [{->}] (v1) edge (v2) node [at start, left=85pt, above=40pt] {$(\dd_{g-1}^{w-2}, 1)$};
\draw [{->}] (v3) edge (v2) node [at start, left=60pt, fill=white] {$\cdot (-1)^{w-1} \Delta_g$};
\draw [{->}] (v3) edge (v4) node [at start, left=90pt, below=40pt] {$(\dd_{g-1}^{w-1}, 1)$};
\draw [{->}] (v2) edge (v5) node [at start, right=10pt, above = 40pt] {$(\dd_{g-1}^{w-1}, 1)$};
\draw [{->}] (v4) edge (v5)  node [at start,   fill=white] {$\cdot (-1)^{w} \Delta_g$};
\draw [{->}] (v4) edge (v6) node [at start, right=10pt, below =40pt] {$(\dd_{g-1}^{w}, 1)$};
\end{tikzpicture}
\end{center}
Each column corresponds to a fixed degree of intersection points. We only label intersection points by their first $g$ coordinates $(r_1,\ldots,r_g)$, because all the $y$ coordinates are 0. The top row corresponds to intersection points where the last coordinate is 1/2. The bottom row corresponds to intersection points where the last coordinate is 0. The differential on the top row cannot count bigons in the last coordinate since it is already 1/2 there, so it is a differential on the first $g-1$ coordinates. The differential on the bottom row counts both bigons in the first $g-1$ coordinates, as well as one bigon in the last coordinate, hence we have a horizontal arrow and a diagonal arrow, respectively. An example of this diagram for $g=3, w=1$ is carried out below in Example \ref{ex:diffl}.

{\it Show that $\ker(\dd^w_g) \subset \text{im}(\dd^{w-1}_g)$.} We take an element in the kernel of the differential: $\dd_g^w((x,\frac{1}{2})+ (y,0))=0$ for $x \in C_{g-1}^{w-1}$ and $y \in C_{g-1}^w$. Then with the decomposition of the differential into three arrows we have three terms
\begin{equation}\label{eq:diff_ker}
(\dd_{g-1}^{w-1}x,\frac{1}{2}) + (\dd_{g-1}^w y, 0)+ (-1)^{w}\Delta_g (y,\frac{1}{2})=0.
\end{equation}
Since only the second term has basis elements ending in 0, that term must be zero and $\dd_{g-1}^w y = 0 \in C^w_{g-1}.$ We would like to use induction, but have to check one case before doing so. If all $a_{1,j}=a_{2,j}$ for $1 \leq j \leq g-1$ then we must have $a_{1,g} \neq a_{2,g}$, horizontal maps are 0 and diagonal maps are isomorphisms. In particular $\ker \dd^w|_{C^w_g}=(C^{w-1}_{g-1},\frac{1}{2})=\text{image}(\cdot\Delta_g)|_{(C^{w-1}_{g-1},0)})$ and so the sequence is exact in degree $w$. Otherwise, $[a_1] \neq [a_2]$ in the first $g-1$ coordinates and by induction in $g$
\begin{equation}\label{eq:y_induct_PSS}
 y=\dd_{g-1}^{w-1} y', \quad y' \in C^{w-1}_{g-1}.
\end{equation}
In Equation \eqref{eq:diff_ker} we are left with the remaining first and third terms summing to 0, so taking the first $g-1$ coordinates of that equation (the last coordinate being 1/2)
$$
\dd^{w-1}_{g-1}x + (-1)^{w}\Delta_g y=0 \in C^{w-1}_{g-1}.
$$
Plugging in for $y$ from Equation \eqref{eq:y_induct_PSS} we have
$$
\dd^{w-1}_{g-1}(x + (-1)^{w} \Delta_g y')=0.
$$
Again by induction in $g$ there is some $x' \in C^{w-2}_{g-1}$ so that
$$
x + (-1)^{w} \Delta_g y'=\dd^{w-2}_{g-1} x'.
$$
Putting it all together, we find that our original element in the kernel, $(x, \frac{1}{2}) + (y,0)$ is in fact also the image of $(x',\frac{1}{2})+(y',0)$:
$$
\dd^{w-1}_g((x',\frac{1}{2})+(y',0)) = (\dd^{w-2}_{g-1}x',\frac{1}{2}) + (-1)^{w-1} (\Delta_g y',\frac{1}{2}) + (\dd^{w-1}_{g-1}y',0)
$$
$$
=(x, \frac{1}{2})+\cancel{(-1)^{w} (\Delta_g y',\frac{1}{2})}+ \cancel{(-1)^{w-1} (\Delta_g y', \frac{1}{2})}+(y,0)=(x, \frac{1}{2})+(y,0).
$$
Hence since conversely the image is automatically contained in the kernel for our cochain complex ($\dd^2=0$ from \cite[top of p182]{seidel}), the kernel and image must be equal, and cohomology is 0 when $[a_1] \neq [a_2]$.

\begin{example}\label{ex:diffl}Take $g=3$, $w=1$.
    $$
    0\to \bC(0,0,0) \xrightarrow[]{\dd^0_3} \bC(\frac{1}{2},0,0) \oplus \bC(0,\frac{1}{2},0) \oplus \bC(0,0,\frac{1}{2}) \xrightarrow[]{\dd^1_3} \bC(\frac{1}{2}, \frac{1}{2},0) \oplus \bC(\frac{1}{2},0,\frac{1}{2}) \oplus \bC(0,\frac{1}{2},\frac{1}{2}) \xrightarrow[]{\dd^2_3} \bC(\frac{1}{2},\frac{1}{2},\frac{1}{2})\to 0
    $$
    By direct calculation: $  \text{image}(\dd^0_3) = \bC \cdot  ( \Delta_1 (\frac{1}{2}, 0,0) + \Delta_2 (0, \frac{1}{2},0) + \Delta_3 (0, 0,\frac{1}{2}))
    $.  The diagram above in degree $w=1$ here is
    \begin{center}
\begin{tikzpicture}
\node (v1) at (-5,1) {$0$};
\node at (-5,0) {$\oplus$};
\node (v3) at (-5,-1) {$\bC(0,0,0)\ni(y',0)$};
\node (v2) at (0,1) {$\bC(0,0,\frac{1}{2})\ni(x,\frac{1}{2})$};
\node at (0,0) {$\oplus$};
\node (v4) at (0,-1) {$\bC(\frac{1}{2},0,0) \oplus \bC(0,\frac{1}{2}, 0)\ni(y,0)$};
\node (v5) at (5,1) {$\bC(\frac{1}{2},0,\frac{1}{2})\oplus  \bC(0,\frac{1}{2},\frac{1}{2})$};
\node at (5,0) {$\oplus$};
\node (v6) at (5,-1) {$\bC(\frac{1}{2}, \frac{1}{2},0)$};
\draw [{->}] (v1) edge (v2) node [at start, left=80pt, above=30pt] {$\dd^{-1}=0$};
\draw [{->}] (v3) edge (v2) node [at start,left=70pt,  fill=white] {$\cdot \Delta_3$};
\draw [{->}] (v3) edge (v4) node [at start, left=80pt, below=30pt] {$(\dd_2^{0}, 1)$};
\draw [{->}] (v2) edge (v5) node [at start, right=70pt, above = 30pt] {$(\dd_2^{0}, 1)$};
\draw [{->}] (v4) edge (v5)  node [at start,right=50pt, fill=white] {$\cdot - \Delta_3$};
\draw [{->}] (v4) edge (v6) node [at start, right=90pt, below =30pt] {$(\dd_2^{1}, 1)$};
\end{tikzpicture}
\end{center}
    For an element in the kernel,
    $$\dd^1_3(a(\frac{1}{2},0,0) + b(0,\frac{1}{2},0) + c(0,0, \frac{1}{2})) =(\dd_{2}^{0}(c(0,0)),\frac{1}{2}) + (\dd_{2}^1 (a(\frac{1}{2},0)+b(0,\frac{1}{2})), 0)-\Delta_3 (a(\frac{1}{2},0)+b(0,\frac{1}{2}),\frac{1}{2})=0
    $$
    where comparing with the notation above 
    $(x,\frac{1}{2})=(c(0,0), \frac{1}{2})$ and $ (y,0)=(a(\frac{1}{2},0)+b(0,\frac{1}{2}),0)$. From the second term, with a 0 in the last coordinate, we deduce that
    $$
    (a\Delta_2 -b \Delta_1 )(\frac{1}{2}, \frac{1}{2},0)=0.
    $$
    By induction, $a(\frac{1}{2},0,0) + b(0,\frac{1}{2},0)$ is in the image $(\dd^0_2 y',0)$ where $y'=m(0,0)$ for some multiple $m$. (Or here we can see directly here that $a, b$ differ by a constant of proportionality from $\Delta_1 , \Delta_2$.) That is, $a=m \Delta_1 $, $b=m \Delta_2$. From the coefficients in the remaining terms we have
    $$
    (\dd^{0}_{2} c(0,0)) - \Delta_3 (a(\frac{1}{2},0)+b(0,\frac{1}{2}))=\dd^0_2(c(0,0) - m \Delta_3  (0,0))  = 0.
    $$
    By induction, $c(0,0) - m \Delta_3  (0,0)$ is in the image of the previous differential. The previous differential is in degree $-1$ from $0 \to CF^0_3$, so $c(0,0) - m \Delta_3  (0,0)$ must equal 0 and $c=m \Delta_3$. Thus $a, b, c$ indeed take on the values for which the original element in the kernel will also be in the image of $\dd^0_3$, and cohomology is 0 in degree 1. 
    $$
    a=m \Delta_1 , \quad b=m \Delta_2 , \quad c=m \Delta_3 \implies a(\frac{1}{2},0,0) + b(0,\frac{1}{2},0) + c(0,0, \frac{1}{2}) \in \text{image}(\dd^0_3) .
    $$
This concludes the example.
\end{example}

{\bf Finishing the proof: $[a_1]=[a_2]$.} If $[a_1]=[a_2]$ then all $\Delta_k=0$ so by Equation \eqref{eq:diffl_calc} we know $\dd=0$, and cohomology coincides with the cochain complex. In other words,
\begin{equation} \label{eqn:finishing}
HF^w(\hat{\ell}_{\sk,[a_1+\tau b]}, \hat{\ell}_{\sk,[a_2+\tau b]}) = \begin{cases}
    0, & [a_1] \neq [a_2]\\
    \bC^{\binom{g}{w}}, & [a_1] = [a_2]
\end{cases} \quad = \quad \Ext^w(\cL^{\otimes \sk} \otimes \bbL_{[a_1 + \tau b]}, \cL^{\otimes \sk} \otimes \bbL_{[a_2+\tau b]})
\end{equation}
where the last equality is from the $\sk=\sk'$ cases of Theorem \ref{thm: Ext on V}.

\begin{remark}
    This agrees with Morse homology with twisted coefficients (corresponding to our unitary connection 1-forms) under the PSS isomorphism \cite[Equation (12.14)]{seidel}, see  \cite[Example 2.14 for $S^1$]{BHS19}. See also \cite[Appendix C]{abouzaid2021homological}.
\end{remark}

By \eqref{eqn:finishing}, for any $\sk \in \bZ$, 
any $w\in \{0,1,\ldots, g\}$, and any pair
$(\hat{\ell}, \hat{\ell}')\in \cM(\hat{\ell}_{\sk})_\tau\times \cM(\hat{\ell}_{\sk})_\tau$, we have
\begin{equation} \label{eqn:HF-k-k}
HF^w\big(\hat{\ell}, \hat{\ell}'\big) \cong 
\Ext^w\big(\Mir^{\sk}_\tau(\hat{\ell}) ,\Mir^{\sk}_\tau(\hat{\ell}')\big). 
\end{equation}

We now vary the pair $(\hat{\ell}, \hat{\ell}')$ in $\cM(\hat{\ell}_{\sk})_\tau\times \cM(\hat{\ell}_{\sk})_\tau$ to obtain a family version of \eqref{eqn:HF-k-k}. We fix $\tau\in \cH$ and $\sk\in \bZ$. Then there is an isomorphism
$$
\Mir^{\sk}_\tau\times \Mir^{\sk}_\tau: 
\widecheck{P}:= \cM(\hat{\ell}_{\sk})_\tau\times \cM(\hat{\ell}_{\sk})_\tau
\longrightarrow P:= \Pic^{\sk}(V_\tau) \times \Pic^{\sk}(V_\tau)
$$
of complex projective manifolds. Let $\cE^w_{\tau,\sk,\sk}$ be the coherent sheaf  of $\cO_P$-modules on $P$ in part 3 of Theorem \ref{thm:E-tau-k-k}, so that the fiber of $\cE^w_{\tau, \sk, \sk}$ over $(L,L')\in P$ is $\Ext^w(L,L')$.

\begin{theorem}  For any $w\in \{0,1,\ldots, g\}$ there is a coherent sheaf
$\widecheck{\cE}^w_{\tau, \sk, \sk}$ of $\cO_{\widecheck{P}}$-modules on $\widecheck{P}$
whose fiber over $(\hat{\ell}, \hat{\ell}')\in \widecheck{P}$ is the $w$-th Floer cohomology $HF^w(\hat{\ell}, \hat{\ell}')$, such that 
\begin{equation}\label{eqn:E-tau-k-k-mirror}
\widecheck{\cE}^w_{\tau, \sk, \sk} \cong 
(\Mir^{\sk}_\tau \times \Mir^{\sk}_\tau)^* \cE^w_{\tau, \sk, \sk}.
\end{equation}
In particular, $\widecheck{\cE}_{\tau,\sk, \sk}$  is supported along the diagonal in 
$\cM(\hat{\ell}_{\sk})_\tau\times \cM(\hat{\ell}_{\sk})_\tau$. 
\end{theorem}

Finally, we vary $\tau\in \cH$ to obtain a global and universal
family version of \eqref{eqn:HF-k-k}. There is an isomorphism 
$$
\Mir^{\sk}_{\cH}\times_{\cH} \Mir^{\sk}_{\cH}: 
\widecheck{P}:= \cM(\hat{\ell}_{\sk})_\cH\times_\cH \cM(\hat{\ell}_{\sk})_\cH
\longrightarrow P:= \Pic^{\sk}(V_\cH) \times_\cH \Pic^{\sk}(V_\cH)
$$
of complex manifolds. Let $\cE^w_{\cH,\sk,\sk}$ be the coherent sheaf  of $\cO_P$-modules on $P$ in part 3 of Theorem \ref{thm:E-k-k}, so that the fiber of $\cE^w_{\cH, \sk, \sk}$ over $(\tau, L,L')\in P$ is $\Ext^w(L,L')$. We have the following global and universal family version of \eqref{eqn:HF-k-k}.

\begin{theorem}  For any $w\in \{0,1,\ldots, g\}$ there is a coherent sheaf 
$\widecheck{\cE}^w_{\cH, \sk, \sk}$ of $\cO_{\widecheck{P}}$-modules on $\widecheck{P}$
whose fiber over $(\tau, \hat{\ell}, \hat{\ell}')\in \widecheck{P}$ is the $w$-th Floer cohomology $HF^w(\hat{\ell}, \hat{\ell}')$, such that 
\begin{equation}\label{eqn:E-tau-k-k-mirror-global}
\widecheck{\cE}^w_{\cH, \sk, \sk} \cong 
(\Mir^{\sk}_\cH\times_{\cH} \Mir^{\sk}_\cH)^* \cE^w_{\cH, \sk, \sk}
\end{equation}
In particular, $\widecheck{\cE}^w_{\cH,\sk, \sk}$  is supported along the diagonal in 
$\cM(\hat{\ell}_{\sk})_\cH \times_{\cH}  \cM(\hat{\ell}_{\sk})_\cH$. 
\end{theorem}

\subsubsection{Between $\hat{\ell}_{\infty,[a_1+\tau b_1]} = (\ell_{\infty,b_1},\varepsilon_{a_1})$ and $\hat{\ell}_{\infty,[a_2+\tau b_2]} = (\ell_{\infty,b_2},\varepsilon_{a_2})$ }
\label{subsubsection:infinity-infinity}

For the objects defined in Section \ref{sec: fiber vertical Lagrangian}, note that the Lagrangians $\ell_{\infty,b_1} = \ell_{\infty, b_2}$ if and only if $b_1-b_2\in \bZ^g$, in which case we do a similar perturbation by a Hamiltonian as above but use $\sin(2\pi \theta_j)$ instead of $\cos(2\pi r_j)$. If $b_1-b_2 \neq \bZ^g$, then $\ell_{\infty,b_1}\cap \ell_{\infty,b_2}$ is empty and $[a_1+\tau b_1]\neq [a_2+\tau b_2] \in V_\tau$, so 
\begin{equation}
CF^*(\hat{\ell}_{\infty,[a_1 + \tau b_1]}, \hat{\ell}_{\infty,[a_2 + \tau b_2]}) =0   = \Ext^*(\cO_{[a_1+\tau b_1]}, \cO_{[a_2+\tau b_2]}). 
\end{equation}

\subsubsection{Between $\hat{\ell}_{\sk,[a_1+\tau b_1]} = (\ell_{\sk,b_1},\varepsilon_{a_1})$ and $\hat{\ell}_{\infty,[a_2+\tau b_2]} = (\ell_{\infty,b_2},\varepsilon_{a_2})$ }
\label{subsubsection:k-infinity}

The vertical Lagrangian $\ell_{\infty, b_2}$ always intersects the affine Lagrangian $\ell_{\sk,b_1}$ at  exactly one point given by $(r, \theta)=(b_2, b_1-\sk b_2)$.  The morphism complex is $CF^*(\hat \ell_{\sk,[a_1+\tau b_1]}, \hat{\ell}_{\infty, [a_2+\tau b_2]})\cong \bC$, which is generated by the one intersection point in degree $0$. So
$HF^*(\hat \ell_{\sk,[a_1+\tau b_1]}, \hat{\ell}_{\infty, [a_2+\tau b_2]})= CF^*(\hat \ell_{\sk,[a_1+\tau b_1]}, \hat{\ell}_{\infty, [a_2+\tau b_2]})$.

On the mirror side we have 
$$
\Ext^w(\cL^{\otimes \sk}_{[a_1+\tau b_1]}, \cO_{[a_2+\tau b_2]})=\begin{cases}
    \cL^{\otimes (-\sk)}_{[a_1+\tau b_1]}|_{[a_2+\tau b_2]} \cong \bC,&  w =0\\
    0, & w \neq 0
\end{cases}
$$
where $\cO_{[a_2+\tau b_2]}$ is the skyscraper sheaf supported at $x = e^{-2\pi\bi(a_2+\tau b_2)} $.  
In other words, given $\tau\in \cH$, $L \in \Pic^{\sk}(V_\tau)$ and $[v]\in V_\tau$, we have
\begin{equation}\label{eqn:Ext-L-skyscraper}
\Ext^w(L, \cO_{[v]})=\begin{cases}
   \left. L^{-1}\right|_{[v]} \cong \bC,&  w=0;\\
    0, & w \neq 0,
\end{cases}
\end{equation} 
Therefore, for any $\sk\in \bZ$, any $w\in \{0,1,\ldots, g\}$, and any pair $(\hat{\ell}, \hat{\ell}')  \in \cM(\hat{\ell}_{\sk})_\tau\times \cM(\hat{\ell}_\infty)_\tau$, we have
\begin{equation} \label{eqn:HF-k-infinity}
HF^w(\hat{\ell}, \hat{\ell}') \cong \Ext^w\big(\Mir^{\sk}_\tau(\hat{\ell}),  \Mir^{\infty}_\tau(\hat{\ell}')\big)
\end{equation}
where $\Mir^{\sk}_\tau: \cM(\hat{\ell}_{\sk})_\tau \to \Pic^{\sk}(V_\tau)$ and $\Mir^\infty_\tau: \cM(\hat{\ell})_\tau\to V_\tau$ are defined in Section \ref{sec:moduli-of-objects}. 
See Section \ref{sec:prod_vert} and Remark \ref{rem:SYZ_signs} for an example of a mirror correspondence, with the vertical Lagrangian, on the product structures.

We now vary the pair $(\hat{\ell}, \hat{\ell}')$ in $\cM(\hat{\ell}_{\sk})_\tau\times \cM(\hat{\ell}_\infty)_\tau$ to obtain a family version of 
\eqref{eqn:HF-k-infinity}.  Let $\cP_{\sk} \to V_\tau \times \Pic^{\sk}(V_\tau)$ be the universal line bundle defined in Section \ref{sec:poincare}. To match convention, define
$$
\mathrm{Flip}_\tau: \Pic^{\sk}(V_\tau)\times V_\tau \to V_\tau\times \Pic^{\sk}(V_\tau), \quad (L, [v])\mapsto ([v], L).
$$
Given $w\in \{0,1,\ldots, g\}$,  define a locally free sheaf $\cE^w_{\tau, \sk,\infty}$ of 
$\cO_P$-modules on $P= \Pic^{\sk}(V_\tau)\times V_\tau$ by 
\begin{equation}\label{eqn:E-k-infinity-tau}
\cE^w_{\tau, \sk, \infty}=\begin{cases}
  \mathrm{Flip}_\tau^*\cP_k^{-1}, &  w=0\\
    0, & w \neq 0
\end{cases}
\end{equation} 
In particular, $\cE^0_{\tau, \sk,\infty}$ is a line bundle on $\cP_{\sk}(V_\tau)\times V_\tau$. The fiber
of $\cE^w_{\tau,\sk,\infty}$ over $(L,[v])\in \Pic^{\sk}(V_\tau)\times V_\tau$
is $\Ext^w(L, \cO_{[v]})$. There is an isomorphism
$$
\Mir^{\sk}_\tau\times \Mir^\infty_\tau: \widecheck{P}:=\cM(\hat{\ell}_{\sk})_\tau\times \cM(\hat{\ell}_\infty) \lra P= \Pic^{\sk}(V_\tau)\times V_\tau
$$
of complex projective manifolds. 
\begin{theorem} For any $w\in \{0,1,\ldots, g\}$ there is a locally free sheaf $\widecheck{\cE}_{\tau, \sk,\infty}$ of $\cO_{\widecheck{P}}$-modules
on $\widecheck{P}$ whose fiber over $(\hat{\ell}, \hat{\ell}')\in \widecheck{P}$ is the $w$-th Floer cohomology $HF(\hat{\ell}, \hat{\ell}')$, such that
\begin{equation}\label{eqn:E-k-infinity-tau-mirror}
\widecheck{\cE}^w_{\tau, \sk,\infty}\cong (\Mir^{\sk}_\tau\times \Mir^\infty_\tau)^* \cE^w_{\tau, \sk,\infty}.
\end{equation}
In particular $\widecheck{\cE}^0_{\tau, \sk,\infty}$ is a line bundle and $\widecheck{\cE}^w_{\tau, \sk,\infty}=0$ if $w\neq 0$.
\end{theorem}

Finally, we vary $\tau\in \cH$ to obtain a global and universal
family version of \eqref{eqn:HF-k-infinity}. 
Define
$$
\mathrm{Flip}_\cH: \Pic^{\sk}(V_\cH)\times_{\cH} V_\cH \to V_\cH\times_{\cH} \Pic^{\sk}(V_\cH), \quad (\tau, L, [z])\mapsto (\tau, [z], L).
$$
Given $w\in \{0,1,\ldots, g\}$,  define a locally free sheaf $\cE^w_{\cH, \sk,\infty}$ of 
$\cO_P$-modules on $P := \Pic^{\sk}(V_\cH)\times_{\cH} V_{\cH}$ by 
\begin{equation}\label{eqn:E-k-infinity}
\cE^w_{\cH, \sk, \infty}=\begin{cases}
  \mathrm{Flip}_\cH^*\cP_k(\cH)^{-1}, &  w=0,\\
    0, & w \neq 0, 
\end{cases}
\end{equation} 
where $\cP_{\sk}(\cH)\to V_{\cH}\times_{\cH} \Pic^{\sk}(V_{\cH})$ is the universal line bundle in Section \ref{sec:universal-picard}. Then the fiber of $\cE^w_{\cH,\sk,\infty}$ over
$(\tau, L, [z])\in P$ is $\Ext^w(L, \cO_{[z]})$. There is an isomorphism 
$$
\Mir^{\sk}_{\cH}\times_{\cH} \Mir^{\infty}_{\cH}: 
\widecheck{P}:= \cM(\hat{\ell}_{\sk})_\cH\times_\cH \cM(\hat{\ell}_\infty)_\cH
\longrightarrow P:= \Pic^{\sk}(V_\cH) \times_\cH V_\cH
$$
of complex manifolds.  We have the following global and universal family version of \eqref{eqn:HF-k-infinity}.

\begin{theorem}  For any $w\in \{0,1,\ldots, g\}$ there is a locally free sheaf 
$\widecheck{\cE}^w_{\cH, \sk, \infty}$ of $\cO_{\widecheck{P}}$-modules on $\widecheck{P}$
whose fiber over $(\tau, \hat{\ell}, \hat{\ell}')\in \widecheck{P}$ is the $w$-th Floer cohomology $HF^w(\hat{\ell}, \hat{\ell}')$, such that 
\begin{equation}\label{eqn:E-k-infinity-tau-mirror-global}
\widecheck{\cE}^w_{\cH, \sk, \infty} \cong 
(\Mir^{\sk}_\cH\times_{\cH} \Mir^{\infty}_\cH)^* \cE^w_{\cH, \sk, \infty}.
\end{equation}
In particular $\widecheck{\cE}^0_{\cH, \sk,\infty}$ is a line bundle and $\widecheck{\cE}^w_{\cH, \sk,\infty}=0$ if $w\neq 0$.
\end{theorem}



\subsection{Product structures} 
\label{sec: fiber product structure}

In this section, we describe the product structures on Floer complexes between objects of the form $\hat \ell_{\bk,[v]}$. The product map 
\begin{equation}\label{eqn:mu2}\mu^2: CF(\hat{\ell}_{\sk_2,[a_2 + \tau b_2]}, \hat{\ell}_{\sk_3,[a_3 + \tau b_3]})\otimes CF(\hat{\ell}_{\sk_1,[a_1 + \tau b_1]}, \hat{\ell}_{\sk_2,[a_2 + \tau b_2]})\to CF(\hat{\ell}_{\sk_1,[a_1 + \tau b_1]}, \hat{\ell}_{\sk_3,[a_3 + \tau b_3]})
\end{equation}
 is defined by 
\begin{equation}\label{eq: product def}
\mu^2(\varrho_{p_2}p_2,  \varrho_{p_1}p_1) 
 =\sum \limits_{\substack{q\in \ell_{\bk_1,b_1}\cap \ell_{\bk_3,b_3}\\ \ind([u])=0}} \# \cM (q, p_1, p_2; [u]) e^{2\pi\bi \int u^*\omega_\tau} \Upsilon(\varrho_{p_2}, \varrho_{p_1}; [u])q.
\end{equation}
Let us explain the ingredients in the above formula; see \cite[Section 5, Definition 5.8]{ACLLb} for a more in depth introduction of this definition.   The inputs are $\varrho_{p_1}\in \hom((\varepsilon_{a_1})_{p_1}, (\varepsilon_{a_2})_{p_1})$ and  $\varrho_2\in \hom((\varepsilon_{a_2})_{p_2}, (\varepsilon_{a_3})_{p_2})$, where $p_1\in \ell_{\bk_1,b_1}\cap \ell_{\bk_2,b_2}$ and $p_2\in \ell_{\bk_2,b_2}\cap \ell_{\bk_3,b_3}$.  The output is generated by points $q\in \ell_{\bk_3}\cap \ell_{\bk_3}$ for which $p_1, p_2, q$ form a holomorphic triangle. The moduli space $\cM(q, p_1,p_2; [u])$ consists of all homotopy classes $[u]\in \pi_2(\bT^{2g}, \ell_{\bk_1,b_1}\cup \ell_{\bk_2,b_2}\cup \ell_{\bk_3,b_3})$, where the map $u:(\mathbb D, \partial \mathbb D, \{z_0, z_1, z_2\})\to (\bT^{2g}, \ell_{\sk_1, b_1}\cup \ell_{\sk_2, b_2}\cup \ell_{\sk_3, b_3}, \{q, p_1, p_2\})$ is a solution to the Cauchy-Riemann equation $\overline\partial_Ju=0$ for the standard $J$ on $\bT^{2g}\cong\bR^{2g}/\bZ^{2g}$, with boundary conditions on the  Lagrangians $\ell_{\bk_1,b_1}$, $\ell_{\bk_2,b_2}$, and  $\ell_{\bk_3,b_3}$ in the counterclockwise order as shown in Figure \ref{fig: triangle}.   Here, $\mathbb D$ is a disc with three marked points, $z_0$, $z_1$, and $z_2$, which are  mapped to the corresponding
output and input generators involved in the product.   Denote by $\cM(q, p_1, p_2)$ the moduli space all such triangles, independent of the homotopy class $[u]$.  Assuming transversality,
the moduli space of such holomorphic triangles has dimension equal to $\ind([u])$, and the product is
given by the count of the index 0 solutions.  

The exponent $\int u^*\omega_\tau$ is the area of $u$ with respect to the complexified symplectic form $\omega_\tau$.  The factor $\Upsilon(\varrho_{p_2},\varrho_{p_1};[u])\in \hom((\varepsilon_{a_1})_q, (\varepsilon_{a_3})_q)$ is the holonomy of the $U(1)$-connection around the boundary of $u$, composed with the inputs, $\varrho_{p_1}$ and $\varrho_{p_2}$, at the corners, and we will discuss this in more detail later.  Together, the weight $e^{2\pi\bi \int u^*\omega_\tau} \Upsilon(\varrho_{p_2}, \varrho_{p_1}; [u])$ depends only on the homotopy class of $u$ as shown in \cite[Lemma 5.7]{ACLLb}.
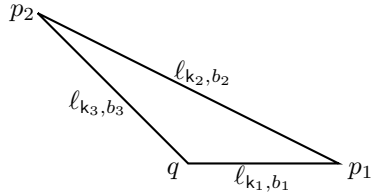
\begin{figure}[h]
\centering
\begin{tikzpicture}

\draw[thick] (-2.5,0.5) -- (-0.5,-1.5) -- (1.5,-1.5) -- (-2.5,0.5);
\node at (-0.7,-1.6) {$q$};
\node at (1.8,-1.6) {$p_1$};
\node at (-2.7,0.5) {$p_2$};
\node at (0.5,-1.7) {$\ell_{\bk_1,b_1}$};
\node at (-0.3, -0.3) {$\ell_{\bk_2,b_2}$};
\node at (-1.7,-0.7) {$\ell_{\bk_3, b_3}$};
\end{tikzpicture}
    \caption{A triangle contributing to $\cM(q,p_1, p_2)$.}
    \label{fig: triangle}
\end{figure}

To compute the $\mu^2$ in Equation \eqref{eqn:mu2}, below let us first figure out a description for the holomorphic triangles contributing to $\cM(q, p_1, p_2)$.

 \begin{lemma}\label{lem: u_m} Given 
 \begin{eqnarray*}
 & p_1\equiv \left( \frac{\lambda_1+b_2-b_1}{\bk_2-\bk_1}, \frac{-\bk_1\lambda_1+\bk_2 b_1-\bk_1b_2}{\bk_2-\bk_1}\right) \in \ell_{\bk_1, b_1}\cap \ell_{\bk_2, b_2},  & \text{for some } \lambda_1\in I_{g,|\sk_2-\sk_1|}= \{0,\ldots, |\bk_2-\bk_1|-1\}^g,\\
 & p_2\equiv \left( \frac{\lambda_2+b_3-b_2}{\bk_3-\bk_2}, \frac{-\bk_2\lambda_2+\bk_3 b_2-\bk_2b_3}{\bk_3-\bk_2}\right) \in \ell_{\bk_2, b_2}\cap \ell_{\bk_3, b_3},& \text{for some } \lambda_2\in I_{g,|\sk_3-\sk_2|}= \{0,\ldots, |\bk_3-\bk_2|-1\}^g,
 \end{eqnarray*} 
 the output of $\mu^2(p_2, p_1)$ is generated by the set of
 \begin{equation}\label{eq: triangle q}
      q\equiv\left(\frac{\lambda_{1}+\lambda_{2}+(b_3-b_1)+(\bk_3-\bk_2)w}{\bk_3-\bk_1}, \  \frac{-\bk_1(\lambda_1+\lambda_2)+\bk_3b_1-\bk_1b_3-\bk_1(\bk_3-\bk_2)w}{\bk_3-\bk_1}\right), 
 \end{equation}
 where $w\in I_{g,|\sk_3-\sk_1|}= \{0, \ldots, |\bk_3-\bk_1|-1\}^g.$
There are infinitely many holomorphic triangles  $u_m$ contributing to the count of $\mu^2(p_2,p_1)$, one for each $m\in \bZ^g$.  (The explicit description of each $u_m$ is given in the proof below.)
\end{lemma}

\begin{proof}

To express more precisely each triangle that contributes to $\cM(p_1, p_2, q)$, for some output $q$, let us  write the coordinates  of their lifts $\tilde p_1, \tilde p_2, \tilde q$ in the universal cover $\bR^{2g}$.  To do so, we can pick arbitrary lifts for $\ell_{\bk_1,b_1}$ and $p_1$, and then trace through the triangle to determine the lifts for $\ell_{\bk_2,b_2}$, $p_2$, $\ell_{\bk_3,b_3}$, and $q$.  We can arbitrarily pick the lift of $\ell_{\bk_1,b_1}=\{(r,\theta): \theta\equiv b_1-\bk_1 r\}$ to be
\begin{equation}\label{eq: lift lj}
   \tilde  \ell_{\bk_1,b_1}=\{(r,\theta): \theta=b_1-\bk_1 r\}.
\end{equation} Then we can choose the lift of the input point $p_1$ along $\tilde \ell_{k_1,b_1}$ to be 
\begin{equation} \label{eq: lift p1}
    \tilde p_1=\left(r_{p_1}=\frac{\lambda_{1}+b_2-b_1}{\bk_{2}-\bk_1}, \ \theta_{p_1}=b_1-\bk_1 r_{p_1}=b_2-\bk_2 r_{p_1} +m_{\bk_2}\right),  \lambda_{1}\in I_{g, |\sk_2-\sk_1|}.
\end{equation}
Note that we could have chosen $r_{p_1}$ to be $\frac{\lambda_{1}+b_2-b_1}{\bk_2-\bk_1}$ plus any integer, and the choice we make does not matter for this first vertex of the triangle.    Then since $p_1\in \ell_{\bk_1,b_1}\cap \ell_{\bk_2,b_2}$, we see that $\tilde p_1$ determines the lift $\tilde \ell_{\bk_2,b_2}$.  Indeed, we get from the formula for $\theta_{p_1}$ that
\begin{equation} \label{eq: lift lk}
    \tilde \ell_{\bk_2,b_2}=\{(r,\theta): \theta=b_2-\bk_2 r+m_{\bk_2}\}, \quad m_{\bk_2}=(\bk_2-\bk_1)r_{p_1}-(b_2-b_1)=\lambda_{1}.
\end{equation}
 Then the lift of the input point $p_2$ is determined by tracing along the Lagrangian  $\tilde \ell_{\bk_2,b_2}$.  For each $m\in \bZ^g$, we have a triangle with 
\begin{equation}\label{eq: lift p2}
    \tilde p_2=\left(r_{p_2}=\frac{\lambda_{2}+b_3-b_2}{\bk_3-\bk_2}+m, \ \theta_{p_2}=b_2-\bk_2 r_{p_2}+m_{\bk_2} =b_3-\bk_3 r_{p_2}+m_{\bk_3}\right),   \lambda_{2}\in I_{g,|\sk_3-\sk_2|}.
\end{equation}
Then since $p_2\in \ell_{\bk_2,b_2}\cap \ell_{\bk_3,b_3}$, we see that $\tilde p_2$ determine the lift $\tilde \ell_{\bk_3,b_3}$.  Indeed, substituting the formula for $m_\bk$ above and using the formula for $\theta_{p_2}$, we get
\begin{equation} \label{eq: lift ll}
    \tilde \ell_{\bk_3,b_3}=\{(r,\theta): \theta=-\bk_3 r+m_{\bk_3}\}, \quad m_{\bk_3}= (\bk_3-\bk_2)r_{p_2}-(b_3-b_2)+m_{\bk_2}= \lambda_{1}+\lambda_2+m(\bk_3-\bk_2).
\end{equation}
This determines the lift of the output point $\tilde q=(r_q, \theta_q)$. This is because $\theta_q=b_3-\bk_3 r_q+m_{k_3}$ and $\theta_q=b_1-\bk_1 r_q$, so it must be that $r_q=\frac{m_{\bk_3}+b_3-b_1}{\bk_3-\bk_1}$.  So 
\begin{equation}\label{eq: q from p1p2}
\tilde q=\left(r_q=\frac{\lambda_{1}+\lambda_{2}+(b_3-b_1)+(\bk_3-\bk_2)m}{\bk_3-\bk_1}, \ \theta_q=b_3-\bk_3 r_q+m_{\bk_3} = b_1-\bk_1 r_q\right).  
\end{equation} 
One can see that for each choice of $m\in \bZ^g$, there is a triangle in $u_m$.

Writing $m=(\bk_3-\bk_1)\tm+w$, with $\tm \in \bZ^g$ and the remainder $w \in I_{g, |\sk_3-\sk_1|}$, we see that 
\begin{equation}\label{eq: reminder}
r_q=\frac{\lambda_{1}+\lambda_{2}+(b_3-b_1)+(\bk_3-\bk_2)w}{\bk_3-\bk_1}+(\bk_3-\bk_2)\tm\equiv \frac{\lambda_{1}+\lambda_{2}+(b_3-b_1)+(\bk_3-\bk_2)w}{\bk_3-\bk_1},
\end{equation}
which gives the expression for $q$ in Equation \eqref{eq: triangle q}.
\end{proof}

\begin{remark} \label{rmk: triangle ordering} When the slopes of two linear Lagrangians $\ell$ and $\ell'$ are $\sk$ and $\sk'$, respectively, with $\sk\neq \sk'$, then degree of an element in the morphism $CF(\ell, \ell')$ between them is either $0$ (if $\sk<\sk'$) or $g$ (if $\sk>\sk'$) as in explained in Section \ref{subsec:mors_fiber}.  The degree of the output is the sum of the degrees of the inputs. So if $\bk_1<\bk_2$, then either $\bk_1<\bk_2<\bk_3$ or $\bk_3<\bk_1<\bk_2$, and if $\bk_1>\bk_2$, then $\bk_2<\bk_3<\bk_1$.  So in Equation \eqref{eq: triangle area}, $(\bk_2-\bk_1)(\bk_3-\bk_2)(\bk_3-\bk_1)$ is positive.  That is, there are three cases in total, and note that $\sk_3<\sk_1<\sk_2$ and $\sk_2<\sk_3<\sk_1$ can be obtained by cyclic rotating the labels from $\sk_1<\sk_2<\sk_3$.
\end{remark}

\begin{lemma}\label{lem: triangle area}
    With respect to the complexified K\"ahler form $\omega_\tau$, the area of the triangle $u_{m}$ in Lemma \ref{lem: u_m} is  \begin{equation}\label{eq: triangle area}
\int u_{m}^*\omega_{\tau}=\frac{S_m^T\tau S_m}{2(\bk_2-\bk_1)(\bk_3-\bk_2)(\bk_3-\bk_1)},
\end{equation}
where
\begin{equation}\label{eq: Sm}
S_m=(\bk_3-\bk_2)(\lambda_{1}+b_2-b_1)-(\bk_2-\bk_1)(\lambda_{2}+b_3-b_2)-(\bk_2-\bk_1)(\bk_3-\bk_2)m.
\end{equation}
\end{lemma}

\begin{proof}

Below we will also use the notation 
\[
\Delta r_{q1}:=-\Delta r_{1q}:=r_{p_1}-r_{q}, \ \Delta r_{12}:=-\Delta r_{21}:=r_{p_2}-r_{p_1}, \text{ and } \Delta r_{2q}:=-\Delta r_{q2}:=r_q-r_{p_2}.
\]
We have 
\begin{equation}\label{eq: Delta r1r2}
\begin{split}
   & \Delta r_{q1}=\frac{\lambda_{1}+b_2-b_1}{\bk_{2}-\bk_1}-\frac{\lambda_{1}+\lambda_{2}+(b_3-b_1)+(\bk_3-\bk_2)m}{\bk_3-\bk_1}=\frac{S_m}{(\bk_2-\bk_1)(\bk_3-\bk_1)},\\
   & \Delta r_{q2}=\frac{\lambda_{2}+b_3-b_2}{\bk_3-\bk_2}+m-\frac{\lambda_{1}+\lambda_{2}+(b_3-b_1)+(\bk_3-\bk_2)m}{\bk_3-\bk_1}=-\frac{S_m}{(\bk_3-\bk_2)(\bk_3-\bk_1)},\\
\end{split}
\end{equation}
where $S_m$ is given in Equation \eqref{eq: Sm}.
Then 
\begin{equation} \label{eq: Delta r12}
  \Delta r_{12} = \Delta r_{q2}-\Delta r_{1q}
     =\left(-\frac{1}{\bk_3-\bk_2}-\frac{1}{\bk_2-\bk_1}\right)
\frac{S_m}{\bk_3-\bk_1}
      = -\frac{S_m}{(\bk_3-\bk_2)(\bk_2-\bk_1)}
\end{equation}

The triangle $u_{\tm}$ is spanned by two vectors in $\bR^{2g}$,
\begin{equation}\label{eq: triangle sides}
 (\Delta r_{q1}, -\bk_1 \Delta r_{q1}) \quad \text{and}\quad
(\Delta r_{q2}, -\bk_3 \Delta r_{q2})=-\frac{\bk_3-\bk_1}{\bk_3-\bk_2}(\Delta r_{q1},-\bk_3\Delta r_{q1}),
\end{equation}
where the formulas for $\Delta r_{q1}$ and $\Delta r_{q2}$ are given in Equation \eqref{eq: Delta r1r2}.
To compute the area of this triangle, let us do a change of basis by considering vectors 
\begin{equation}
\begin{split}
&\vec{u}:=(\Delta r_{1q}, 0)=\left(\frac{S_m}{(\bk_2-\bk_1)(\bk_3-\bk_1)}, 0\right) \\
& \vec{v}:=(0, \Delta r_{1q})=\left(0,\frac{S_m}{(\bk_2-\bk_1)(\bk_3-\bk_1)}\right).
\end{split}
\end{equation}
Then the two vectors in Equation \eqref{eq: triangle sides} spanning the triangle are exactly
\begin{equation}
\vec u-\bk_1 \vec v\quad \text{and}  \quad -\frac{\bk_2-\bk_1}{\bk_3-\bk_2}(\vec u-\bk_3 \vec v).
\end{equation}
The area of this triangle is equal to the area of the triangle spanned by $\vec u$ and $\vec v$ multiplying the determinant of the change of basis matrix in the $2$-dimensional plane spanned by $(\vec u, \vec v)$, which is $\det\begin{bmatrix} 1& -\frac{\bk_2-\bk_1}{\bk_3-\bk_2}\\-\bk_1 &\bk_3\frac{\bk_2-\bk_1}{\bk_3-\bk_2} \end{bmatrix}= \frac{(\bk_3-\bk_1)(\bk_2-\bk_1)}{\bk_3-\bk_2}$. Note that this determinant is always positive due to \cite[Lemma 3.18]{AbouzaidThesis} as stated in Remark \ref{rmk: triangle ordering}.  The vectors $\vec u$ and $\vec v$ are orthogonal with respect to $\omega_\tau$, so the area of the triangle spanned by $\vec u$ and $\vec v$ is easier to compute, which is  \[\frac{1}{2}\Delta r_{1q}^T\tau\Delta r_{1q}=\frac{S_m^T\tau S_m}{2(\bk_3-\bk_1)^2(\bk_3-\bk_1)^2}.\]  Multiplying this by the determinant of the change of basis, i.e. $\frac{(\bk_3-\bk_1)(\bk_2-\bk_1)}{\bk_3-\bk_2}$, gives us Equation \eqref{eq: triangle area} in the statement of this lemma.
\end{proof}

Now we discuss the holonomy factor, which is \begin{equation}
\Upsilon(\varrho_1,\varrho_2; [u_m])=\Hol_{\nabla_{a_3}}(\partial_{p_2\to q}u_m)\circ \varrho_2\circ \Hol_{\nabla_{a_2}}(\partial_{p_1\to p_2}u_m) \circ \varrho_1 \circ \Hol_{\nabla_{a_1}}(\partial_{q\to p_1}u_m)\in \hom((\varepsilon_{a_1})_q, (\varepsilon_{a_3})_q),
\end{equation}
where $\partial_{q\to p_1}u_m$ denotes the path that is the portion of the boundary of $u_m$ in $\ell_{\sk_1,b_1}$ from $q$ to $p_1$, and similarly for $\partial_{p_1\to p_2}u_m$ and $\partial_{p_2\to q}u_m$.  Using the connection 1-form given in Equation \eqref{eq: connection 1-form}, we have 
\[
\Hol_{\nabla_{a_1}}(\partial_{q\to p_1}u_m)=
e^{-\int_{q}^{p_1}2\pi\bi a_1dr} =e^{-2\pi\bi a_1^T  \Delta r_{q1}}  \in \hom\big((\varepsilon_{a_1})_q, (\varepsilon_{a_1})_{p_1}\big).
\]
Similarly, $\Hol_{\nabla_{a_2}}(\partial_{p_1\to p_2}u_m)= e^{-2\pi\bi a_{2}^T \Delta r_{12}}\in \hom\big((\varepsilon_{a_2})_{p_1}, (\varepsilon_{a_2})_{p_2}\big)$ and $\Hol_{\nabla_{a_3}}(\partial_{p_2\to q}u_m)=e^{-2\pi\bi a_{3}^T \Delta r_{2q}}\in \hom\big((\varepsilon_{a_3})_{p_1}, (\varepsilon_{a_3})_{p_2}\big)$.  The composition of these three isomorphisms from parallel transports, along with
  $\varrho_{p_1} \in \hom((\varepsilon_{a_1})_{p_1}, (\varepsilon_{a_2})_{p_1})$ and $\varrho_{p_2}\in \hom((\varepsilon_{a_2})_{p_2}, (\varepsilon_{a_3})_{p_2})$ at the corners, gives 
\begin{equation}\label{eq: holonomy factor}
\Upsilon(\varrho_2, \varrho_1; [u_m])=\exp \left[-2\pi\bi\left(a_{3}^T\Delta r_{2q}+a_{2}^T \Delta r_{12} + a_{1}^T\Delta r_{q1}\right)\right]\varrho_{p_2}\varrho_{p_1}.
\end{equation}
Note that in the above expression, the order of the composition does not matter because we have a rank 1 local system, so it is just a multiplication of complex numbers. 
We can use Equations \eqref{eq: Delta r1r2} and \eqref{eq: Delta r12} to express $\Upsilon(\varrho_2, \varrho_1; [u_m])$ in Equation \eqref{eq: holonomy factor} as 
\begin{equation}
\begin{split}
\Upsilon(\varrho_2, \varrho_1; [u_m])=&\exp \left[-2\pi\bi\Big(a_{3}^T(\bk_2-\bk_1)- a_{2}^T(\bk_3-\bk_1)+ a_{1}^T(\bk_3-\bk_2)\Big)\frac{S_m}{(\bk_2-\bk_1)(\bk_3-\bk_2)(\bk_3-\bk_1)}\right]\varrho_{p_2}\varrho_{p_1}\\
=&\exp \left[2\pi\bi\Big((a_{3}-a_{2})^T(\bk_2-\bk_1)-(a_{2}-a_{1})^T(\bk_3-\bk_2)\Big)\frac{-S_m}{(\bk_2-\bk_1)(\bk_3-\bk_2)(\bk_3-\bk_1)}\right]\varrho_{p_2}\varrho_{p_1}.
\end{split}
\end{equation}

Putting everything together, we have the following product formula.

\begin{proposition}\label{prop:mu 2 on A side}
Given 
 \begin{eqnarray*}
 & p_1\equiv \left( \frac{\lambda_1+b_2-b_1}{\bk_2-\bk_1}, \frac{-\bk_1\lambda_1+\bk_2 b_1-\bk_1b_2}{\bk_2-\bk_1}\right) \in \ell_{\bk_1, b_1}\cap \ell_{\bk_2, b_2},  & \lambda_1\in \{0,\ldots, |\bk_2-\bk_1|-1\}^g,\\
 & p_2\equiv \left( \frac{\lambda_2+b_3-b_2}{\bk_3-\bk_2}, \frac{-\bk_2\lambda_2+\bk_3 b_2-\bk_2b_3}{\bk_3-\bk_2}\right) \in \ell_{\bk_2,b_2}\cap \ell_{\bk_3, b_3},& \lambda_2\in \{0,\ldots, |\bk_3-\bk_2|-1\}^g,
 \end{eqnarray*}
 $\varrho_{p_1}\in \hom((\varepsilon_{a_1})_{p_1}, (\varepsilon_{a_2})_{p_1})\cong \bC$, and  $\varrho_2\in \hom((\varepsilon_{a_2})_{p_2}, (\varepsilon_{a_3})_{p_2})\cong\bC$, the product is 
\begin{equation}\label{eq:mu2_Aside}
\begin{split}
 & \mu^2(\varrho_{p_2}  p_2, \varrho_{p_1} p_1)\\
 &=\varrho_{p_2}\varrho_{p_1}\sum \limits_{m\in \bZ^g} e^{\pi \bi \frac{S_m^T\tau S_m}{(\bk_2-\bk_1)(\bk_3-\bk_2)(\bk_3-\bk_1)}} e^{2\pi\bi\Big((a_{3}-a_{2})^T(\bk_2-\bk_1)-(a_{2}-a_{1})^T(\bk_3-\bk_2)\Big)\frac{-S_m}{(\bk_2-\bk_1)(\bk_3-\bk_2)(\bk_3-\bk_1)}} q_m\\
 &=\varrho_{p_2}\varrho_{p_1}\sum \limits_{w\in I_{g, |\sk_3-\sk_1|}}D^w q_{w}.
 \end{split}
\end{equation}
where
\begin{equation} \label{eq: Dw}
\begin{split}
D^w &=\sum \limits_{\tm\in \bZ^g} e^{\pi \bi \left(\tm-\frac{S_{w}}{(\bk_2-\bk_1)(\bk_3-\bk_2)(\bk_3-\bk_1)}\right)^T(\bk_2-\bk_1)(\bk_3-\bk_2)(\bk_3-\bk_1)\tau\left(\tm- \frac{S_{w}}{(\bk_2-\bk_1)(\bk_3-\bk_2)(\bk_3-\bk_1)}\right)}  \\
& \hspace{18em}
e^{2\pi\bi\Big((a_{3}-a_{2})^T(\bk_2-\bk_1)-(a_{2}-a_{1})^T(\bk_3-\bk_2)\Big)\left(\tm-\frac{S_{w}}{(\bk_2-\bk_1)(\bk_3-\bk_2)(\bk_3-\bk_1)}\right)}\\
& = \vartheta\left[-\frac{S_{w}}{(\bk_2-\bk_1)(\bk_3-\bk_2)(\bk_3-\bk_1)},\left((a_{3} - a_{2})^T(\bk_2-\bk_1)-(a_{2}-a_{1})^T(\bk_3-\bk_2) \right)\right]\\
& \hspace{30em}\Big((\bk_3-\bk_2)(\bk_2-\bk_1)(\bk_3-\bk_1)\tau, 1\Big).
 \end{split}
 \end{equation}
\end{proposition}
\begin{proof}Simply substitute Equations \eqref{eq: triangle area} and \eqref{eq: holonomy factor} into the definition, i.e. Equation \eqref{eq: product def}. To get the 2nd equality above, we write $m=(\bk_3-\bk_1)\tm+w$, with $\tm \in \bZ^g$ and the reminder $w \in I_{g, |\sk_3-\sk_1|}$.   Then as pointed out in Equation \eqref{eq: reminder}, $q_m=q_w$ when $m$ and $w$ are related in this way.  Also  note that $\frac{-S_m}{(\bk_2-\bk_1)(\bk_3-\bk_2)(\bk_3-\bk_1)}=\tm- \frac{S_{w}}{(\bk_2-\bk_1)(\bk_3-\bk_2)(\bk_3-\bk_1)}$.  
\end{proof}

The mirror functor $\Phi_\tau:\Fuk(\bT^{2g},\omega_\tau)\to \Coh(V_\tau)$ maps objects  $\hat \ell_{\sk,[z]}$ to $\cL_{\sk,[z]}$, for $\sk\in\bZ$ and $z=a+\tau b\in \bC$.  At first order, as pointed out in Equation \eqref{eq:mirror morphisms}, $\Phi_\tau^1$ is an isomorphism on the morphism complexes 
\begin{equation}\label{eq: Phi^1}
\Phi_\tau^1:CF^*(\hat \ell_{\sk_1, [z_1]}, \hat \ell_{\sk_2, [z_2]})\xrightarrow{\cong} \Ext^*(\cL_{\sk_1,[z_1]}, \cL_{\sk_2,[z_2]})=H^*(V_\tau, \cL_{\sk_2-\sk_1, [z_2-z_1]}).
\end{equation}
Given a generator $p_{\sk_1, b_1, \sk_2,b_2}(\lambda)\in \ell_{\sk_1,b_1}\cap \ell_{\sk_2,b_2}$ of the Floer cochain complex $CF^*(\hat \ell_{\sk_1, [z_1]}, \hat \ell_{\sk_2, [z_2]})$ whose coordinates are given in Equation \eqref{eq:T4_Lag_intersec}, as explained in Section \ref{subsec:mors_fiber}, it is of degree $0$ if $\sk_1<\sk_2$ and of degree $g$ if $\sk_1>\sk_2$.  It's image under $\Phi_\tau^1$ is the section $s_{\sk_2-\sk_1, z_2-z_1, \lambda}\in H^0(V_\tau, \sk_2-\sk_1, [z_2-z_1])=\Ext^0(\cL_{\sk_1,[z_1]}, \cL_{\sk_2,[z_2]})$ if $\sk_1<\sk_2$, or the dual $s^{\sk_1-\sk_2, z_1-z_2, \lambda}\in \Ext^g(\cL_{\sk_1,[z_1]}, \cL_{\sk_2,[z_2]})$ if $\sk_1>\sk_2$.  For the reminder of the section, we show the above isomorphism in Equation \eqref{eq: Phi^1} is compatible with the product $\mu^2$, more precisely stated in Proposition \ref{prop: mu 2 commute} below that $\mu^2$ commutes with the composition of the extension groups under the mirror functor $\Phi_\tau$.  We first state a lemma that is a special case of Proposition \ref{multiplication}.  
\begin{lemma}\label{cor: multn_formula_theta} When $\sk_1<\sk_2<\sk_3$, we have the following multiplication formula for the theta functions
\begin{multline}\label{eq:multiplication 123}\vartheta\left[\frac{\lambda_{2}+b_3-b_2}{\bk_3-\bk_2}, a_{3}-a_{2}\right]\left((\bk_3-\bk_2)\tau, x^{\bk_3-\bk_2}\right) \cdot \vartheta\left[\frac{\lambda_{1}+b_2-b_1}{\bk_2-\bk_1}, a_{2}-a_{1}\right]\left((\bk_2-\bk_1)\tau, x^{\bk_2-\bk_1}\right)\\
=\sum_{w\in I_{g, \sk_3-\sk_1}} D^w\cdot \vartheta\left[\frac{\lambda_1+\lambda_2+b_3-b_1+(\bk_3-\bk_2)w}{\bk_3-\bk_1}, a_{3}-a_{1} \right]\left((\bk_3-\bk_1)\tau, x^{\bk_3-\bk_1}\right)
\end{multline}
with the $D^w$ given in Equation \eqref{eq: Dw}.  Equivalently, this can be written as 
\begin{equation}\label{eq: product formula section}s_{\sk_3-\sk_2, z_3-z_2, \lambda_2}\otimes s_{\sk_2-\sk_1, z_2-z_1, \lambda_1} = \sum_{w\in I_{g,\sk_3-\sk_1}} D^{w} s_{\sk_3-\sk_1, z_3-z_1, \lambda_w},
\end{equation}
where $I_{g, \sk_3-\sk_1}\ni \lambda_w\equiv \lambda_1+\lambda_2+(\sk_3-\sk_2)w \mod (\sk_3-\sk_1)\bZ^g$.
\end{lemma}

\begin{proof}
Equation \eqref{eq:multiplication 123} can be obtained by substituting $\bk'$ in Proposition \ref{multiplication} by  $\bk_3-\bk_2$, $\bk''$ by $\bk_2-\bk_1$, $c'$ by $\lambda_{2}+b_3-b_2$, $c''$ by $\lambda_{1}+b_2-b_1$, $d'$ by $a_{3}-a_{2}$, $d''$ by $a_{2}-a_{1}$, $x'$ by $x^{\bk_3-\bk_2}$, and $x''$ by $x^{\bk_2-\bk_1}$.

 The theta function $\vartheta\left[\frac{\lambda_{1}+b_2-b_1}{\bk_2-\bk_1}, a_{2}-a_{1}\right]\left((\bk_2-\bk_1)\tau, x^{\bk_2-\bk_1}\right)$ defines a section  $s_{\sk_2-\sk_1, z_2-z_1, \lambda_1}$, and the theta function $\vartheta\left[\frac{\lambda_{2}+b_3-b_2}{\bk_2-\bk_1}, a_{3}-a_{2}\right]\left((\bk_3-\bk_2)\tau, x^{\bk_3-\bk_2}\right)$ defines a section  $s_{\sk_3-\sk_2, z_3-z_2, \lambda_2}$. 
 
 Similarly,
$\vartheta\left[\frac{\lambda_1+\lambda_2+b_3-b_1+(\sk_3-\sk_2)w}{\bk_3-\bk_1}, a_{3}-a_{1}\right]\left((\bk_3-\bk_1)\tau, x^{\bk_3-\bk_1}\right)$ defines a section  $s_{\sk_3-\sk_1, z_3-z_1, \lambda_w}$, where the above choice for $\lambda_w$ regarding modulo by $(\sk_3-\sk_1)\bZ^g$ is justified because of Equation \eqref{eqn:integral-shift}, that adding an element of $(\sk_3-\sk_1)\bZ^g$ to $\lambda$ leaves the above theta function invariant.  Therefore, Equation \eqref{eq: product formula section} is equivalent to \eqref{eq:multiplication 123}.
\end{proof}

\begin{proposition}\label{prop: mu 2 commute} The following diagram commute
\[\begin{tikzcd}
 CF^{j_2}(\hat \ell_{\sk_2,[z_2]},\hat \ell_{\sk_3,[z_3]}) \otimes CF^{j_1}(\hat \ell_{\sk_1,[z_1]},\hat \ell_{\sk_2,[z_2]}) \arrow[r, "\mu^2"]\arrow[d, "\Phi_\tau\otimes \Phi_\tau"'] & CF^{j_1+j_2}(\hat \ell_{\sk_1,[z_1]},\hat \ell_{\sk_3,[z_3]})\arrow[d,"\Phi_\tau"] \\
 \Ext^{j_2}(\cL_{\sk_2,[z_2]},\cL_{\sk_3,[z_3]}) \otimes \Ext^{j_1}(\cL_{\sk_1,[z_1]},\cL_{\sk_2,[z_2]}) \arrow[r, "\otimes"]& \Ext^{j_1+j_2}(\cL_{\sk_1,[z_1]},\cL_{\sk_3,[z_3]})
 \end{tikzcd}.
 \]
\end{proposition}
\begin{proof}  As pointed out in Remark \ref{rmk: triangle ordering}, there are three cases: $\sk_1<\sk_2<\sk_2$ in which case the degrees in the above diagram satisfies $j_1=j_2=0$, $\sk_2<\sk_3<\sk_1$ in which case $j_1=g$ and $j_2=0$, or $\sk_3<\sk_1<\sk_2$ in which case $j_1=0$ and $j_2=g$. 

Let us first discuss the first case, i.e. when $\sk_1<\sk_2<\sk_3$.  Consider generators $p_1\in \ell_{\sk_1, b_1}\cap \ell_{\sk_2,b_2}$ and $p_2\in \ell_{\sk_2, b_2}\cap \ell_{\sk_3,b_3}$ of the respective Floer complexes in the form stated in Proposition \ref{prop:mu 2 on A side}.  Then 
$\Phi_\tau^1(p_1)=s_{\sk_2-\sk_1, z_2-z_1,\lambda_1}$ and $\Phi_\tau^1(p_2)=s_{\sk_3-\sk_2,z_3-z_2, \lambda_2}$. 
Proposition \ref{prop:mu 2 on A side} gives that
\begin{equation}\label{eq:Phi mu2}
\Phi_\tau^1(\mu^2(p_1,p_2))=\Phi_\tau^1\left(\sum_{w\in I_{g, \sk_3-\sk_1}}D^wq_w\right)=\sum_{w\in I_{g, \sk_3-\sk_1}}D^w\Phi_\tau^1(q_w)
=\sum_{w\in I_{g, \sk_3-\sk_1}}D^ws_{\sk_3-\sk_1, z_3-z_1, \lambda_w},
\end{equation}
where $\lambda_w$ is the same as the one given in Lemma \ref{cor: multn_formula_theta}.
Combining Equations  \eqref{eq: product formula section} and \eqref{eq:Phi mu2}, we get 
$\Phi_\tau^1(\mu^2(p_1,p_2))= \Phi_\tau^1(p_2)\otimes \Phi_\tau^1(p_1).$

Below we explain that the other two cases follows from Serre duality, i.e. the morphism complexes satisfy $\Hom^*(A_1, A_2)=\Hom^{g-*}(A_2,A_1)^\vee$ for objects $A_1$  and $A_2$.  Note that Serre duality for $\Ext^*$ is stated in Equation \eqref{eq:B-model Serre} and for $CF^*$ is explained in Section \ref{subsec:mors_fiber}.

When $\sk_2<\sk_3<\sk_1$, we are considering the product 
\[\Ext^{0}(\cL_{\sk_2,[z_2]},\cL_{\sk_3,[z_3]}) \otimes \Ext^{g}(\cL_{\sk_1,[z_1]},\cL_{\sk_2,[z_2]}) \to  \Ext^{g}(\cL_{\sk_1,[z_1]},\cL_{\sk_3,[z_3]}).
\]
By Serre duality, we can replace $\Ext^g(\cL_{\sk_1,[z_1]},\cL_{\sk_2,[z_2]})$ above by $\Ext^0(\cL_{\sk_2,[z_2]},\cL_{\sk_1,[z_1]})^\vee$, and similarly for  $\Ext^{g}(\cL_{\sk_1,[z_1]},\cL_{\sk_3,[z_3]})$.  In turn, the above map induces the following equivalent computation (also see Lemma \ref{lem:C_s_dual_calc} and Remark \ref{rmk:dual product Serre} for a more detailed formula of the product involving the duals)  
\[\Ext^0(\cL_{\sk_3,[z_3]}, \cL_{\sk_1,[z_1]})\otimes \Ext^0(\cL_{\sk_2,[z_2]},\cL_{\sk_3,[z_3]})
\to \Ext^0(\cL_{\sk_2,[z_2]},\cL_{\sk_1,[z_1]}).
\]
Since $\sk_2<\sk_3<\sk_1$, this is the same as the first case and this product matches the Floer product 
\[ CF^{0}(\hat \ell_{\sk_3,[z_3]},\hat \ell_{\sk_1,[z_1]}) \otimes CF^{0}(\hat \ell_{\sk_2,[z_2]},\hat \ell_{\sk_3,[z_3]}) \to CF^{0}(\hat \ell_{\sk_2,[z_2]},\hat \ell_{\sk_1,[z_1]}), 
\]
which similarly dualizes to 
\[ CF^{0}(\hat \ell_{\sk_2,[z_2]},\hat \ell_{\sk_3,[z_3]}) \otimes CF^{g}(\hat \ell_{\sk_1,[z_1]},\hat \ell_{\sk_2,[z_2]}) \to CF^{g}(\hat \ell_{\sk_1,[z_1]},\hat \ell_{\sk_3,[z_3]}). 
\] 

Similarly, when $\sk_3<\sk_1<\sk_2$, the product computation
\[\Ext^{g}(\cL_{\sk_2,[z_2]},\cL_{\sk_3,[z_3]}) \otimes \Ext^{0}(\cL_{\sk_1,[z_1]},\cL_{\sk_2,[z_2]}) \to  \Ext^{g}(\cL_{\sk_1,[z_1]},\cL_{\sk_3,[z_3]}).
\]
can be reduced to that of
\[\Ext^{0}(\cL_{\sk_1,[z_1]},\cL_{\sk_2,[z_2]}) \otimes \Ext^{0}(\cL_{\sk_3,[z_3]},\cL_{\sk_1,[z_1]}) \to  \Ext^{0}(\cL_{\sk_3,[z_3]},\cL_{\sk_2,[z_2]}),
\]
and similarly for $CF^*$.
\end{proof}

When $\sk_1<\sk_2 < \sk_3$, and $j_1=j_2=j_3=0$, the commutative diagram in Proposition \ref{prop: mu 2 commute} can be rewritten as
\[\begin{tikzcd}
 HF^0(\hat \ell_{\mathsf{0}, [0]},\hat \ell_{\sk_3-\sk_2, [z_3-z_2]}) \otimes HF^0(\hat \ell_{\szero, [0]},\hat \ell_{\sk_2-\sk_1, [z_2-z_1]}) \arrow[r, "\mu^2"]\arrow[d, "\Phi_\tau\otimes \Phi_\tau"'] & HF^0(\hat \ell_{\szero,[0]}, \hat \ell_{\sk_3-\sk_1, [z_3-z_1]})\arrow[d,"\Phi_\tau"] \\
 H^0(V_\tau, \cL_{\sk_3-\sk_2, [z_3-z_2]}) \otimes H^0(V_\tau, \cL_{\sk_2-\sk_1, [z_2-z_1]}) \arrow[r, "\otimes"]& H^0(V_\tau, \cL_{\sk_3-\sk_1, [z_3-z_1]})
 \end{tikzcd}.
 \]
Let $\hat{\ell}_{\sk} := \hat{\ell}_{\sk, [0]}$. Then for any $\sk, \sk' >0$, we have a commutative diagram
\[
\begin{tikzcd}
 HF^0(\hat \ell_{\szero},\hat \ell_{\sk}) \otimes HF^0(\hat \ell_{\szero} ,\hat \ell_{\sk'}) \arrow[r, "\mu^2"]\arrow[d, "\Phi_\tau\otimes \Phi_\tau"'] & HF^0(\hat \ell_{\szero}, \hat \ell_{\sk+\sk'})\arrow[d,"\Phi_\tau"] \\
 H^0(V_\tau, \cL_\tau^{\otimes \sk}) \otimes H^0(V_\tau, \cL_\tau^{\otimes \sk'}) \arrow[r, "\otimes"]& H^0(V_\tau, \cL_\tau^{\otimes(\sk +\sk')})
 \end{tikzcd}.
 \]

For each $\tau\in \cH$, we define a graded vector space over $\bC$:
\begin{equation}\label{eqn:widecheck-S-tau-section 4}
\widecheck{S}_\tau := \bigoplus_{\sk\geq 0} HF^0(\hat \ell_\szero, \hat \ell_{\sk}).
\end{equation}
where $\hat{\ell}_{\sk}\subset \bT^{2g}$ and $\deg s =\sk$ if $s\in HF^0(\hat\ell_\szero, \hat\ell_{\sk})$.  
Let 
\begin{equation}
S_\tau =\bigoplus_{d\geq 0} H^0 (V_\tau, \cL_\tau^{\otimes d})
\end{equation}
be the ring of sections defined in Section \ref{sec:categories-of-B-branes}. We obtain the following analogue of 
Theorem 1.1 of \cite{AbouzaidThesis}. 

\begin{theorem}\label{thm:ring-isomorphism} $(\widecheck{S}_\tau, \mu^2)$ is a graded commutative $\bC$-algebra, and 
$$
\Phi_\tau: (\widecheck{S}_\tau, \mu^2) \lra (S_\tau, \otimes)
$$
is an isomorphism of graded commutative $\bC$-algebras. 
\end{theorem}

\begin{remark}
The ring $(\widecheck{S}_\tau, \mu^2)$ has been studied in \cite{Zaslow05, Aldi-Zaslow, Aldi09}.
\end{remark}

\subsection{Categories of A-branes and homological mirror symmetry}\label{sec:cateogries-of-A-branes} 
Given $\tau\in \cH_\tau$,  we recall from Section \ref{sec:A-branes}, some categories of A-branes in $(\bT^{2g},\omega_\tau)$.
We consider strictly unobstructed Lagrangians so that the differential of the Lagrangian Floer complex squared to zero and the Lagrangian Floer cohomology groups $HF^*$ are defined. 
Moreover, we consider affine Lagrangians (which are in particular strictly unobstructed) so that $HF^*$ is defined not only over a Novikov ring/field but is defined over $\bC$, so that we may study global mirror symmetry over the moduli of complex structures on $V_\tau$.

\begin{itemize}
\item Let $\Fuk_{\aff}(\bT^{2g},\omega_\tau)$ be the category whose objects are affine Lagrangians in $(\bT^{2g},\omega_\tau)$ equipped with flat $U(1)$-connections and their shifts, and 
$\Hom(\hat{\ell}[j], \hat{\ell}'[j']) = HF^*(\hat{\ell}, \hat{\ell}')[j'-j]$, the Lagrangian 
intersection Floer cohomology group with complex coefficients, which is a graded
vector space over $\bC$. 
\item Let $\cA_\tau$ be the full subcategory of $\Fuk_{\aff}(\bT^{2g},\omega_\tau)$ whose objects
are $\{ \hat{\ell}_{\bk,[v]}[j] \mid \bk, j\in \bZ, [v]\in V_\tau^+\}$, i.e.  SYZ mirrors of objects in the category $\cB_\tau$ defined in Section \ref{sec:categories-of-B-branes}.

\item Given any line bundle $L \in \Pic(V_\tau)$ with $c_1(L)=\omega_{V_\tau}$, let $\cA_L$ be the full subcategory of $\Fuk_{\aff}(\bT^{2g}, \omega_\tau)$ whose objects are SYZ mirrors of objects in $\cB_L$ defined in Section \ref{sec:categories-of-B-branes}.
\end{itemize}

We have the following version of homological mirror symmetry at the level of cohomologies. 
\begin{theorem} \label{thm:HMS}
For every $\tau\in \cH$, the functor
\begin{equation} \label{eqn:A-to-B}
\Mir_\tau: \cA_\tau \lra \cB_\tau
\end{equation}
sending $\hat{\ell}_{\sk, [z]}[j]$ to its SYZ mirror $\cL_{\sk, [z]}[j]= \cL^{\otimes \sk}\otimes \bbL_{[z]} [j]$, where
$\sk, j\in \bZ$ and $[z]\in V_\tau^+$, is an equivalence of categories. Therefore, we have 
a fully faithful embedding
\begin{equation}\label{eqn:A-to-DCoh}
\Phi_\tau : \cA_\tau \longrightarrow D^b \Coh(V_\tau)
\end{equation}
of cohomological categories. In particular, the product structures match under $\Phi_\tau$: 
\begin{equation}
\Phi_\tau \circ \mu^2_{\cA_\tau} = \mu^2_{D^b\Coh(V_\tau)} \circ \Phi_\tau.
\end{equation}
\end{theorem}
\begin{proof} This follows from \eqref{eq:mirror morphisms}, \eqref{eqn:finishing}, and Proposition \ref{prop: mu 2 commute}.
\end{proof}

Since $\cA_L$ and $\cB_L$ are full subcategories of $\cA_\tau$ and $\cB_\tau$, respectively, we have the following corollary of Theorem \ref{thm:HMS}.
\begin{corollary} \label{cor:coreHMS} 
For every $\tau\in \cH$, and any $L =\cL\otimes \bbL_{[z]} \in \Pic^{\mathsf{1}}(V_\tau)$, 
the functor 
\begin{equation} \label{eqn:A-to-B-L-section 4}
\Mir_{\tau,L}: \cA_L \lra \cB_L
\end{equation}
sending $\hat{\ell}_{\sk,[kz]}$ to 
$L^{\otimes k} = \cL_{\sk, [kz]}$
is an equivalence of categories. 
Therefore, we have
a fully faithful embedding
\begin{equation}\label{eqn:AL-to-DCoh-section 4}
\Phi_{\tau,L} : \cA_L \longrightarrow D^b \Coh(V_\tau)
\end{equation}
of cohomological categories. In particular, the product structures match under $\Phi_{\tau, L}$: 
\begin{equation}
\Phi_{\tau, L}\circ \mu^2_{\cA_L} = \mu^2_{D^b\Coh(V_\tau)} \circ \Phi_{\tau,L}.
\end{equation}
\end{corollary}

Recall from Section \ref{sec:categories-of-B-branes} that $\langle \cB_L\rangle = D^b\Coh(V_\tau)$, so Corollary \ref{cor:coreHMS} is a ``core HMS"  
\cite[Definition 1.6]{PS23} at the cohomological level.

\subsection{Example of a product computation involving a vertical Lagrangian}\label{sec:prod_vert} 

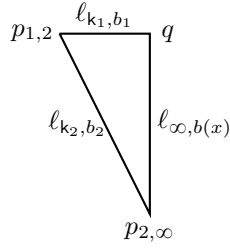
\begin{figure}[ht]
\centering
\begin{tikzpicture}[scale=1.2]
\draw [thick](0,0)--(1,0)--(1,-2)--(0,0);
\node at (-0.3,0) {$p_{1,2}$};
\node at (1.2,0) {$q$};
\node at (1, -2.2) {$p_{2,\infty}$};
\node at (0.5,0.2) {$\ell_{\bk_1,b_1}$};
\node at (0.2, -1) {$\ell_{\bk_2,b_2}$};
\node at (1.5,-1) {$\ell_{\infty,{b(x)}}$};
\end{tikzpicture}
    \caption{A triangle contributing to $\cM(q, p_{2,\infty}, p_{1,2})$.}
    \label{fig: straight triangle}
\end{figure}

The goal of this section is just to do an example of computing a product that involves a vertical Lagrangian of the form defined in Section \ref{sec: fiber vertical Lagrangian}, in particular we compute the product
\begin{equation}\mu^2: CF(\hat{\ell}_{\sk_2,[a_2 + \tau b_2]}, \hat{\ell}_{\infty,[a(x) + \tau b(x)]})\otimes CF(\hat{\ell}_{\sk_1,[a_1 + \tau b_1]}, \hat{\ell}_{\sk_2,[a_2 + \tau b_2]})\to CF(\hat{\ell}_{\sk_1,[a_1 + \tau b_1]}, \hat{\ell}_{\infty,[a(x) + \tau b(x)]}),
\end{equation}
where $\sk_1<\sk_2$, $b(x):=\frac{1}{2\pi}\Omega^{-1}\log |x|$, and $a(x):=-\frac{1}{2\pi}\arg(x)-Bb(x)$ defined by 
\begin{equation}\label{eq: vertical triangle product}
\mu^2(\varrho_{p_{2, \infty}}p_{p_{2, \infty}},  \varrho_{p_{1,2}}p_{p_{1,2}}) =\sum \limits_{\substack{q\in \ell_{\bk_1,b_1}\cap \ell_{\infty,b(x)}\\ \ind([u])=0}} \# \cM (q, p_{2,\infty}, p_{1,2}; [u]) e^{2\pi\bi \int u^*\omega_\tau} \Upsilon(\varrho_{p_{2,\infty}}, \varrho_{p_{1,2}};[u])q.
\end{equation}

\begin{remark}\label{rem:SYZ_signs} This A-side product we are computing in this section is mirror to the composition 
$\mathcal L_{\bk_1, [a_1+\tau b_1]} \to \mathcal L_{\bk_2, [a_2+\tau b_2]} \to \mathcal O_x.$
This composition can be represented by $\vartheta\left[\frac{\lambda_{p_{1,2}}+b_2-b_1}{\bk_2-\bk_1}, a_{2}-a_{1}\right]\left((\bk_2-\bk_1)\tau, x^{\bk_2-\bk_1}\right)$, which is a section in $\hom(\mathcal L_{\bk_1, [a_1+\tau b_1]}, \mathcal L_{\bk_2, [a_2+\tau b_2]})=H^0(V_\tau, \mathcal L_{\bk_2-\bk_1, [a_2-a_1+ \tau (b_2-b_1)]})$ together with evaluation at $x.$
The skyscraper sheaf $\mathcal O_{x}$ (or $\mathcal O\to \mathcal L\to \mathcal O_{x}$) is mirror to $\hat{\ell}_{\infty, [a(x)+\tau b(x)]}=\hat{\ell}_{\infty, [-z]}$, where $z=-a(x)-\tau b(x)$, $b(x):=\frac{1}{2\pi}\Omega^{-1}\log |x|$, and $a(x):=-\frac{1}{2\pi}\arg(x)-Bb(x)$, so  $x=e^{2\pi\bi z}$. 
\end{remark}

Similar to what we did in Section \ref{sec: fiber product structure}, we first describe the holomorphic triangles in the moduli space above, then we discuss the computation of the area of the triangles with respect to the complexified  K\"{a}her form, and finally we compute the holonomy factor.  

To describe the triangles, we again  consider lifts in the universal cover $\bR^{2g}$.  To do so, we can pick an arbitrary lifts for $\ell_{\bk_1,b_1}$ and $p_{1,2}$, and then trace through the triangle to determine the lifts for $\ell_{\sk_2,b_2}$, $p_{2,\infty}$, $\ell_{\infty,b(x)}$, and $q$.  We can arbitrarily pick the lift of $\ell_{\bk_1,b_1}=\{(r,\theta): \theta\equiv b_1-\bk_1 r\}$ to be
\begin{equation}\label{eq: lift l1}
   \tilde  \ell_{\bk_1,b_1}=\{(r,\theta): \theta=b_1-\bk_1 r\}.
\end{equation} Then the lift of  $\ell_{\sk_2, b_2}$ is 
\begin{equation}\label{eq: lift l2}
\tilde  \ell_{\bk_2,b_2}=\{(r,\theta): \theta=b_2-\bk_2 r+m_{\bk_2}\}
\end{equation} for some $m_{\bk_2}\in \bZ^g$. 
We can choose the lift of the input point $p_{1,2}$ along $\tilde \ell_{\sk_1,b_1}$ to be 
\begin{equation} \label{eq: lift p12}
    \tilde p_{1,2}=\left(r_{p_{1,2}}=\frac{\lambda_{p_{1,2}}+b_2-b_1}{\bk_{2}-\bk_1}+m, \ \theta_{p_{1,2}}=b_1-\bk_1 r_{p_{1,2}}=b_2-\bk_2 r_{p_{1,2}} +m_{\bk_2}\right),  \lambda_{p_{1,2}}\in \{0,\ldots, |\bk_2-\bk_1|-1\}^g,
\end{equation}
hence $m_{\bk_2} = (\bk_2-\bk_1)r_{p_{1,2}}-(b_2-b_1)=\lambda_{p_{1,2}}+(\bk_2-\bk_1)m$.
This determines the lifts of $p_{2,\infty}$ and $q$ to be
\begin{equation}
\tilde p_{2,\infty}=\Big(b(x), \ b_2-\bk_2 b(x)+m_{\bk_2} \Big)
\end{equation}
\begin{equation}
\tilde q=\Big(b(x), \ b_1-\bk_1 b(x) \Big).
\end{equation}
So, we can see that for each $m\in \bZ^g$, there is a triangle $u_m$ in the moduli space $\cM (q, p_{2,\infty}, p_{1,2})$ that contributes to the product in Equation \eqref{eq: vertical triangle product}.

Now we compute the area $\int u_m^*\omega_\tau$ of reach triangle $u_m$ with respect to the complexified K\"{a}hler form.  Denote by
\[
\Delta:=\Delta_{\bk_1,b_1 \bk_2,b_2, x, m}:=b(x)-m-\frac{\lambda_{p_{1,2}}+b_2-b_1}{\bk_2-\bk_1}.
\]
The triangle is half of the parallelogram spanned by 
\[
\left(\Delta, -\bk_2\Delta\right)\text{ and } \Big( \Delta, -\bk_1\Delta \Big).
\]
Like we did in the proof of Lemma \ref{lem: triangle area}, in the 2-plane spanned by the above two vectors, we make a change of variable from these two vectors to $(\Delta, 0)$ and $(0, \Delta)$.  Then the area $\int u_m^*\omega_\tau$ of the triangle $u_m$ is equal to  $\frac{1}{2}\Delta^T\tau \Delta$ multiplying the determinant of the change of variable matrix $\det \begin{pmatrix}1& 1 \\ -\bk_2 & -\bk_1\end{pmatrix}=\bk_2-\bk_1$. So the triangle's area is
\begin{equation} \label{eq: vertical triangle area}
\int u_m^*\omega_\tau= \frac{1}{2}(\bk_2-\bk_1)\left(b(x)-m-\frac{\lambda_{p_{1,2}}+b_2-b_1}{\bk_2-\bk_1} \right)^T\tau\left(b(x)-m-\frac{\lambda_{p_{1,2}}+b_2-b_1}{\bk_2-\bk_1}\right)
\end{equation}

Next we compute the holonomy term 
\begin{equation}\label{eq: honolomy factor vertical triangle def}
\Upsilon_(\varrho_{p_{2,\infty}}, \varrho_{p_{1,2}};[u_m])=\Hol_{\nabla_{a(x)}}(\partial_{p_{2,\infty}\to q}u_m)\circ \varrho_{p_2,\infty}\circ \Hol_{\nabla_{a_{2}}}(\partial_{p_{1,2}\to p_{2,\infty}}u_m) \circ \varrho_{p_{1,2}} \circ \Hol_{\nabla_{a_{1}}}(\partial_{q\to p_{1,2}}u_m).
\end{equation}
Using the connection 1-form given in Equation \eqref{eq: connection 1-form}, i.e. $\nabla_a=d+2\pi\bi adr$, we have 
\[
\Hol_{\nabla_{a_1}}(\partial_{q\to p_{1,2}}u_m)=
e^{-\int_{q}^{p_{1,2}}2\pi\bi a_1dr} =\exp \left[- 2\pi\bi a_{1}^T\left(\frac{\lambda_{p_{1,2}}+b_2-b_1}{\bk_2-\bk_1}+m-b(x)\right)\right], 
\]
and 
\[
\Hol_{\nabla_{a_2}}(\partial_{p_{1,2}\to p_{2,\infty}}u_m)=
e^{-\int_{p_{1,2}}^{p_{2,\infty}}2\pi\bi a_2dr} =\exp \left[ - 2\pi\bi a_{2}^T\left(b(x)-m-\frac{\lambda_{p_{1,2}}+b_2-b_1}{\bk_2-\bk_1}\right)\right] 
\]
Using the connection 1-form given in Equation \eqref{eq: connection 1-form vertical},i.e. $\nabla_{a(x)}=d+2\pi\bi a(x)d\theta$, we have
\[
\Hol_{\nabla_{a(x)}}(\partial_{p_{2,\infty}\to q} u_m)=
e^{-\int_{p_{2,\infty}}^{q}2\pi\bi a(x)d\theta} =\exp \left[ -2\pi\bi a(x)^T (\bk_2-\bk_1) \left(b(x)-m-\frac{\lambda_{p_{1,2}}+b_2-b_1}{\bk_2-\bk_1}\right)\right].
\]
Putting them together and continuing with Equation \eqref{eq: honolomy factor vertical triangle def}, we get 
\begin{equation}\label{eq: holonomy factor vertical triangle}
\Upsilon(\varrho_{p_{2,\infty}}, \varrho_{p_{1,2}}; [u_m])
=\varrho_{p_{2,\infty}}\varrho_{p_{1,2}}\exp 2\pi\bi\left[-\Big(a_{2}-a_{1}+({\bk_2-\bk_1})a(x)\Big)^T \left(b(x)-m-\frac{\lambda_{p_{1,2}}+b_2-b_1}{\bk_2-\bk_1}\right)\right].
\end{equation}

Having described the triangles $u_m$, $m\in \bZ^g$, that contributes to the output, and calculated the areas $\int u_m^*\omega_\tau$ in Equation \eqref{eq: vertical triangle area} and the holonomy factor in Equation \eqref{eq: holonomy factor vertical triangle}, we can now compute the product defined in Equation \eqref{eq: vertical triangle product},
\begin{equation}\label{eq: mu2_with_infty}
\begin{split}
& \mu^2(\varrho_{p_{2, \infty}}p_{p_{2, \infty}},  \varrho_{p_{1,2}}p_{p_{1,2}})
\\
= &\ \varrho_{p_{2,\infty}}\varrho_{p_{1,2}}\sum_{m\in\bZ^g}e^{\pi i (\bk_2-\bk_1)\left(b(x)-m-\frac{\lambda_{p_{1,2}}+b_2-b_1}{\bk_2-\bk_1} \right)^T\tau\left(b(x)-m-\frac{\lambda_{p_{1,2}}+b_2-b_1}{\bk_2-\bk_1}\right)} \\
& \hspace{3in} e^{-2\pi\bi\big(a_{2}-a_{1}+({\bk_2-\bk_1})a(x)\big)^T \left(b(x)-m-\frac{\lambda_{p_{1,2}}+b_2-b_1}{\bk_2-\bk_1}\right)}q\\
= \ &\varrho_{p_{2,\infty}}\varrho_{p_{1,2}}\sum_{m\in\bZ^g} e^{\pi i (\bk_2-\bk_1)b(x)^T\tau b(x)} e^{-2\pi\bi (\bk_2-\bk_1)b(x)^T\tau\left(m+\frac{\lambda_{p_{1,2}}+b_2-b_1}{\bk_2-\bk_1}\right)}\\
& \hspace{1.5in}  e^{\pi i (\bk_2-\bk_1) \left(m+\frac{\lambda_{p_{1,2}}+b_2-b_1}{\bk_2-\bk_1} \right)^T \tau\left(m+\frac{\lambda_{p_{1,2}}+b_2-b_1}{\bk_2-\bk_1} \right)} \\
& \hspace{2in} e^{-2\pi\bi \big((a_{2}-a_{1})+(\bk_2-\bk_1)a(x)\big)^T b(x)}  e^{2\pi\bi(a_{2}-a_{1}+(\bk_2-\bk_1)a(x))^T \left(m+\frac{\lambda_{p_{1,2}}+b_2-b_1}{\bk_2-\bk_1}\right)}q\\
\stackrel{\text{\textcircled{\raisebox{-0.9pt}{1}}
} }{=}\ &\varrho_{p_{2,\infty}}\varrho_{p_{1,2}} e^{\pi i (\bk_2-\bk_1)b(x)^T\tau b(x)} e^{-2\pi\bi \big((a_{2}-a_{1})+(\bk_2-\bk_1)a(x)
\big)^T b(x)}  \\
& \cdot \sum_{m\in\bZ^g} x^{(\bk_2-\bk_1)\left(m+\frac{\lambda_{p_{1,2}}+b_2-b_1}{\bk_2-\bk_1}\right)}e^{\pi i (\bk_2-\bk_1) \left(m+\frac{\lambda_{p_{1,2}}+b_2-b_1}{\bk_2-\bk_1} \right)^T \tau\left(m+\frac{\lambda_{p_{1,2}}+b_2-b_1}{\bk_2-\bk_1} \right)}e^{2\pi\bi\left(a_{2}-a_{1}\right)^T \left(m+\frac{\lambda_{p_{1,2}}+b_2-b_1}{\bk_2-\bk_1}\right)} q \\
\stackrel{\text{\textcircled{\raisebox{-0.9pt}{2}}}}{=}\ &\varrho_{p_{2,\infty}}\varrho_{p_{1,2}} e^{\pi i (\bk_2-\bk_1)b(x)^T\tau b(x)} e^{-2\pi\bi \big((a_{2}-a_{1})+(\bk_2-\bk_1)a(x)\big)^T b(x)}\\
&\hspace{2in}
\cdot \vartheta\left[\frac{\lambda_{p_{1,2}}+b_2-b_1}{\bk_2-\bk_1}, a_{2}-a_{1} \right]\left((\bk_2-\bk_1)\tau, x^{\bk_2-\bk_1}\right)q.
\end{split}
\end{equation}
For more explanation of the calculation above, the equality labeled by \textcircled{\raisebox{-0.9pt}{2}} is obtained by taking all terms not involving $m+\frac{\lambda_{p_{1,2}}+b_2-b_1}{\bk_2-\bk_1}$ out in front of the summation sign, and then remaining terms sum up to give a multi-theta function. Below we explain how to obtain the equality labeled by \textcircled{\raisebox{-0.9pt}{1}}, which involves rewriting the following term in the calculation as
\begin{eqnarray*}
& &  e^{-2\pi\bi (\bk_2-\bk_1)\left(b(x)^T\tau\left(m+\frac{\lambda_{p_{1,2}}+b_2-b_1}{\bk_2-\bk_1}\right)\right)}\\
& =& e^{-i (\bk_2-\bk_1)\left(\Omega^{-1}\log|x|\right)^T\tau\left(m+\frac{\lambda_{p_{1,2}}+b_2-b_1}{\bk_2-\bk_1}\right)} \\
&=& e^{-i (\bk_2-\bk_1)(\log|x|)^T\Omega^{-1}(B+i\Omega)\left(m+\frac{\lambda_{p_{1,2}}+b_2-b_1}{\bk_2-\bk_1}\right)} \\
&=& e^{-i (\bk_2-\bk_1)(\log|x|)^T\Omega^{-1}B\left(m+\frac{\lambda_{p_{1,2}}+b_2-b_1}{\bk_2-\bk_1}\right)}  e^{ (\bk_2-\bk_1)(\log|x|)^T\left(m+\frac{\lambda_{p_{1,2}}+b_2-b_1}{\bk_2-\bk_1}\right)} \\
&=& e^{-i (\bk_2-\bk_1)(\log|x|)^T\Omega^{-1}B\left(m+\frac{\lambda_{p_{1,2}}+b_2-b_1}{\bk_2-\bk_1}\right)}  |x|^{ (\bk_2-\bk_1)\left(m+\frac{\lambda_{p_{1,2}}+b_2-b_1}{\bk_2-\bk_1}\right)}\\
&=& e^{-2\pi\bi (\bk_2-\bk_1)B b(x)\left(m+\frac{\lambda_{p_{1,2}}+b_2-b_1}{\bk_2-\bk_1}\right)}  |x|^{ (\bk_2-\bk_1)\left(m+\frac{\lambda_{p_{1,2}}+b_2-b_1}{\bk_2-\bk_1}\right)}.
\end{eqnarray*}
Furthermore, since $a(x)+Bb(x)=-\frac{1}{2\pi}\arg(x)$, we have 
\[
|x|^{ (\bk_2-\bk_1)\left(m+\frac{\lambda_{p_{1,2}}+b_2-b_1}{\bk_2-\bk_1}\right)} e^{-2\pi\bi(\bk_2-\bk_1)(a(x)+Bb(x))^T\left(m+\frac{\lambda_{p_{1,2}}+b_2-b_1}{\bk_2-\bk_1}\right)}
= x^{(\bk_2-\bk_1)\left(m+\frac{\lambda_{p_{1,2}}+b_2-b_1}{\bk_2-\bk_1}\right)}.
\]

Note that the result of the product computation in Equation \eqref{eq: mu2_with_infty} fits the mirror description pointed out in  Remark \ref{rem:SYZ_signs}.  The factor in front of the theta function in Equation \eqref{eq: mu2_with_infty} is due to the choice of the section of the line bundle in $H^0(V_\tau, \mathcal L_{\bk_2-\bk_1, [a_2-a_1+ \tau (b_2-b_1)]})$.

\bibliographystyle{amsalpha}
\bibliography{glob_gen2_hms}

\end{document}